\documentclass[11pt]{amsart}

\usepackage{amsmath}
\usepackage{amssymb}
\usepackage{pifont}
\usepackage{graphicx}
\usepackage{ascmac}

\usepackage{amscd}
\usepackage{latexsym}
\usepackage{amsfonts}

\usepackage{constants}
\newcommand{\parenthezises}[1]{\arabic{#1}}
\newconstantfamily{c}{symbol=c,format=\parenthezises}

\newconstantfamily{e}{symbol=\epsilon,format=\parenthezises}

\newconstantfamily{m}{symbol=M,format=\parenthezises}

\usepackage{framed}

\newtheorem{thm}{Theorem}[section]
\newtheorem{lem}[thm]{Lemma}
\newtheorem{cor}[thm]{Corollary}
\newtheorem{rk}[thm]{Remark}

\newtheorem{pro}[thm]{Proposition}

\newcommand\spt{\mbox{spt}}
\newcommand\tang{\mbox{Tan}}
\newcommand\dist{\mbox{dist}}

\numberwithin{equation}{section}

\begin{document}

\title[Convergence of the Allen-Cahn equation]
{Convergence of the Allen-Cahn equation \\ with a zero Neumann boundary condition \\ on non-convex domains}

\author{Takashi Kagaya}
\address{Institute of Mathematics for Industry, Kyushu University, Fukuoka, 819-0395, Japan}
\email{kagaya@imi.kyushu-u.ac.jp}


\thanks{The author is partially supported by JSPS Research Fellow Grant number 16J00547.}


\begin{abstract}
We study a singular limit problem of the Allen-Cahn equation with a homogeneous Neumann boundary condition on non-convex domains with smooth boundaries under suitable assumptions for initial data. 
The main result is the convergence of the time parametrized family of the diffused surface energy to Brakke's mean curvature flow with a generalized right angle condition on the boundary of the domain. 
\end{abstract}

\maketitle 

\section{Introduction}\label{sec:int}

The Allen-Cahn equation was introduced to model the motion of phase boundaries by surface tension \cite{AC}. 
In this paper, we consider the Allen-Cahn equation with a homogeneous Neumann boundary condition 
\begin{align}
\partial_t u_\varepsilon =&\; \Delta u_\varepsilon - \dfrac{W^\prime(u_\varepsilon)}{\varepsilon^2} \quad \mbox{in} \quad  \Omega \times (0,\infty), \label{ac}\\
\langle \nabla u_\varepsilon, \nu \rangle =&\; 0 \quad \mbox{on} \quad \partial \Omega \times (0,\infty), \label{neumann}\\
u_\varepsilon(x,0) =&\; u_{\varepsilon, 0}(x) \quad \mbox{for} \quad x \in \Omega, \label{initial}
\end{align}
where $\Omega \subset \mathbb{R}^n$ is a bounded domain with smooth boundary, $\varepsilon$ is a small positive parameter, $\nu$ is the outer unit normal to $\partial \Omega$, $W$ is a bi-stable potential with two wells of equal depth at $\pm 1$ and $u_\varepsilon$ is a real-valued function indicating the phase state at each point. 
This equation is the $L^2$ gradient flow of 
\begin{equation}\label{def-energy} 
E_\varepsilon[u] := \int_\Omega \dfrac{\varepsilon |\nabla u|^2}{2} + \dfrac{W(u)}{\varepsilon} \; dx 
\end{equation}
sped up by the factor $1/\varepsilon$. 
Heuristically, for a given family of functions $\{u_\varepsilon\}_{0 < \varepsilon < 1}$ with $\sup_{\varepsilon} E_\varepsilon [u_\varepsilon] < \infty$, $u_\varepsilon$ is close to a characteristic function, with a transition layer of width approximately $\varepsilon$ and slope approximately $C/\varepsilon$. 
Thus $\Omega$ is mostly divided into two regions $\{u_\varepsilon \approx 1\}$ and $\{u_\varepsilon \approx -1\}$ for sufficiently small $\varepsilon$. 
With this heuristic picture, one may expect that the following diffused interface energy 
\begin{equation}\label{def-mut} 
\mu_{\varepsilon}^t := \left(\dfrac{\varepsilon |\nabla u_\varepsilon(\cdot, t)|^2}{2} + \dfrac{W(u_{\varepsilon}(\cdot,t))}{\varepsilon} \right) \mathcal{L}^n \lfloor_\Omega 
\end{equation}
behaves more or less like surface measures of moving phase boundaries. 
Furthermore, one may also expect that the motion of the ``transition layer'' is a mean curvature flow with the right angle condition on $\partial \Omega$ because a formal $L^2$ gradient flow of the surface area is its mean curvature flow. 
A rigorous proof was given by Mizuno and Tonegawa \cite{MT} in the most general setting, which requires extensive use of tools from the geometric measure theory. 
Those authors proved that the family of limit measures of $\mu^t_\varepsilon$ is Brakke's mean curvature flow with a generalized right angle condition on $\partial \Omega$ (see \cite{B} for the details of Brakke's mean curvature flow). 
However, they assumed that the domain is convex. 
Accordingly, we consider the singular limit of \eqref{ac}--\eqref{initial} without the assumption of convexity. 

The singular limit problem of the Allen-Cahn equation without a boundary has been studied by many researchers with different settings and assumptions. 
Here, we focus on some results related to the Brakke flows. 
Ilmanen \cite{I} proved that the family of the diffused surface energy converges to a Brakke flow, and this strategy was extended by \cite{LST, TT} for the singular limit problem of an Allen-Cahn type equation with a transport term. 
One of the keys to analyzing this singular limit problem is to examine the vanishing of the discrepancy measure 
\[
d\xi_{\varepsilon} := \left(\dfrac{\varepsilon|\nabla u_{\varepsilon}(x,t)|^2}{2} - \dfrac{W(u_{\varepsilon}(x,t))}{\varepsilon} \right) d\mathcal{L}^n\lfloor_\Omega(x)dt. 
\]

Mizuno and Tonegawa \cite{MT} use the convexity of the domain essentially in this step, in particular, to prove the uniform boundedness of the discrepancy $\varepsilon |\nabla u_{\varepsilon}|^2/2 - W(u_\varepsilon)/\varepsilon$ from above. 
In the present paper, we give a modified estimate of the upper bound of the discrepancy in the case that the domain is not necessarily convex, and show the vanishing of the discrepancy measure along the line of \cite{LST, TT} to prove that the limit of diffused surface measures is Brakke's mean curvature flow with a generalized right angle condition. 

For the singular limit problem of \eqref{ac}--\eqref{initial} from a different perspective, we refer to \cite{BF, BS, KKR}. 
Those authors basically proved the connection of the singular limit of \eqref{ac}--\eqref{initial} to the unique viscosity solutions of a level set formulation of the mean curvature flow with right angle boundary conditions studied in \cite{GS, Sa}. 
In \cite{KKR}, in order to analyze the asymptotic behavior of the solution of \eqref{ac}--\eqref{initial} as $\varepsilon \to 0$, they apply the comparison principle. 
However, the convexity of the domain is essential for constructing super- and sub-solutions even in their proofs. 
On the other hand, Barles and Da Lio \cite{BF} and Barles and Suganidis \cite{BS} analyzed the connection without the convexity assumption on the domain by introducing a new definition of the generalized propagation of fronts in $\mathbb{R}^n$. 
We also note that in \cite{BF, BS, KKR} we do not know whether or not the particular individual level set obtained as a singular limit of \eqref{ac}--\eqref{initial} satisfies the mean curvature flow equation or the boundary conditions in the sense of measure. 

We refer to more results related to ours. 
Tonegawa \cite{T2} extended Ilmanen's work \cite{I} in bounded domains and proved that the limit measures have integer density a.e.\ modulo division by a constant. 
This result can be applied to our problem, and thus the limit measures of \eqref{def-mut} satisfy the integrality in the interior of the domain, whereas we do not know the integrality of the limit measures on the boundary of the domain. 
If the densities are equal to $1$ a.e.\ in the domain, the interior regularity follows from \cite{B, KaT, T3}. 
For the Brakke flow with a generalized right angle condition, Edelen \cite{E} proved the existence theory (by using a different construction from ours), the compactness theory, the regularity theory associated to tangent flows and so on. 
In order to consider the contact angle of the ``transition layer'' on the boundary of the domain, we mention contact angle conditions in the sense of measure. 
A right angle condition for rectifiable varifolds was studied by Gr\"uter and Jost \cite{GJ}, and general angle conditions for general varifolds were considered by the author and Tonegawa \cite{KT}. 
The contact angle of varifolds can be discussed by using outer unit normal vector of the boundary of the domain and the generalized co-normal vector of the varifolds as in \cite{KT}, and we will consider this kind of discussion after the main results in this paper. 
For a better understanding of the ``phase separation'', we refer to \cite{KT2, M, T} in singular limit problems for critical points of \eqref{def-energy} under the constraint of the total mass of $u$. 

The paper is organized as follows. 
Section \ref{sec:not} lists basic notations and recalls some notions related to varifolds. 
Section \ref{sec:as} lists assumptions and the main theorems of the present paper. 
In section \ref{sec:mono}, we fix some notations related to the reflection argument and recall the boundary monotonicity formula proved in \cite{MT}. 
Section \ref{sec:bdd-dis} shows that the growth rate of the discrepancy with respect to $\varepsilon$ is bounded by a negative power of $\varepsilon$. 
We estimate the density ratio of the diffused surface measure in section \ref{sec:den-rat} and prove the vanishing of the discrepancy energy in section \ref{sec:vani-dis}. 
Finally, we prove the main theorems in section \ref{sec:main}. 


\section{Notations and basic definitions}\label{sec:not}

\subsection{Basic notations}

In this paper, $n$ refers to positive integers. 
For $0<r<\infty$ and $a \in \mathbb{R}^n$ let 
\[B_r(a) := \{ x \in \mathbb{R}^n : |x-a| < r\}. \]  
We denote by $\mathcal{L}^k$ the Lebesgue measure on $\mathbb{R}^k$ and by $\mathcal{H}^k$ the 
$k$-dimensional Hausdorff measure on $\mathbb{R}^n$ for positive integers $k$. 
The restriction of $\mathcal{H}^k$ to a set $A$ is denoted by $\mathcal{H}^k \lfloor_A$.  
We let 
\[ \omega_k := \mathcal{L}^k(\{x \in \mathbb{R}^k : |x| < 1\}). \]
For $x,y \in \mathbb{R}^n$ and $s>t$, we define the backward heat kernels 
\begin{equation}\label{bhk} 
\rho_{(y,s)}(x,t) := \dfrac{1}{(4\pi(s-t))^{\frac{n-1}{2}}} e^{-\frac{|x-y|^2}{4(s-t)}}. 
\end{equation}
For any Radon measure $\mu$ on $\mathbb{R}^n$, $\phi \in C_c(\mathbb{R}^n)$ and $\mu$ measurable set $A$, we often write 
\[ \mu(\phi) := \int_{\mathbb{R}^n} \phi \; d\mu, \quad \mu(A) := \int_{A} d\mu. \] 
Let the support of $\mu$ be 
\[ \spt \mu := \{x \in \mathbb{R}^n : \mu(B_r(x)) > 0 \; \mbox{for} \; r>0 \}. \]

\subsection{Homogeneous maps and varifolds}

Let $\mathbf{G}(n,n-1)$ be the space of $(n-1)$-dimensional subspace of $\mathbb{R}^n$. 
For $S \in \mathbf{G}(n,n-1)$, we identify $S$ with the corresponding orthogonal projection of $\mathbb{R}^n$ onto $S$. 
For two elements $A$ and $B$ of $\mbox{Hom}(\mathbb{R}^n, \mathbb{R}^n)$, we define a scalar product as 
\[ A \cdot B := \sum_{i,j} A_{ij}B_{ij}. \]
The identity of $\mbox{Hom}(\mathbb{R}^n, \mathbb{R}^n)$ is denoted by $I$. 

We recall some notions related to varifold and refer to \cite{A1,S} for more details. 
Let $X \subset \mathbb{R}^n$ be open in the following and $G_{n-1}(X) := X \times \mathbf{G}(n, n-1)$. 
A general $(n-1)$-varifold in $X$ is a Radon measure on $G_{n-1}(X)$ and $\mathbf{V}_{n-1}(X)$ denotes the set of all general $(n-1)$-varifold in $X$. 
For $V \in \mathbf{V}_{n-1}(X)$, let $\|V\|$ be the weight measure of V, namely, 
\[ \|V\|(\phi) := \int_{G_{n-1}(X)} \phi(x) \; dV(x,S) \quad \mbox{for} \; \phi \in C_c(X). \]
We say that $V \in \mathbf{V}_{n-1}(X)$ is rectifiable if there exists an $\mathcal{H}^{n-1}$ measurable countably $(n-1)$-rectifiable set $M \subset X$ 
and a locally $\mathcal{H}^{n-1}$ integrable function $\theta$ defined on $M$ such that 
\begin{equation}\label{recti} 
V(\phi) = \int_M \phi(x,\tang_xM) \theta(x)\; d\mathcal{H}^{n-1}(x) \quad \mbox{for} \; \phi \in C_c(G_{n-1}(X)), 
\end{equation}
where $\mbox{Tan}_xM \in \mathbf{G}(n, n-1)$ is the approximate tangent space that exists $\mathcal{H}^{n-1}$-a.e.\ on $M$. 
Additionally, if $\theta \in \mathbb{N}$ $\mathcal {H}^{n-1}$-a.e.\ on $M$, we say that $V$ is integral. 
A rectifiable $(n-1)$-varifold is uniquely determined by its weight measure through the formula \eqref{recti}. 
For this reason, we naturally say a Radon measure $\mu$ on $X$ is rectifiable (or integral) if there exists a rectifiable (or integral) varifold such that the weight measure is equal to $\mu$. 
The set of all rectifiable and integral $(n-1)$-varifolds in $X$ is denoted by $\mathbf{RV}_{n-1}(X)$ and $\mathbf{IV}_{n-1}(X)$, respectively. 

For $V \in \mathbf{V}_{n-1}(X)$, let $\delta V$ be the first variation of $V$, namely, 
\[ \delta V(g) := \int_{G_n(X)} \nabla g(x) \cdot S \; dV(x,S) \quad \mbox{for} \; g \in C^1_c(X;\mathbb{R}^n). \]
Let $\|\delta V\|$ be the total variation when it exists, and if $\|\delta V\|$ is locally bounded, 
we may apply the Riesz representation theorem and the Lebesgue decomposition theorem (see \cite[Theorem 1.38, Theorem 1.31]{EG}) to $\delta V$ with respect to $\|V\|$. 
Then, we obtain a $\|V\|$ measurable function $h: X \to \mathbb{R}^n$, a Borel set $Z \subset X$ such that $\|V\|(Z)=0$ and a $\|\delta V\|\lfloor_{Z}$ measurable function $\nu_{\rm sing}: Z \to \mathbb{R}^n$ with $|\nu_{\rm sing}| = 1$ $\|\delta V\|$-a.e.\ on $Z$ such that 
\begin{equation}\label{def-mean} 
\delta V(g) = - \int_X \langle h, g\rangle \; d\|V\| + \int_Z \langle \nu_{\rm sing}, g \rangle \; d\|\delta V\| \quad \mbox{for} \; g \in C^1_c(X ; \mathbb{R}^n). 
\end{equation}
The vector field $h$ is called the generalized mean curvature vector of $V$, the vector field $\nu_{\rm sing}$ is called the (outer-pointing) generalized co-normal of $V$ and the Borel set $Z$ is called the generalized boundary of $V$ (see also \cite{Ma} for more details about varifolds with boundary).  


\section{Assumptions and main result}\label{sec:as}

\subsection{Assumptions and a previous result}

In the following, we assume that $\Omega \subset \mathbb{R}^n$ is a bounded domain with smooth boundary $\partial \Omega$. 
Suppose $W \in C^3(\mathbb{R})$ satisfies the following: 
\begin{itemize}
\item[(W1)] $W(\pm 1) = 0$ and $W(s) > 0$ for all $s \neq \pm 1$, 
\item[(W2)] for some $-1 < \gamma < 1$, $W^\prime < 0$ on $(\gamma, 1)$ and $W^\prime > 0$ on $(-1, \gamma)$, 
\item[(W3)] for some $0 < \alpha < 1$ and $\beta > 0$, $W^{\prime \prime}(s) \ge \beta$ for all $\alpha \le |s| \le 1$. 
\end{itemize}
A typical example of such $W$ is $(1-s^2)^2/4$, for which we may set $\alpha = \sqrt{2/3}$, $\beta = 1$ and $\gamma = 0$. 

For a given sequence of positive numbers $\{\varepsilon_i\}_{i=1}^\infty$ with $\lim_{i \to \infty} \varepsilon_i = 0$, suppose $u_{\varepsilon_i, 0} \in C^1(\overline{\Omega})$ satisfies 
\begin{align} 
&\| u_{\varepsilon_i, 0}\|_{L^\infty(\Omega)} \le 1, \label{as-u1} \\
&\sup_i \sup_{x \in \Omega, \; 0<r} \omega_{n-1}^{-1}r^{1-n}\int_{B_r(x) \cap \Omega} \dfrac{\varepsilon_i |\nabla u_{\varepsilon_i,0}(y)|^2}{2} + \dfrac{W(u_{\varepsilon_i,0}(y))}{\varepsilon_i} \; dx \le D_0, \label{as-u3} \\
& \sup_i \max_{x \in \overline{\Omega}} \varepsilon_i |\nabla u_{\varepsilon_i,0}| \le \Cl[c]{c-as-gra}, \label{as-u4} \\
& \max_{x \in \overline{\Omega}} \dfrac{\varepsilon_i |\nabla u_{\varepsilon_i,0}(y)|^2}{2} - \dfrac{W(u_{\varepsilon_i,0}(y))}{\varepsilon_i} \le \Cl[c]{c-as-dis} \varepsilon_i^{-\lambda} \quad \mbox{for} \quad i \in \mathbb{N}, \label{as-u5} \\
& \langle \nabla u_{\varepsilon_i, 0}(x), \nu \rangle = 0 \quad \mbox{for} \quad x \in \partial \Omega, i \in \mathbb{N}, \label{as-u6}
\end{align}
where $D_0, \Cr{c-as-gra}, \Cr{c-as-dis}$ and $\lambda \in [3/5,1)$ are some universal constants. 
We note that the boundedness of the domain $\Omega$ and the assumption \eqref{as-u3} imply 
\begin{equation}\label{as-u2}
\sup_i E_{\varepsilon_i}[u_{\varepsilon_i, 0}] \le \Cl[c]{c-ene0} 
\end{equation}
for some constant $\Cr{c-ene0}$ depending only on $n, D_0$ and the diameter of $\Omega$. 
The conditions \eqref{as-u1} and \eqref{as-u2} are assumed in \cite{MT}. 
\eqref{as-u1} may be dropped if we assume a suitable growth rate upper bound on $W$ as Mizuno and Tonegawa commented in \cite{MT}. 
We need the additional assumptions \eqref{as-u3}--\eqref{as-u6} to apply the argument for the vanishing of the discrepancy measure in \cite{KT, LST}. 

\begin{rk}
We note that for a surface $\Gamma$ with $90$ degree contact angles on $\partial \Omega$ it is possible to construct diffuse approximations that satisfy the assumptions \eqref{as-u1}--\eqref{as-u6} as the following. 
Our construction is standard as in \cite{I, M}. 
Let $\Omega_d$ be 
\[ \Omega_d := \{ (y_1, y^\prime) \in \mathbb{R}^n : y_1 \in \mathbb{R}, |y^\prime| < d \} \] 
for $d>0$ and define $\tilde{\Gamma} := \overline{\Omega}_d \cap \{y_1=0\}$. 
By the standard existence theory for ordinary differential equations, we may choose the unique function $q \in C^4(\mathbb{R})$ such that 
\[ q(0) = 0, \quad \lim_{s \to \pm \infty} q(s) = \pm1, \quad q^\prime(s) = \sqrt{2W(q(s))} \quad \mbox {in} \; \mathbb{R}. \]
Then it is easy to see that the $C^4$ function $v_{\varepsilon_i}(y) := q(y_1/\varepsilon_i)$ defined on $\overline{\Omega}_d$ satisfies 
\begin{equation}\label{def-v-e}
\begin{aligned} 
&\int_{B_r(y_0) \cap \Omega_d} \dfrac{\varepsilon_i |\nabla v_{\varepsilon_i}|^2}{2} + \dfrac{W(v_{\varepsilon_i})}{\varepsilon_i} \; dy \le \sigma \omega_{n-1} r^{n-1} \quad \mbox{for} \; r>0, \; y_0 \in \mathbb{R}^n, \\
&\varepsilon_i |\nabla v_{\varepsilon_i}(y)| \le \max_{|s| \le 1} \sqrt{2W(s)}, \; \dfrac{\varepsilon_i |\nabla v_{\varepsilon_i}(y)|^2}{2} = \dfrac{W(v_{\varepsilon_i}(y))}{\varepsilon_i} \quad \mbox{for} \; y \in \overline{\Omega}_d, \\
&\langle \nabla v_{\varepsilon_i}, \nu_d \rangle = 0 \quad \mbox{on} \; \partial \Omega_d, 
\end{aligned} 
\end{equation}
where $\sigma := \int_{-1}^1 \sqrt{2W(s)} \; dx$ and $\nu_d$ is the out ward unit normal to $\partial \Omega_d$. 
Now we assume that $\tilde{U}$ is a neighborhood of $\tilde{\Gamma}$ and that $\phi$ is a bijective $C^1$ map from $\tilde{U}$ onto $U := \phi(\tilde{U})$ such that 
\[ \phi(\Omega_d \cap \tilde{U}) = \Omega \cap U, \quad \phi(\partial \Omega_d \cap \tilde{U}) = \partial \Omega \cap U, \quad \sup_{x \in U}\|\nabla \phi^{-1}(x)\|\le 1, \quad \sup_{y \in \tilde{U}}\| \nabla \phi(y)\| \le C \]
for a suitable $d>0$ and a constant $C>0$, where $\| \cdot \|$ is the operator norm. 
By using this mapping, \eqref{def-v-e} implies that $u_{\varepsilon_i,0}(x) := v_{\varepsilon_i}\circ \phi^{-1}(x)$ satisfies the assumptions \eqref{as-u1}--\eqref{as-u6} with a positive constant $D_0$ depending only on $\sigma, n$ and $C$, $\Cr{c-as-gra} = 1$ and $\Cr{c-as-dis} = 0$ on the set $\overline{\Omega} \cap U$. 
By expanding $u_{\varepsilon_i,0}$ as a mostly constant function to satisfy the assumptions outside of $U$, we may see the possibility of the initial assumptions in the present paper. 
In this construction, the diffused interface energy for $u_{\varepsilon_i,0}$ should behave like the surface measure of the surface $\Gamma := \phi(\tilde{\Gamma})$ and $\Gamma$ intersects $\partial \Omega$ with 90 degrees. 
\end{rk}

By the standard parabolic existence and regularity theory, for each $i$, there exists a unique solution $u_{\varepsilon_i}$ with 
\begin{equation}\label{ref-sol} 
u_{\varepsilon_i} \in C([0,\infty);C^1(\overline{\Omega})) \cap C^\infty(\overline{\Omega} \times (0, \infty)). 
\end{equation}
By the maximum principle and \eqref{as-u1}, 
\begin{equation}\label{u-bdd} 
\sup_{x \in \overline{\Omega}, t>0} |u_{\varepsilon_i}| \le 1, 
\end{equation}
and due to the gradient structure and \eqref{as-u2}, 
\begin{equation}\label{energy-bdd} 
E_{\varepsilon_i}[u_{\varepsilon_i}(\cdot, T)] + \int_0^T \int_\Omega \varepsilon_i (\partial_t u_{\varepsilon_i})^2 \; dx dt = E_{\varepsilon_i}[u_{\varepsilon_i, 0}] \le \Cr{c-ene0} 
\end{equation}
for any $T>0$. 

The convergence of the diffused interface energy measures is proved in \cite{MT}. 
The proof is based on the gradient structure and dose not require the convexity of $\Omega$. 

\begin{pro}[{\cite[Proposition 5.2]{MT}}] \label{pro-con-den}
Under the assumptions {\rm (W1)-(W3)}, \eqref{as-u1} and \eqref{as-u2}, let $u_{\varepsilon_i}$ be the solution of \eqref{ac}. 
Define $\mu_{\varepsilon_i}^t$ as in \eqref{def-mut}. 
Then there exists a family of Radon measures $\{\mu^t\}_{t \ge 0}$ on $\mathbb{R}^n$ and a subsequence (denoted by the same index) such that $\mu_{\varepsilon_i}^t$ converges to $\mu^t$ as $i \to \infty$ for all $t \ge 0$ on $\mathbb{R}^n$. 
\end{pro}

By the definition \eqref{def-mut} and Proposition \ref{pro-con-den}, we see ${\rm spt}\mu^t \subset \overline{\Omega}$ for all time $t \ge 0$. 

\subsection{Main results}

In this paper, our goal is to extend the convergence theory in \cite{MT} to remove the convexity assumption on $\Omega$ as the following. 

\begin{thm}\label{main1}
Under the assumptions {\rm (W1)-(W3)} and \eqref{as-u1}-\eqref{as-u6}, let $u_{\varepsilon_i}$ be the solution to \eqref{ac}. 
Define $\mu_{\varepsilon_i}^t$ as in \eqref{def-mut}. 
Let $\varepsilon_i$ be the subsequence such that Proposition \ref{pro-con-den} holds and $\mu^t$ be the limit of $\mu^t_{\varepsilon_i}$ for all $t \ge 0$. 
Then, $\mu^t$ is rectifiable on $\mathbb{R}^n$ for a.e. $t \ge 0$. 
\end{thm}

Theorem \ref{main1} shows that the rectifiability of $\mu^t$ holds up to (and including) the boundary $\partial \Omega$ for a.e.\ $t \ge 0$.
By Theorem \ref{main1}, we may define rectifiable varifolds $V^t \in \mathbf{RV}_{n-1}(\mathbb{R}^n)$ as $\|V^t\| = \mu^t$ if $\mu^t$ is rectifiable. 
If $\mu^t$ is not rectifiable, we define $V^t \in \mathbf{V}_{n-1}(\mathbb{R}^n)$ to be an arbitrary varifold with $\|V^t\| = \mu^t$ 
(for example $V^t(\phi) := \int_{\mathbb{R}^n} \phi(\cdot, \mathbb{R}^{n-1} \times \{0\}) \; d\mu^t$ for $\phi \in C_c(G_{n-1}(\mathbb{R}^n)$)). 

\begin{rk}
As we mention in Section \ref{sec:int}, the integrality of the limit varifolds in the interior of $\Omega$ follows from \cite{T2}. 
That is, $\sigma^{-1} V^t\lfloor_\Omega \in \mathbf{IV}_{n-1}(\Omega)$ for a.e.\ $t \ge 0$, where $\sigma = \int^1_{-1} \sqrt{2W(s)}\; ds$. 
\end{rk}

\begin{thm}\label{main2}
Let $V^t$ be defined as above. Then $\|\delta V^t\| (\mathbb{R}^n) = \|\delta V^t\|(\overline{\Omega})$ is finite for a.e.\ $t \ge 0$ and $\int^T_0 \|\delta V^t\|(\overline{\Omega}) \; dt$ is finite for all $T>0$. 
\end{thm}

By Theorem \ref{main2}, we can apply the Riesz representation theorem and the Lebesgue decomposition theorem as in \eqref{def-mean} for a.e.\ $t \ge 0$, and thus the generalized mean curvature vector of $V^t$ is well defined for a.e.\ $t \ge 0$. 
However, to prove that the set of the limit varifolds is a Brakke flow with a generalized right angle condition on the boundary, we have to define the tangential component of the first variation on $\partial \Omega$.
For $t \ge 0$, define 
\begin{equation}\label{tang-first} 
\delta V^t\lfloor_{\partial \Omega}^\top (g) := \delta V^t\lfloor_{\partial \Omega}(g-\langle g, \nu \rangle \nu) \quad \mbox{for} \quad g \in C^1_c(\mathbb{R}^n; \mathbb{R}^n). 
\end{equation}

\begin{thm}\label{main3}
Let $V^t$ be defined as above. 
Also define $\delta V^t\lfloor_{\partial \Omega}^\top$ as in \eqref{tang-first}. 
Then the varifolds $V^t$ satisfy the following: 
\begin{itemize}
\item[(A1)] For a.e.\ $t \ge 0$, $\|\delta V^t\lfloor_{\partial \Omega}^\top + \delta V^t\lfloor_\Omega\| \ll \|V^t\|$ and there exists $h_b^t \in L^2(\|V^t\|)$ such that 
\[
\delta V^t\lfloor_{\partial \Omega}^\top + \delta V^t\lfloor_\Omega = - h_b^t \|V^t\|. 
\] 
\item[(A2)] $h_b^t$ satisfies 
\[ \int^\infty_0 \int_{\overline{\Omega}}|h_b^t|^2 \; d\|V^t\| dt \le \Cr{c-ene0}. \]
\item[(A3)] For $\phi \in C^1(\overline{\Omega} \times [0,\infty); \mathbb{R}_+)$ with $\langle \nabla \phi(\cdot, t), \nu \rangle = 0$ on $\partial \Omega$ and for any $0 \le t_1 < t_2 < \infty$, 
\begin{equation}\label{brakke-neumann} 
\|V^t\|(\phi(\cdot,t)) \Big|_{t=t_1}^{t_2} \le \int^{t_2}_{t_1} \int_{\overline{\Omega}} -\phi|h^t_b|^2 + \langle \nabla \phi, h^t_b \rangle + \partial_t \phi \; d\|V^t\|dt. 
\end{equation}
\end{itemize}
\end{thm}

\begin{rk}
From (A1) of Theorem \ref{main3} and the Radon-Nikodym theorem as in \eqref{def-mean}, for a.e.\ $t$ and any $g \in C_c^1(\mathbb{R}^n; \mathbb{R}^n)$ with $\langle g, \nu \rangle = 0$ on $\partial \Omega$, we obtain by the definition of $\delta V^t\lfloor_{\partial \Omega}^\top$ 
\[
- \int \langle h^t_b, g \rangle \; d\|V^t\| = \delta V^t\lfloor_{\partial \Omega}^\top(g) + \delta V^t\lfloor_\Omega(g) = \delta V^t(g) = -\int \langle h^t, g \rangle \; d\|V^t\| + \int_{Z^t} \langle \nu^t_{\rm sing}, g \rangle \; d\|\delta V^t\|, 
\]
where the generalized mean curvature vector $h^t$, the generalized co-normal $\nu_{\rm sing}^t$ and the generalized boundary $Z^t$ of $V^t$ is defined as in \eqref{def-mean}. 
In particular, the absolute continuity in (A1) of Theorem \ref{main3} implies $\|\delta V^t\lfloor_\Omega\| \ll \|V^t\lfloor_\Omega\|$, and hence we see the last term of the equality is zero for $g \in C^1_c(\Omega; \mathbb{R}^n)$. 
In other ward, we can see that (1) $h^t_b$ coincides with $h^t$ $\|V^t\|$-a.e.\ in $\Omega$; and (2) $Z^t$ is a subset of $\partial \Omega$ (or $Z^t \cap \Omega$ is empty). 
Furthermore, applying the properties (1) and (2) for the equality, we have 
\[ -\langle g, h^t_b \rangle \|V^t\|\lfloor_{\partial \Omega} = -\langle g, h^t \rangle \|V^t\|\lfloor_{\partial \Omega} + \langle g, \nu^t_{\rm sing} \rangle \|\delta V^t\|\lfloor_{Z^t} \]
for $g \in C(\partial \Omega; \mathbb{R}^n)$ with $\langle g, \nu \rangle = 0$ on $\partial \Omega$ and a.e.\ $t \ge 0$. 
Thus, from $\|V^t\|(Z^t)=0$, we can also see that (3) $\nu_{\rm sing}^t$ is perpendicular to $\partial \Omega$ $\|\delta V^t\|$-a.e.\ on $Z^t$; and (4) $h^t_b$ is the projection of $h^t$ to the tangent space of $\partial \Omega$ $\|V^t\|$-a.e.\ on $\partial \Omega$. 
Hence (A1) of Theorem \ref{main3} corresponds to the 90 degree angle condition of $V^t$ (see also \cite{KT}). 
For the classical sense, see Figure \ref{fig1} and we note that the divergence theorem on a smooth and oriented hypersurface $M \subset \mathbb{R}^n$ implies 
\[ \delta V_M (g) = \int_M \nabla g (x) \cdot {\rm Tan}_x M \; d\mathcal{H}^{n-1}(x) = -\int_M \langle h_M, g\rangle d\mathcal{H}^{n-1} + \int_{\partial M} \langle \nu_M, g \rangle \; d\mathcal{H}^{n-2} \]
for any $g \in C^1_c(\mathbb{R}^n; \mathbb{R}^n)$, where the varifold $V_M$ is defined by 
\[ V_M(\phi) := \int_M \phi(x, {\rm Tan}_x M) \; d\mathcal{H}^{n-1}(x) \quad {\rm for} \quad \phi \in G_{n-1}(\mathbb{R}^n), \]
${\rm Tan}_x M$ is the tangent space of $M$ at $x \in M$, $h_M$ is the mean curvature vector of $M$, $\partial M$ is the boundary of $M$ and $\nu_M$ is the co-normal vector of $M$. 
\end{rk}

\begin{figure}[t]
\begin{minipage}[t]{0.49\linewidth}
\centering
\includegraphics[width=4.5cm]{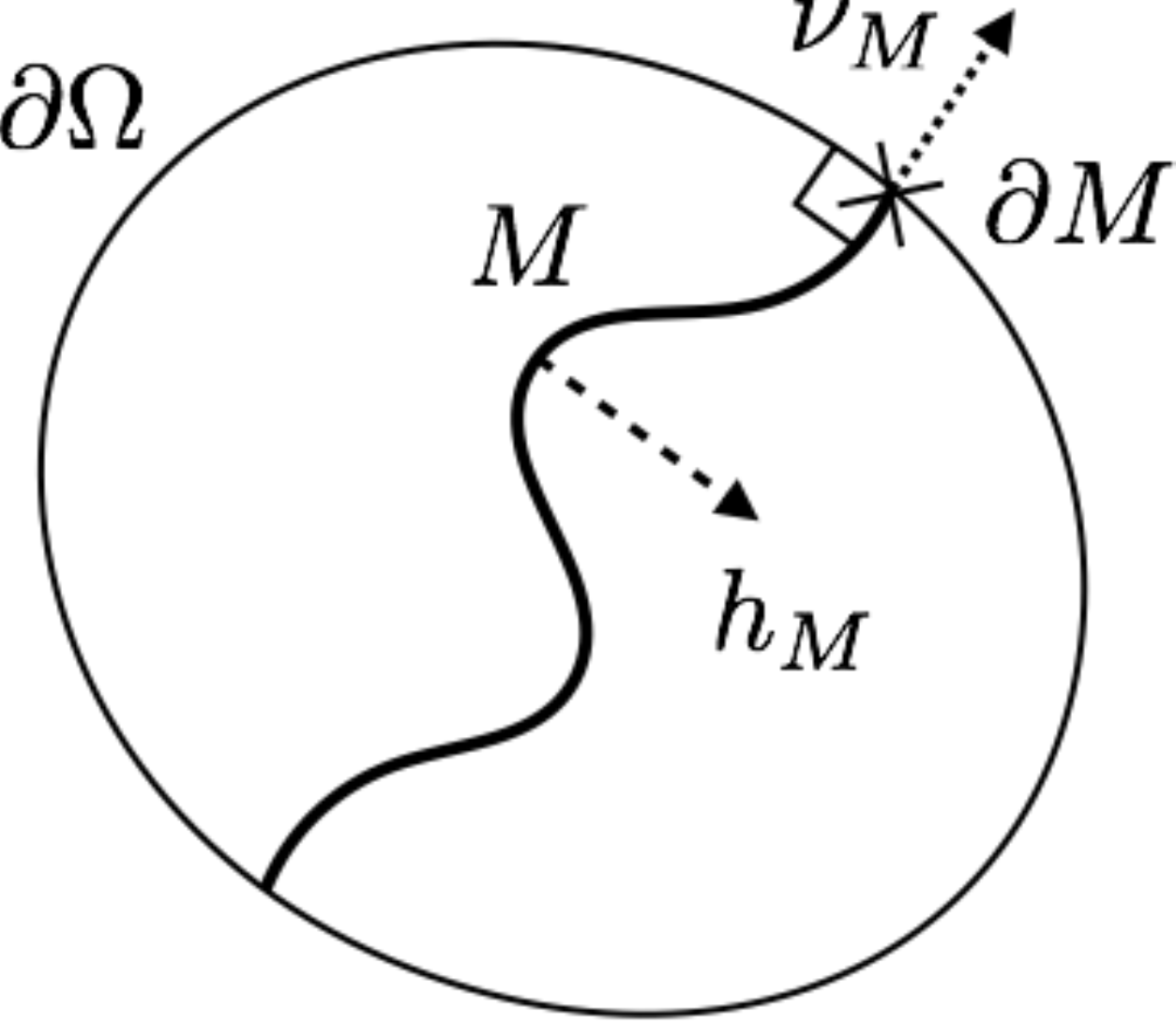}
\caption{Picture of $h_M$, $\partial M$ and $\nu_M$ when $V_M$ satisfies the conditions for $V^t$ of (A1) in Theorem \ref{main3}.}
\label{fig1}
\end{minipage}
\begin{minipage}[t]{0.49\linewidth}
\centering
\includegraphics[width=5cm]{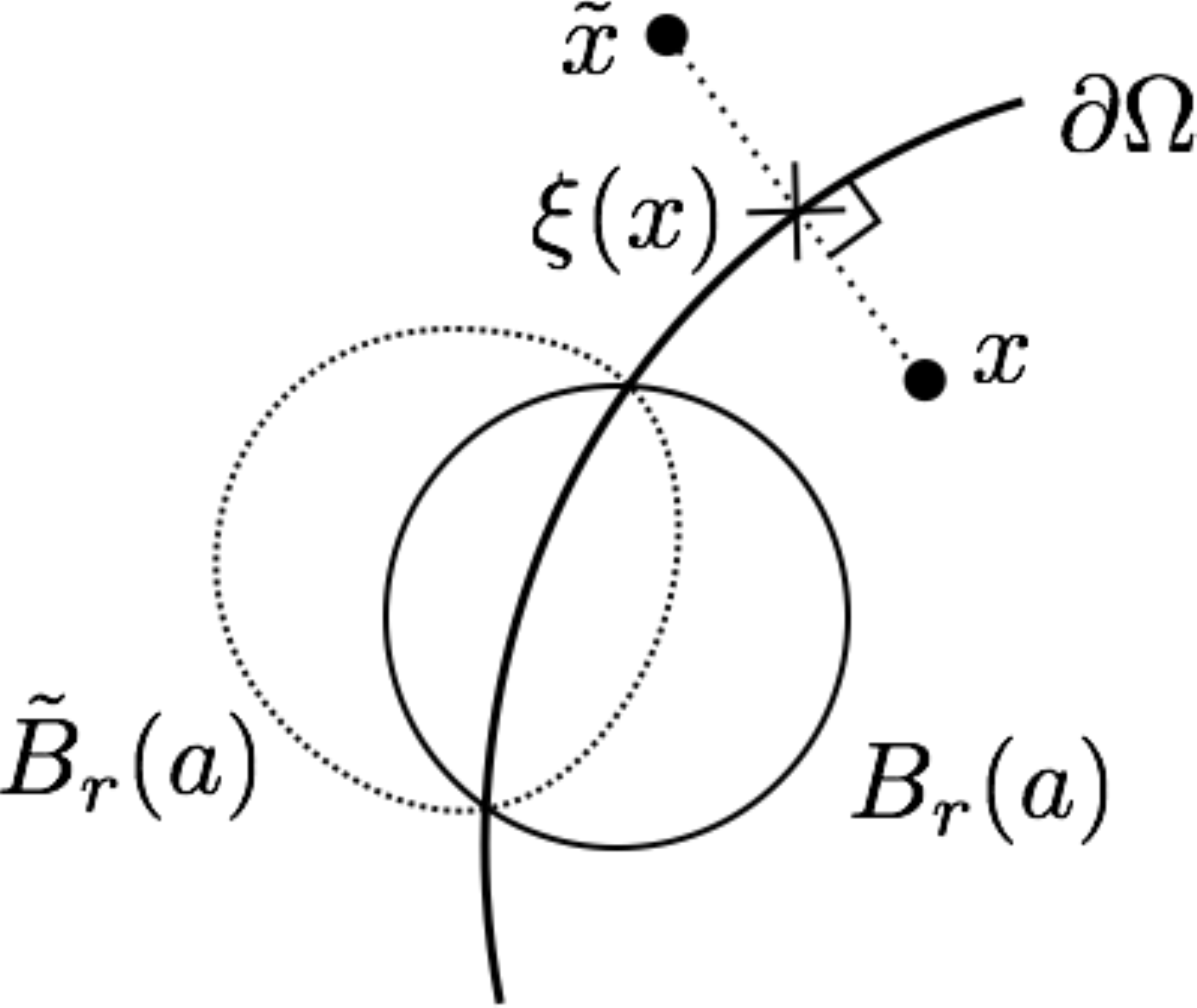}
\caption{Picture of $\xi(x)$, $\tilde{x}$ and $\tilde{B}_r(a)$.}
\label{fig2}
\end{minipage}
\end{figure}



\section{Monotonicity formula}\label{sec:mono}

One of the key tools for analyzing the singular limit problem of the Allen-Cahn equation is the Huisken or Ilmanen type monotonicity formula. 
The boundary monotonicity formula can be derived by using the reflection argument as in \cite{MT}. 
To present the statement, we need some more notations associated with the reflection argument. 
Define $\kappa$ as 
\begin{equation}\label{def-kappa} 
\kappa := \| \mbox{principal curvature of } \partial \Omega\|_{L^{\infty}(\partial \Omega)}. 
\end{equation}
For $s>0$, define a subset $N_s$ of $\mathbb{R}^n$ by 
\[ N_s := \{ x \in \mathbb{R}^n : \dist(x, \partial \Omega) < s \}. \]
There exists a sufficiently small
\begin{equation}\label{def-c2} 
\Cl[c]{c-curva} \in (0, (6\kappa)^{-1}] 
\end{equation}
depending only on $\partial \Omega$ such that all points $x \in N_{6\Cr{c-curva}}$ have a unique point $\xi(x) \in \partial \Omega$ such that $\dist(x, \partial \Omega) = |x - \xi(x)|$ (see also Figure \ref{fig2}). 
By using this $\xi(x)$, we define the reflection point $\tilde{x}$ of $x$ with respect to $\partial \Omega$ as 
\[ \tilde{x} := 2\xi(x) - x \] 
and the reflection ball $\tilde{B}_r(x)$ of $B_r(a)$ with respect to $\partial \Omega$ as 
\[ \tilde{B}_r(a) := \{ x \in \mathbb{R}^n : |\tilde{x} - a| < r \}. \]
We also fix a function $\eta \in C^\infty (\mathbb{R})$ such that 
\[ 0 \le \eta \le 1, \quad \dfrac{d \eta}{d r} \le 0, \quad {\rm spt} \eta \subset [0,\Cr{c-curva}/2), \quad \eta = 1 \; \mbox{on} \; [0,\Cr{c-curva}/4]. \]
For $s > t > 0$ and $x,y \in N_{\Cr{c-curva}}$, we define the truncated version of the $(n-1)$-dimensional backward heat kernel and the reflected backward heat kernel as 
\[ \rho_{1,(y,s)}(x,t) := \eta(|x-y|)\rho_{(y,s)}(x,t), \quad \rho_{2,(y,s)}(x,t) := \eta(|\tilde{x} -y|) \rho_{(y,s)}(\tilde{x},t), \]
where $\rho_{(y,s)}$ is defined as in \eqref{bhk}. 
For $x \in N_{2\Cr{c-curva}} \setminus N_{\Cr{c-curva}}$ and $y \in N_{\Cr{c-curva}/2}$, we have 
\[ |\tilde{x} - y| \ge |\tilde{x} - \xi(y)| - |\xi(y) - y| > \Cr{c-curva} - \dfrac{\Cr{c-curva}}{2} = \dfrac{\Cr{c-curva}}{2}. \]
Thus we may smoothly define $\rho_{2,(y,s)}=0$ for $x \in \mathbb{R}^n \setminus N_{\Cr{c-curva}}$ and $y \in N_{\Cr{c-curva}/2}$. 
We also define the (signed) discrepancy measure $\xi_{\varepsilon_i}^t$ as 
\[
\xi_{\varepsilon_i}^t := \left(\dfrac{\varepsilon_i|\nabla u_{\varepsilon_i}(\cdot,t)|^2}{2} - \dfrac{W(u_{\varepsilon_i}(\cdot,t))}{\varepsilon_i} \right) \mathcal{L}^n\lfloor_\Omega.
\] 

\begin{pro}[Boundary monotonicity formula \cite{MT}]\label{thm:monotonicity}
There exist constants $0 < \Cl[c]{c-mono-1}, \Cl[c]{c-mono-2}<\infty$ depending only on $n, \Cr{c-ene0}$ and $\partial \Omega$ such that 
\begin{equation}\label{monotonicity} 
\begin{aligned}
&\; \dfrac{d}{dt} \left(e^{\Cr{c-mono-1}(s-t)^\frac{1}{4}} \int_\Omega \rho_{1,(y,s)}(x,t) + \rho_{2,(y,s)}(x,t) \; d\mu_{\varepsilon^i}^t(x)\right) \\
&\; \le e^{\Cr{c-mono-1}(s-t)^\frac{1}{4}}\left(\Cr{c-mono-2} + \int_\Omega \dfrac{\rho_{1,(y,s)}(x,t) + \rho_{2,(y,s)}(x,t)}{2(s-t)} \; d\xi_{\varepsilon_i}^t(x)\right) 
\end{aligned}
\end{equation}
for all $s>t>0, y \in N_{\Cr{c-curva}/2}$ and $i \in \mathbb{N}$, 
\begin{equation}\label{monotonicity-2}
\dfrac{d}{dt} \left(e^{\Cr{c-mono-1}(s-t)^\frac{1}{4}} \int_\Omega \rho_{1,(y,s)}(x,t) \; d\mu_{\varepsilon^i}^t(x)\right) \le e^{\Cr{c-mono-1}(s-t)^\frac{1}{4}}\left(\Cr{c-mono-2} + \int_\Omega \dfrac{\rho_{1,(y,s)}(x,t)}{2(s-t)} \; d\xi_{\varepsilon_i}^t(x)\right) 
\end{equation}
for all $s>t>0, y \in \mathbb{R}^n \setminus N_{\Cr{c-curva}/2}$ and $i \in \mathbb{N}$. 
\end{pro}

The proof of Proposition \ref{thm:monotonicity} in \cite{MT} does not require the convexity of $\Omega$, thus we refer to \cite{MT} for the details. 
We also remark that the reflected monotonicity formula in \cite{MT} can be expand for outside points $y$ of $\Omega$ as \eqref{monotonicity} because the condition $y\in\overline{\Omega}$ is not used in the proof in \cite{MT}. 


\section{Upper bound for the discrepancy}\label{sec:bdd-dis}

In this section, we estimate the growth rate of the discrepancy as follows. 

\begin{pro}\label{thm-dis-bdd}
There exists a constant $\Cl[c]{c-up-bdd-xi}$ depending only on $n, \kappa, \Cr{c-as-gra}, \Cr{c-as-dis}, \Cr{c-curva}, W$ and $\Omega$ such that 
\begin{equation}\label{dis-bdd} 
\sup_{\Omega \times [0, \infty)} \dfrac{\varepsilon_i|\nabla u_{\varepsilon_i}|^2}{2} - \dfrac{W(u_{\varepsilon_i})}{\varepsilon_i} \le \Cr{c-up-bdd-xi}\varepsilon_i^{-\lambda} 
\end{equation}
for any $0 < \varepsilon_i < 1$, where $\lambda$ is the constant in the assumption \eqref{as-u6}.
\end{pro}

In order to prove Proposition \ref{thm-dis-bdd}, we have to control the normal derivative of $|\nabla u_{\varepsilon_i}|^2$ at the boundary of $\Omega$. 

\begin{lem}\label{normal-di-u}
Let $\Omega^\prime$ be an arbitrary domain with smooth boundary and $A_x$ be the second fundamental form of $\partial \Omega^\prime$ at $x \in \partial \Omega^\prime$. 
Suppose that $v \in C^2(\overline{\Omega^\prime})$ satisfies $\langle \nabla v, \nu^\prime\rangle=0$ on $\partial \Omega^\prime$, where $\nu^\prime$ is the unit normal to $\Omega^\prime$. 
Then 
\[ \dfrac{\partial}{\partial \nu^\prime}\dfrac{|\nabla v|^2}{2} = A_x(\nabla v, \nabla v) \]
at $x \in \partial \Omega^\prime$. 
\end{lem}

This control has been used in a number of papers (for example, see \cite{CH,MT,T}), thus we refer to these papers for the proof. 

From Lemma \ref{normal-di-u}, we have to estimate $|\nabla u_{\varepsilon_i}|$ on the time-space domain $\Omega \times (0,\infty)$ to control the normal derivative of $|\nabla u_{\varepsilon_i}|^2$ at the boundary of $\Omega$. 
In the following, we use a parabolic re-scaling. 
Let 
\begin{equation}\label{scale-domain}
\Omega_{\varepsilon_i} = \{y \in \mathbb{R}^n : \varepsilon_i y \in \Omega\} 
\end{equation}
and we define the function 
\begin{equation}\label{scale-u}
v_{\varepsilon_i}(y,\tau) := u_{\varepsilon_i}(\varepsilon_i y, \varepsilon_i^2 \tau) \quad \mbox{for} \quad y \in \overline{\Omega}_{\varepsilon_i}, \quad \tau \in [0,\infty). 
\end{equation}
We note that 
\begin{equation}\label{scale-kappa}
\kappa_{\varepsilon_i} := \|\mbox{principal curvature of} \; \partial \Omega_{\varepsilon_i}\|_{L^\infty(\partial \Omega_{\varepsilon_i})} = \varepsilon_i \kappa
\end{equation}
holds and $v_{\varepsilon_i}$ satisfies 
\begin{equation}\label{scale-eq}
\begin{cases}
\partial_\tau v_{\varepsilon_i} = \Delta v_{\varepsilon_i} - W'(v_{\varepsilon_i}) \quad {\rm in} \quad \Omega_{\varepsilon_i} \times (0,\infty), \\
\langle \nabla v_{\varepsilon_i}, \nu_{\varepsilon_i} \rangle = 0 \quad {\rm on} \quad \partial \Omega_{\varepsilon_i} \times (0,\infty), 
\end{cases}
\end{equation}
where $\nu_{\varepsilon_i}$ is the outward unit normal to $\partial \Omega_{\varepsilon_i}$. 

\begin{lem}\label{lem:bdd-nabra}
There exists a constant $\Cl[c]{c-bdd-nabla}$ depending only on $\Cr{c-as-gra}, \Cr{c-curva}$ and $W$ such that 
\begin{equation}\label{bdd-nabla-u} 
\sup_{\Omega \times [0, \infty)} \varepsilon_i |\nabla u_{\varepsilon_i}| \le \Cr{c-bdd-nabla} 
\end{equation}
for all $0 < \varepsilon_i < 1$. 
\end{lem}

\begin{proof}
Let $\Omega_{\varepsilon_i}$ and $v_{\varepsilon_i}$ be defined in \eqref{scale-domain} and \eqref{scale-u}, respectively. 
By the assumption \eqref{as-u4} and the property \eqref{u-bdd}, we can see that 
\begin{equation}\label{bdd-nabla1} 
\|v_{\varepsilon_i}\|_{L^\infty(\Omega_{\varepsilon_i} \times (0,\infty))} \le 1, \quad \sup_{y \in \Omega_{\varepsilon_i}} |\nabla v_{\varepsilon_i}(y,0)| \le \Cr{c-as-gra}. 
\end{equation}
By a standard gradient estimate (see \cite[Theorem 7.2 of Chapter 5]{LSU}), we obtain the boundedness of $\sup_{\Omega_{\varepsilon_i} \times (0,\infty)} |\nabla v_{\varepsilon_i}|$ depending only on $\|v_{\varepsilon_i}\|_{L^\infty(\Omega_{\varepsilon_i} \times (0,\infty))}, \sup_{y \in \Omega_{\varepsilon_i}} |\nabla v_{\varepsilon_i}(y,0)|, W$ and the second fundamental form of $\partial \Omega_{\varepsilon_i}$. 
Applying \eqref{scale-kappa} and \eqref{bdd-nabla1}, we have the uniform boundedness of $\sup_{\Omega_{\varepsilon_i} \times (0,\infty)} |\nabla v_{\varepsilon_i}|$ for $0 < \varepsilon_i < 1$. 
This implies the conclusion by the definition of $v_{\varepsilon_i}$. 
\end{proof}

If we assume the convexity of the domain $\Omega$, the normal derivative of the discrepancy $\varepsilon_i|\nabla u_{\varepsilon_i}|^2/2 - W(u_{\varepsilon_i})/\varepsilon_i$ is non-positive at any boundary point of $\Omega$ since all principal curvatures of $\partial \Omega$ are non-positive and Lemma \ref{normal-di-u} holds, hence the maximum principle for the discrepancy works well as in \cite{MT}. 
In the following proof for Proposition \ref{thm-dis-bdd}, we apply the distance function from the boundary $\partial \Omega$ to control the normal derivative of the discrepancy on the boundary $\partial \Omega$ and argue a modified maximum principle. 

\begin{proof}[Proof of Proposition \ref{thm-dis-bdd}]
For simplicity we omit the subscript $i$. 
Let $\Omega_\varepsilon$ and $v_\varepsilon$ be defined as \eqref{scale-domain} and \eqref{scale-u}, respectively. 
For $G\in C^\infty(\mathbb{R})$ and $\phi \in C^\infty(\overline{\Omega}_\varepsilon)$ to be chosen latter, define 
\begin{equation}\label{def-xi-ep} 
\tilde{\xi}_\varepsilon (y,\tau) := \dfrac{|\nabla v_\varepsilon(y,\tau)|^2}{2} - W(v_\varepsilon(y,\tau)) - G(v_\varepsilon(y,\tau)) + \varepsilon\phi(y) 
\end{equation}
for $y \in \overline{\Omega}_\varepsilon$ and $\tau \in [0,\infty)$. 
We compute $ \partial_\tau \tilde{\xi}_\varepsilon - \Delta \tilde{\xi}_\varepsilon$ and obtain 
\begin{equation*}
\begin{aligned}
\partial_\tau \tilde{\xi}_\varepsilon - \Delta \tilde{\xi}_\varepsilon = &\; \langle \nabla v_\varepsilon, \nabla \partial_\tau v_\varepsilon \rangle - (W^\prime + G^\prime) \partial_\tau v_\varepsilon - |\nabla^2 v_\varepsilon|^2 - \langle \nabla v_\varepsilon, \nabla \Delta v_\varepsilon \rangle \\
&\; + (W^\prime + G^\prime) \Delta v_\varepsilon + (W^{\prime\prime} + G^{\prime\prime}) |\nabla v_\varepsilon|^2 - \varepsilon \Delta \phi 
\end{aligned}
\end{equation*}
for $y \in \overline{\Omega}_\varepsilon$ and $\tau \in (0,\infty)$. 
Substituting the equation \eqref{scale-eq}, we have 
\begin{equation}\label{para-xi1}
\partial_\tau \tilde{\xi}_\varepsilon - \Delta \tilde{\xi}_\varepsilon = W^\prime(W^\prime + G^\prime) - |\nabla^2 v_\varepsilon|^2 + G^{\prime\prime}|\nabla v_\varepsilon|^2 - \varepsilon \Delta \phi. 
\end{equation}
Differentiating \eqref{def-xi-ep} with respect to $y_j$ and by using the Cauchy-Schwarz inequality, we have 
\begin{equation}\label{xi-cauchy}
\begin{aligned}
|\nabla v_\varepsilon|^2|\nabla^2 v_\varepsilon|^2 \ge&\; \sum_{j=1}^n\left(\sum_{i=1}^n \partial_{y_i}v_\varepsilon \partial_{y_i y_j}v_\varepsilon \right)^2 
= \sum_{j=1}^n (\partial_{y_j} \tilde{\xi}_\varepsilon + (W^\prime + G^\prime) \partial_{y_j} v_\varepsilon - \varepsilon \partial_{y_j}\phi)^2 \\
\ge &\; 2\langle (W^\prime + G^\prime) \nabla v_\varepsilon - \varepsilon \nabla \phi, \nabla \tilde{\xi}_\varepsilon\rangle + (W^\prime + G^\prime)^2 |\nabla v_\varepsilon|^2 \\
&\; - 2\varepsilon(W^\prime + G^\prime) \langle \nabla v_\varepsilon, \nabla \phi\rangle. 
\end{aligned}
\end{equation}
On $\{|\nabla v_\varepsilon| \neq 0\}$, divide \eqref{xi-cauchy} by $|\nabla v_\varepsilon|^2$ and substitute into \eqref{para-xi1} to obtain 
\begin{equation}\label{para-xi2}
\begin{aligned}
\partial_\tau \tilde{\xi}_\varepsilon - \Delta \tilde{\xi}_\varepsilon \le&\; -(G^\prime)^2 - W^\prime G^\prime - \dfrac{2\langle (W^\prime + G^\prime) \nabla v_\varepsilon - \varepsilon \nabla \phi, \nabla \tilde{\xi}_\varepsilon\rangle}{|\nabla v_\varepsilon|^2} \\
&\; + \dfrac{2\varepsilon(W^\prime + G^\prime)}{|\nabla v_\varepsilon|^2}\langle \nabla v_\varepsilon, \nabla \phi\rangle + G^{\prime\prime}|\nabla v_\varepsilon|^2 - \varepsilon \Delta \phi. 
\end{aligned}
\end{equation}
Now we choose $G$ as 
\[ G(s) := \varepsilon^{1-\lambda}\left(1-\dfrac{1}{8}(s-\gamma)^2\right), \]
where $\gamma$ is as in the assumption (W2). 
Deu to the choice of $G$, the properties 
\begin{equation}\label{pro-g}
0 < G < \varepsilon^{1-\lambda}, \quad G^\prime W^\prime \ge 0, \quad G^{\prime \prime} = -\dfrac{ \varepsilon^{1-\lambda} }{4} 
\end{equation}
hold. 
Next, in oder to choose $\phi$, we define $\psi \in C^\infty([0, \infty);\mathbb{R}^+)$ as 
\begin{equation*}
\psi(s) = s \quad \mbox{for} \; s \in [0, \Cr{c-curva}/2], \quad \psi^\prime(s) = 0 \quad \mbox{for} \; s \in [\Cr{c-curva}, \infty), \quad |\psi^\prime| \le 1, \quad |\psi^{\prime\prime}| \le 4/\Cr{c-curva}. 
\end{equation*}
Let $\phi$ be defined by $\phi(y) := \kappa (\Cr{c-bdd-nabla}^2 + 1)\psi({\rm dist}(\partial \Omega_\varepsilon, y))$ and $\nu_\varepsilon$ be the outward unit normal to $\Omega_\varepsilon$. 
For $\varepsilon < 1$, we note the distance function is smooth on 
\[ N_\varepsilon := \{y \in \overline{\Omega}_\varepsilon : {\rm dist}(\partial \Omega_\varepsilon, y) \le \Cr{c-curva}\} \]
and 
\[ |\nabla {\rm dist}(\Omega_\varepsilon, y)| = 1, \quad \Delta {\rm dist}(\Omega_\varepsilon, y) \le \dfrac{(n-1)\kappa \varepsilon}{1-\Cr{c-curva}\kappa\varepsilon} \le \dfrac{6(n-1)\kappa}{5}\varepsilon \]
for $y \in N_\varepsilon$ since $\Cr{c-curva} \in (0,(6\kappa)^{-1}]$ and \eqref{scale-kappa} holds (see \cite{GT} for the details). 
Furthermore, we may see  
$\frac{\partial}{\partial \nu_\varepsilon}{\rm dist}(\Omega_\varepsilon, y) = - 1$
on $\partial \Omega_\varepsilon$. 
Hence $\phi$ is smooth and satisfies 
\begin{equation}\label{pro-phi}
0 < \phi \le \Cl[m]{m-up-bdd-xi2}, \quad |\nabla \phi| \le \Cr{m-up-bdd-xi2}, \quad \Delta \phi \le \Cr{m-up-bdd-xi2} 
\end{equation}
on $\overline{\Omega}_\varepsilon$ and $\frac{\partial}{\partial \nu_\varepsilon} \phi = -\kappa (\Cr{c-bdd-nabla}^2 + 1)$ on $\partial \Omega_\varepsilon$, where $\Cr{m-up-bdd-xi2}$ is a positive constant depending only on $n, \kappa$, $\Cr{c-curva}$ and $\Cr{c-bdd-nabla}$. 
By applying the inequalities \eqref{pro-g} and \eqref{pro-phi} for \eqref{para-xi2}, we obtain
\begin{equation}\label{para-xi3}
\partial_\tau \tilde{\xi}_\varepsilon - \Delta \tilde{\xi}_\varepsilon \le - \dfrac{2\langle (W^\prime + G^\prime) \nabla v_\varepsilon - \varepsilon \nabla \phi, \nabla \tilde{\xi}_\varepsilon\rangle}{|\nabla v_\varepsilon|^2} + \dfrac{\Cl[m]{m-up-bdd-xi3}\varepsilon}{|\nabla v_\varepsilon|} - \dfrac{\varepsilon^{1-\lambda}}{4}|\nabla v_\varepsilon|^2 + \Cr{m-up-bdd-xi2} \varepsilon 
\end{equation}
for any point $y$ such that $|\nabla v_\varepsilon(y)| \neq 0$, where $\Cr{m-up-bdd-xi3}$ is a positive constant depending only on $\Cr{m-up-bdd-xi2}$ and $\sup_{|s| \le 1} |W^\prime(s)|$. 
Now we fix an arbitrarily large $\tilde{T}>0$ and suppose for a contradiction that 
\begin{equation}\label{as-dis-con} 
\max_{y \in \overline{\Omega}_\varepsilon, \tau \in [0,\tilde{T}]} \tilde{\xi}_\varepsilon(y,\tau) \ge C\varepsilon^{1-\lambda} 
\end{equation}
for $\varepsilon < 1$ and some positive constant $C$ to be chosen. 
By the positivity of $W$ and $G$, the boundedness \eqref{as-u5}, \eqref{pro-phi} and the definition of $\tilde{\xi}_\varepsilon$, we see that $\tilde{\xi}_\varepsilon$ does not attain the maximum on $\overline{\Omega}_\varepsilon \times \{0\}$ if $C > \Cr{c-as-dis} + \Cr{m-up-bdd-xi2}$. 
Furthermore, by Lemma \ref{normal-di-u}, \eqref{neumann}, \eqref{scale-kappa}, \eqref{bdd-nabla-u} and the choice of $\tilde{\xi}_\varepsilon$ and $\phi$, we also see that 
\[ \langle \nabla \tilde{\xi}_\varepsilon, \nu_\varepsilon \rangle = A_{x, \varepsilon}(\nabla v_{\varepsilon}, \nabla v_{\varepsilon}) - \varepsilon \kappa (\Cr{c-bdd-nabla}^2 + 1) \le \varepsilon \kappa |\nabla v_{\varepsilon}|^2 - \varepsilon \kappa (\Cr{c-bdd-nabla}^2 + 1) < 0 \] 
on $\partial \Omega_\varepsilon \times (0, \tilde{T}]$, where $A_{x, \varepsilon}$ is the second fundamental form of $\partial \Omega_{\varepsilon}$. 
Hence the maximum point $(\tilde{y},\tilde{\tau})$ of the left hand side of \eqref{as-dis-con} is in $\Omega_\varepsilon \times (0, \tilde{T}]$, and we also see 
\begin{equation}\label{pro-max1}
\nabla \tilde{\xi}_\varepsilon(\tilde{y},\tilde{\tau}) = 0, \quad \Delta \tilde{\xi}_\varepsilon(\tilde{y},\tilde{\tau}) \le 0, \quad
\partial_\tau \tilde{\xi}_\varepsilon(\tilde{y},\tilde{\tau}) \ge 0. 
\end{equation}
By \eqref{pro-phi} and \eqref{as-dis-con}, we obtain 
\begin{equation}\label{pro-max2} 
|\nabla v_\varepsilon(\tilde{y},\tilde{\tau})|^2 \ge 2C\varepsilon^{1-\lambda} - 2\varepsilon \Cr{m-up-bdd-xi2} \ge 2\varepsilon^{1-\lambda}(C-\Cr{m-up-bdd-xi2}).  
\end{equation}
For sufficiently large $C$ so that the right hand side of \eqref{pro-max2} is positive, we must have $|\nabla v_\varepsilon| > 0$ in the neighborhood of $(\tilde{y},\tilde{\tau})$, thus we can apply \eqref{pro-max1} and \eqref{pro-max2} for \eqref{para-xi3} to obtain 
\[ 0 \le \dfrac{\varepsilon^{\frac{1}{2}+\frac{\lambda}{2}}\Cr{m-up-bdd-xi3}}{\sqrt{2(C-\Cr{m-up-bdd-xi2})}} - \dfrac{\varepsilon^{2-2\lambda}(C-\Cr{m-up-bdd-xi2})}{2} + \varepsilon \Cr{m-up-bdd-xi2}. \]
We note that $2-2\lambda \le (1+\lambda)/2 < 1$ from $\lambda \in [3/5,1)$. 
Thus choosing $C$ sufficiently large depending only on $\Cr{m-up-bdd-xi2}$ and $\Cr{m-up-bdd-xi3}$, we obtain a contradiction. 
Hence we proved 
\[ \max_{y \in \overline{\Omega}_\varepsilon, \tau \in [0,\tilde{T}]} \xi_\varepsilon(y,\tau) \le C\varepsilon^{1-\lambda} \]
for $\varepsilon < 1$ and sufficiently large $C$ depending only on $n, \kappa, \Cr{c-as-gra}, \Cr{c-as-dis}, \Cr{c-curva}, W$ and $\Omega$. 
Since $G \le \varepsilon^{1-\lambda}$, $\phi$ is nonnegative and $\tilde{T}$ is arbitrary, we obtain \eqref{dis-bdd} by choosing $\Cr{c-up-bdd-xi} = C+1$. 
\end{proof}


\section{Density ratio upper bound}\label{sec:den-rat}

In this section, we prove the upper density ratio bound for diffused interface energy. 
Define 
\[ 
D_{\varepsilon_i} (t) := \max \left\{\sup_{y \in N_{\Cr{c-curva}/2}\cap \overline{\Omega}, \; 0<r<\Cr{c-curva}} \dfrac{\mu_{\varepsilon_i}^t(B_r(y)) + \mu_{\varepsilon_i}^t(\tilde{B}_r(y))}{\omega_{n-1}r^{n-1}}, \sup_{y \in \Omega \setminus N_{\Cr{c-curva}/2}, \; 0<r<\Cr{c-curva}} \dfrac{\mu_{\varepsilon_i}^t(B_r(y))}{\omega_{n-1}r^{n-1}} \right\} 
\]
for $t \in [0,\infty)$. 
Estimates in this section are similar to \cite{LST,TT}. 
We note that \cite{LST,TT} study the singular limit problem of an Allen-Cahn type equation with a transport term on $\mathbb{T}^n := (\mathbb{R}/\mathbb{N})^n$ or $\mathbb{R}^n$, thus we have to expand their argument for our Neumann problem. 
In order to apply the reflection argument for the argument in \cite{LST,TT}, we consider not only the second density ratio but also the first density ratio of $D_{\varepsilon_i}(t)$. 

\begin{pro}\label{thm-den-rat}
For any $T>0$, there exist $\Cl[c]{c-den-rat}$ and $0<\Cl[e]{e-den-rat}<1$ depending only on $T$, $n$, $D_0$, $\alpha$, $W$, $\lambda$, $\kappa$, $\Cr{c-as-gra}$, $\Cr{c-as-dis}$, $\Cr{c-curva}$ and $\Omega$ such that 
\begin{equation}\label{den-rat-con} 
D_{\varepsilon_i}(t) \le \Cr{c-den-rat} 
\end{equation}
for all $t \in [0, T]$ and $\varepsilon_i \in (0, \Cr{e-den-rat})$. 
\end{pro}

\begin{rk}\label{rk:density}
Proposition \ref{thm-den-rat} and Proposition \ref{pro-con-den} show the boundedness of the density ratio for $\mu^t$, thus we can say ``$\mu^t$ behaves as an $(n-1)$-dimensional measure''. 
We also note that for any $(n-1)$-dimensional $C^1$ embedded submanifold $M \subset \Omega$ such that $\partial M \subset \partial \Omega$ the $(n-1)$-dimensional density 
\[ \Theta_*^{n-1}(x; \mu_M) := 
\begin{cases}
\lim_{r \downarrow 0} \dfrac{\mu_M(B_r(x))+\mu_M(\tilde{B}_r(x))}{\omega_{n-1}r^{n-1}} & {\rm if} \quad x \in N_{\Cr{c-curva}/2} \\
\lim_{r \downarrow 0} \dfrac{\mu_M(B_r(x))}{\omega_{n-1} r^{n-1}} & {\rm if} x \not \in N_{\Cr{c-curva}/2}
\end{cases} \]
of the measure $\mu_M := \mathcal{H}^{n-1}\lfloor_M$ satisfies 
\[ \Theta_*^{n-1}(x; \mu_M) = 1 \quad {\rm for} \quad x \in M \cup \partial M, \quad \Theta_*^{n-1}(x; \mu_M) = 0 \quad {\rm for} \quad x \not\in M \cup \partial M \]
since the usual $(n-1)$-dimensional density $\lim_{r \downarrow 0} \mu_M(B_r(x))/(\omega_{n-1} r^{n-1})$ is $1/2$ at $x \in \partial M$ and $\tilde{B}_r(x)$ is the reflected ball of $B_r(x)$ with respect to $\partial \Omega$. 
\end{rk}

In order to prove Proposition \ref{thm-den-rat}, we have to control the reflection ball, thus we cite the following lemma. 

\begin{lem}[{\cite[Lemma 4.2]{KT}}] \label{lem-ref-ball}
Assume $a \in N_{2\Cr{c-curva}}$ and $r>0$ satisfy ${\rm dist}(a,\partial \Omega) \le r$ and $B_r(a) \subset N_{3\Cr{c-curva}}$. 
Then 
\begin{equation}\label{ref-ball-bdd}
\tilde{B}_r(a) \subset B_{5r}(a).
\end{equation}
\end{lem}

\medskip

By the assumption \eqref{as-u3} and Lemma \ref{lem-ref-ball}, it is easy to see 
\begin{equation}\label{pro-d0}
D_{\varepsilon_i}(0) \le (1 + 5^{n-1}) D_0. 
\end{equation}
From now until Lemma \ref{lem-dis-mono}, we assume that 
\begin{equation}\label{as-bdd-den}
\sup_{[0,T_1]} D_{\varepsilon_i}(t) \le D_1 
\end{equation}
holds for some constants $T_1 > 0$ and $D_1 > 0$. 
Here, $D_1 > D_{\varepsilon_i}(0)$ is a constant depending only on $T$, $n$, $D_0$, $\alpha$, $W$, $\lambda$, $\kappa$, $\Cr{c-as-gra}$, $\Cr{c-as-dis}$, $\Cr{c-curva}$, $\Omega$ and not on $\varepsilon_i$, which will be determined in the proof of Proposition \ref{thm-den-rat}. 
Hereafter, to be careful that we do not end up in a circular argument, the dependence of any constant is written in detail. 
We also note that $D_{\varepsilon_i}(t)$ is continuous because of the regularity of $u_{\varepsilon_i}$ as in \eqref{ref-sol}. 
Hence $T_1 > 0$ follows from $D_1 > D_{\varepsilon_i}(0)$ and the continuity of $D_{\varepsilon_i}(t)$. 
For the following argument, we also define 
\[ \lambda^\prime := (1+\lambda)/2 \in (\lambda,1). \]

\begin{lem}\label{lem-lo-bdd-den}
Assume \eqref{as-bdd-den}. 
Then there exist $\Cl[c]{c-lo-bdd-den1}>1$, $0< \Cl[c]{c-lo-bdd-den2} < 1$ and $0<\Cl[e]{e-lo-bdd-den}<1$ depending only on $n$, $D_1$, $\alpha$, $W$, $\lambda$, $\kappa$, $\Cr{c-as-gra}$, $\Cr{c-as-dis}$, $\Cr{c-curva}$ and $\Omega$ with the following property: 
Assume $\varepsilon_i \in (0, \Cr{e-lo-bdd-den})$ and $|u_{\varepsilon_i}(y,s)| < \alpha$ with $y \in \overline{\Omega}$ and $s \in (0,T_1]$. 
Then for any $\max \{0, s-2\varepsilon_i^{2\lambda^\prime}\} \le t \le s$, 
\[
\Cr{c-lo-bdd-den2} \le 
\begin{cases}
\dfrac{1}{R^{n-1}} \left(\mu^t_{\varepsilon_i}(B_R(y)) + \mu^t_{\varepsilon_i}(\tilde{B}_{R}(y))\right) & \mbox{if} \quad y \in N_{\Cr{c-curva}/2}, \\
\dfrac{1}{R^{n-1}} \left(\mu^t_{\varepsilon_i}(B_R(y)) \right) & \mbox{if} \quad y \not\in N_{\Cr{c-curva}/2}
\end{cases}
\]
where $R= \Cr{c-lo-bdd-den1}(s+\varepsilon_i^2 - t)^{1/2}$. 
\end{lem}

\begin{rk}
As we mention in Section \ref{sec:int}, the domain is mostly divided into two regions $\{u_{\varepsilon_i} \approx 1\}$ and $\{u_{\varepsilon_i} \approx -1\}$ for sufficiently small $\varepsilon_i$ and the diffused interface energy should concentrate on the domain $\{|u_{\varepsilon_i}| < \alpha\}$, hence the estimate as in Lemma \ref{lem-lo-bdd-den} holds. 
For the concentration of the energy, we will discuss in Remark \ref{rk:concent}.
\end{rk}

\begin{proof}[Proof of Lemma \ref{lem-lo-bdd-den}]
For simplicity we omit the subscript $i$. 
First, we fix an arbitrary point $y \in N_{\Cr{c-curva}/2} \cap \overline{\Omega}$. 
Let $\gamma_0$ be a positive constant to be chosen. 
For $x \in B_{\gamma_0 \varepsilon}(y)\cap \Omega$, we obtain by \eqref{bdd-nabla-u} 
\begin{equation}\label{gamma-0} 
|u_\varepsilon(x,s)| \le \gamma_0 \sup_{x \in \Omega}\varepsilon |\nabla u_\varepsilon(y,s)| + |u_\varepsilon(y,s)| \le \Cr{c-bdd-nabla} \gamma_0 + \alpha \le \dfrac{1+\alpha}{2} < 1 
\end{equation}
for sufficiently small $\gamma_0$ depending only on $\alpha$ and $\Cr{c-bdd-nabla}$. 
Due to the assumption (W1), there exists a constant $c>0$ such that $W(u(x,s)) \ge c$ for $x \in B_{\gamma_0 \varepsilon}(y)\cap \Omega$, hence we have 
\begin{equation}\label{lo-bdd-xi0} 
\int_{B_{\gamma_0 \varepsilon}(y)} \rho_{1, (y,s+\varepsilon^2)}(x,s) \; d\mu^s_\varepsilon(x) \ge \dfrac{c}{(4\pi)^\frac{n-1}{2}\varepsilon^n}\int_{B_{\gamma_0 \varepsilon}(y) \cap \Omega} e^{-\frac{|x-y|^2}{4\varepsilon^2}} \; dx \ge \Cl[m]{m-lo-bdd-den1}, 
\end{equation}
where $\Cr{m-lo-bdd-den1}$ is a positive constant depending only on $n, \alpha, \Cr{c-bdd-nabla}, c$ and $\Omega$. 
Since $\tilde{B}_r(y) \cap \Omega = \emptyset$ if $r < {\rm dist}(y,\partial \Omega)$ and \eqref{ref-ball-bdd} with $a=y$ holds if $r \ge {\rm dist}(y,\partial \Omega)$, we obtain 
\begin{equation}\label{lo-bdd-xi1}
\begin{aligned}
&\; \int_\Omega \rho_{1,(y,s+\varepsilon^2)}(x,\tau) + \rho_{2,(y,s+\varepsilon^2)}(x,\tau) \; dx \\
=&\; \dfrac{1}{(4\pi(s+\varepsilon^2 - \tau)^{\frac{n-1}{2}}} \int_0^1 \mathcal{L}^n\left(\left\{x \in \Omega : e^{-\frac{|x-y|^2}{4(s+\varepsilon^2-\tau)}} \ge l\right\}\right) + \mathcal{L}^n\left(\left\{x \in \Omega : e^{-\frac{|x-y|^2}{4(s+\varepsilon^2-\tau)}} \ge l\right\}\right) \; dl \\
=&\; \dfrac{1}{(4\pi(s+\varepsilon^2 - \tau)^{\frac{n-1}{2}}} \int_0^1 \mathcal{L}^n\left(B_{(-4(s+\varepsilon^2-\tau)\log l)^{\frac{1}{2}}}(y) \cap \Omega\right) + \mathcal{L}^n\left(\tilde{B}_{(-4(s+\varepsilon^2-\tau)\log l)^{\frac{1}{2}}}(y) \cap \Omega\right) \; dl\\
\le&\; \dfrac{(1+5^n)\omega_n\sqrt{4\pi(s+\varepsilon^2-\tau)}}{\pi^{\frac{n}{2}}} \int_0^1 (-\log l)^{\frac{n}{2}} \; dl \\
=&\; \dfrac{(1+5^n)\omega_n\sqrt{4\pi(s+\varepsilon^2-\tau)}}{\pi^{\frac{n}{2}}} \int_0^\infty a^{\frac{n}{2}+1}e^{-a} \; da = (1+5^n) \sqrt{4\pi(s+\varepsilon^2-\tau)} 
\end{aligned}
\end{equation}
for $0< \tau < s + \varepsilon^2$. 
Here we use the change of variables $l=e^{-a}$. 
Combining \eqref{monotonicity} where $s,t$ substituted by $s+\varepsilon^2, \tau \in [t,s]$ respectively, \eqref{dis-bdd} and \eqref{lo-bdd-xi1}, we have 
\begin{equation}\label{lo-bdd-xi2}
\begin{aligned} 
&\; \dfrac{d}{d\tau}\left( e^{\Cr{c-mono-1}(s+\varepsilon^2-\tau)^\frac{1}{4}} \int_\Omega \rho_{1,(y,s+\varepsilon^2)}(x,\tau) + \rho_{2,(y,s+\varepsilon^2)} (x,\tau) \; d\mu_\varepsilon^\tau\right) \\
&\; \le e^{\Cr{c-mono-1}3^\frac{1}{4}}\left(\Cr{c-mono-2} + \Cr{c-up-bdd-xi}\varepsilon^{-\lambda}(1+5^n) \dfrac{\sqrt{\pi}}{\sqrt{s+\varepsilon^2-\tau}}\right). 
\end{aligned}
\end{equation}
Here $s-t \le 2\varepsilon^{2\lambda^\prime} \le 2$ is used. 
Integrating \eqref{lo-bdd-xi2} over $[t,s]$, we have by \eqref{lo-bdd-xi0} 
\begin{equation}\label{lo-bdd-xi3}
\Cr{m-lo-bdd-den1} \le \Cl[m]{m-lo-bdd-den2} \int_\Omega \rho_{1,(y,s+\varepsilon^2)}(x,t) + \rho_{2,(y,s+\varepsilon^2)}(x,t) \; d\mu_\varepsilon^t + \Cr{m-lo-bdd-den2}(\varepsilon^{2\lambda^\prime} + \varepsilon^{\lambda^\prime - \lambda}), 
\end{equation}
where $\Cr{m-lo-bdd-den2}$ is a positive constant depending only on $\Cr{c-mono-1}, \Cr{c-mono-2}, \Cr{c-up-bdd-xi}$ and $n$. 
We estimate the integral part of \eqref{lo-bdd-xi3}. 
Let $R = C(s+\varepsilon^2 - t)^{1/2}$, where $C$ is a constant to be chosen latter. 
From ${\rm spt}\rho_1 \subset B_{\Cr{c-curva}/2}(y)$ and ${\rm spt}\rho_2 \subset \tilde{B}_{\Cr{c-curva}/2}(y)$, for sufficiently small $\varepsilon$ so that $R < \Cr{c-curva}/2$, we obtain by the assumption $\sup_{[0,T_1]} D_\varepsilon \le D_1$ 
\begin{equation}\label{lo-bdd-xi4}
\begin{aligned}
&\; \int_\Omega \rho_{1,(y,s+\varepsilon^2)}(x,t) + \rho_{2,(y,s+\varepsilon^2)}(x,t) \; d\mu_\varepsilon^t \\
\le &\; \dfrac{C^{n-1}}{(\sqrt{4\pi}R)^{n-1}} \left(\int_{\Omega \cap B_{\Cr{c-curva}/2}(y)} e^{-\frac{C^2|x-y|^2}{4R^2}} \; d\mu_\varepsilon^t + \int_{\Omega \cap \tilde{B}_{\Cr{c-curva}/2}(y)} e^{-\frac{C^2|\tilde{x}-y|^2}{4R^2}} \; d\mu_\varepsilon^t \right) \\
\le &\; \dfrac{C^{n-1}}{(\sqrt{4\pi}R)^{n-1}} \left(\mu_\varepsilon^t(B_R(y)) + \mu_\varepsilon^t(\tilde{B}_R(y))\right) \\
&\; + \dfrac{C^{n-1}}{(\sqrt{4\pi}R)^{n-1}} \int^{e^{-\frac{C^2}{4}}}_0 \mu_\varepsilon^t \left( \left\{x \in (\Omega \cap B_{\Cr{c-curva}/2}(y)) \setminus B_R(y) : e^{-\frac{C^2|x-y|^2}{4R^2}} \ge l \right\} \right) \\
&\; \quad \quad \quad \quad \quad \quad \quad \quad \quad  + \mu_\varepsilon^t \left( \left\{x \in (\Omega \cap \tilde{B}_{\Cr{c-curva}/2}(y)) \setminus \tilde{B}_R(y) : e^{-\frac{C^2|\tilde{x}-y|^2}{4R^2}} \ge l \right\} \right) \; dl \\
\le &\; \dfrac{C^{n-1}}{(\sqrt{4\pi}R)^{n-1}} \left(\mu_\varepsilon^t(B_R(y)) + \mu_\varepsilon^t(\tilde{B}_R(y))\right) + \dfrac{\omega_{n-1}D_1}{\pi^\frac{n-1}{2}} \int^{e^{-\frac{C^2}{4}}}_0 (-\log l)^\frac{n-1}{2} \; dl \\
= &\; \dfrac{C^{n-1}}{(\sqrt{4\pi}R)^{n-1}} \left(\mu_\varepsilon^t(B_R(y)) + \mu_\varepsilon^t(\tilde{B}_R(y))\right) + \dfrac{2^\frac{n+1}{2}\omega_{n-1}D_1}{\pi^\frac{n-1}{2}} \int_\frac{C^2}{8}^\infty a^\frac{n-1}{2} e^{-2a} \; da \\
\le &\; \dfrac{C^{n-1}}{(\sqrt{4\pi}R)^{n-1}} \left(\mu_\varepsilon^t(B_R(y)) + \mu_\varepsilon^t(\tilde{B}_R(y))\right) + 2^\frac{n+1}{2}D_1 e^{-\frac{C^2}{8}}. 
\end{aligned}
\end{equation} 
Here we use the change of variables $l = e^{-2a}$. 
Now, we fix a sufficiently large $C>0$ to satisfy $\Cr{m-lo-bdd-den2}2^\frac{n+1}{2}D_1 e^{-\frac{C^2}{8}} \le \Cr{m-lo-bdd-den1}/4$. 
Setting 
\[ \Cr{c-lo-bdd-den1} = C, \quad \Cr{c-lo-bdd-den2} = \dfrac{(4\pi)^\frac{n-1}{2}\Cr{m-lo-bdd-den1}}{2C^{n-1}\Cr{m-lo-bdd-den2}} \]
and choosing sufficiently small $\Cr{e-lo-bdd-den}$ to satisfy $\Cr{m-lo-bdd-den2}(\Cr{e-lo-bdd-den}^{2\lambda^\prime} + \Cr{e-lo-bdd-den}^{\lambda^\prime-\lambda}) \le \Cr{m-lo-bdd-den1}/4$ and $R \le C(\Cr{e-lo-bdd-den}^{2\lambda^\prime} + \Cr{e-lo-bdd-den}^{2})^{1/2} < \Cr{c-curva}/2$, we obtain the conclusion from \eqref{lo-bdd-xi3} and \eqref{lo-bdd-xi4}. 
The case of $y \in \Omega \setminus N_{\Cr{c-curva}/2}$ may be proved using \eqref{monotonicity-2}. 
\end{proof}

\begin{lem}
Assume \eqref{as-bdd-den}. 
Then there exist $0 < \Cl[e]{e-dis-den} \le \Cr{e-lo-bdd-den}$ and $\Cl[c]{c-dis-den}$ depending only on $n$, $D_1$, $\alpha$, $W$, $\lambda$, $\kappa$, $\Cr{c-as-gra}$, $\Cr{c-as-dis}$, $\Cr{c-curva}$ and $\Omega$ with the following property: 
For any $\varepsilon_i \in (0, \Cr{e-dis-den}]$, $y \in \overline{\Omega}$, $r \in (\varepsilon_i^{\lambda^\prime}, \Cr{c-curva}/2)$ and $t \in [2\varepsilon_i^{2\lambda^\prime},\infty) \cap [0,T_1]$, 
\begin{equation}\label{dis-den-conclusion1} 
\int_{B_r(y)\cap \Omega} \left(\dfrac{\varepsilon_i |\nabla u_{\varepsilon_i}|^2}{2} - \dfrac{W(u_{\varepsilon_i})}{\varepsilon_i} \right)^+ dx + \int_{\tilde{B}_r(y)\cap \Omega} \left(\dfrac{\varepsilon_i |\nabla u_{\varepsilon_i}|^2}{2} - \dfrac{W(u_{\varepsilon_i})}{\varepsilon_i} \right)^+ dx \le \Cr{c-dis-den} \varepsilon^{\lambda^\prime-\lambda} r^{n-1} 
\end{equation}
if $y \in N_{\Cr{c-curva}/2}$ and 
\[
\int_{B_r(y)\cap \Omega} \left(\dfrac{\varepsilon_i |\nabla u_{\varepsilon_i}|^2}{2} - \dfrac{W(u_{\varepsilon_i})}{\varepsilon_i} \right)^+ dx \le \Cr{c-dis-den} \varepsilon^{\lambda^\prime-\lambda} r^{n-1} 
\]
if $y \not \in N_{\Cr{c-curva}/2}$. 
\end{lem}

\begin{proof}
For simplicity we omit the subscript $i$. 
We only need to prove the claim when $T_1 \ge 2\varepsilon^{2\lambda^\prime}$ since the claim is vacuously true otherwise. 
Let $y \in \overline{\Omega}$, $r \in (\varepsilon^{\lambda^\prime}, \Cr{c-curva}/2)$ and $t \in [2\varepsilon^{2\lambda^\prime},\infty) \cap [0,T_1]$ be arbitrary and fixed. 
We define 
\begin{align*} 
&A_1 := \{ x \in B_{10r}(y) \cap \Omega : \mbox{for some} \; \tilde{t} \; \mbox{with} \; t-\varepsilon^{2\lambda^\prime} \le \tilde{t} \le t, |u(x,\tilde{t})| \le \alpha \}, \\
&A_2 := \{ x \in B_{10r + 2\Cr{c-lo-bdd-den1}\varepsilon^{\lambda^\prime}}(y) \cap \Omega : {\rm dist}(A_1, x) < 2\Cr{c-lo-bdd-den1}\varepsilon^{\lambda^\prime} \}. 
\end{align*}
By Vitali's covering theorem applied to $\mathcal{F} = \{\overline{B}_{2\Cr{c-lo-bdd-den1}\varepsilon^{\lambda^\prime}}(x) : x \in A_1\}$, there exists a set of pairwise disjoint balls $\{B_{2\Cr{c-lo-bdd-den1}\varepsilon^{\lambda^\prime}}(x_i)\}_{i=1}^N$ such that 
\begin{equation}\label{dis-den2} 
x_i \in A_1 \; \mbox{for each}\; i=1, \cdots, N, \quad \mbox{and} \quad A_2 \subset \cup_{i=1}^N \overline{B}_{10\Cr{c-lo-bdd-den1}\varepsilon^{\lambda^\prime}}(x_i). 
\end{equation}
By the definition of $A_1$, for each $x_i$ there exists $\tilde{t}_i$ such that 
\[ t-\varepsilon^{2\lambda^\prime} \le \tilde{t}_i \le t, \quad |u(x_i,\tilde{t}_i)| \le \alpha. \]
Thus, the assumption of Lemma \ref{lem-lo-bdd-den} is satisfied for $s=\tilde{t}_i$, $y=x_i$, $t=t-2\varepsilon^{2\lambda^\prime}$ and $R=R_i := \Cr{c-lo-bdd-den1}(\tilde{t}_i + \varepsilon^2 - (t - 2\varepsilon^{2\lambda^\prime}))^\frac{1}{2}$ if $\varepsilon < \Cr{e-lo-bdd-den}$. 
Hence we may conclude that 
\[ \Cr{c-lo-bdd-den2} R_i^{n-1} \le \mu_\varepsilon^{t-2\varepsilon^{2\lambda^\prime}}(B_{R_i}(x_i)) + \mu_\varepsilon^{t-2\varepsilon^{2\lambda^\prime}}(\tilde{B}_{R_i}(x_i)) \quad \mbox{for} \quad i=1, \cdots, N, \]
where here and in the following we set $\tilde{B}_{R_i}(x_i) = \emptyset$ if $x_i \not \in N_{\Cr{c-curva}/2}$. 
Due to the definition of $R_i$ and $-\varepsilon^{2\lambda^\prime} \le \tilde{t}_i - t \le 0$, we obtain 
\[ \Cr{c-lo-bdd-den1}(\varepsilon^{2\lambda^\prime}+\varepsilon^2)^\frac{1}{2} \le R_i \le \Cr{c-lo-bdd-den1}(2\varepsilon^{2\lambda^\prime} + \varepsilon^2)^\frac{1}{2} \le 2 \Cr{c-lo-bdd-den1} \varepsilon^{\lambda^\prime}, \]
which shows 
\begin{equation}\label{dis-den1} 
\Cr{c-lo-bdd-den2}\Cr{c-lo-bdd-den1}^{n-1}\varepsilon^{\lambda^\prime (n-1)} \le \mu_\varepsilon^{t-2\varepsilon^{2\lambda^\prime}}(B_{2\Cr{c-lo-bdd-den1}\varepsilon^{\lambda^\prime}}(x_i)) + \mu_\varepsilon^{t-2\varepsilon^{2\lambda^\prime}}(\tilde{B}_{2\Cr{c-lo-bdd-den1}\varepsilon^{\lambda^\prime}}(x_i)). 
\end{equation}
Note that if $y \not \in N_{11\Cr{c-curva}/2}$ and $\varepsilon$ is sufficiently small so that $2\Cr{c-lo-bdd-den1}\varepsilon^{\lambda^\prime} < \Cr{c-curva}/2$, we can regard $\tilde{B}_{2\Cr{c-lo-bdd-den1}\varepsilon^{\lambda^\prime}}(x_i)$ as the empty set for all $i$. 
Since $\{B_{2\Cr{c-lo-bdd-den1}\varepsilon^{\lambda^\prime}}(x_i)\}_{i=1}^N$ and $\{\tilde{B}_{2\Cr{c-lo-bdd-den1}\varepsilon^{\lambda^\prime}}(x_i)\}_{i=1}^N$ are pairwise disjoint, respectively, $B_{2\Cr{c-lo-bdd-den1}\varepsilon^{\lambda^\prime}}(x_i) \subset B_{10r + 2\Cr{c-lo-bdd-den1}\varepsilon^{\lambda^\prime}}(y)$ and $\tilde{B}_{2\Cr{c-lo-bdd-den1}\varepsilon^{\lambda^\prime}}(x_i) \subset \tilde{B}_{10r + 2\Cr{c-lo-bdd-den1}\varepsilon^{\lambda^\prime}}(y)$, \eqref{dis-den1} gives 
\begin{equation}\label{dis-den3} 
N \Cr{c-lo-bdd-den2}\Cr{c-lo-bdd-den1}^{n-1}\varepsilon^{\lambda^\prime (n-1)} \le \mu_\varepsilon^{t-2\varepsilon^{2\lambda^\prime}}(B_{10r+2\Cr{c-lo-bdd-den1}\varepsilon^{\lambda^\prime}}(y)) + \mu_\varepsilon^{t-2\varepsilon^{2\lambda^\prime}}(\tilde{B}_{10r+2\Cr{c-lo-bdd-den1}\varepsilon^{\lambda^\prime}}(y)) 
\end{equation}
provided $\tilde{B}_{10r+2\Cr{c-lo-bdd-den1}\varepsilon^{\lambda^\prime}}(y) = \emptyset$ if $y \not \in N_{11\Cr{c-curva}/2}$ and $\varepsilon$ is sufficiently small so that $2\Cr{c-lo-bdd-den1}\varepsilon^{\lambda^\prime} < \Cr{c-curva}/2$. 
Thus, the $n$-dimensional volume of $A_2$ is estimated by \eqref{energy-bdd}, \eqref{dis-den2} and \eqref{dis-den3} 
\begin{equation}\label{dis-den4}
\mathcal{L}^n(A_2) \le N\omega_n (10\Cr{c-lo-bdd-den1}\varepsilon^{\lambda^\prime})^n \le \dfrac{10^n\Cr{c-lo-bdd-den1}\omega_n \varepsilon^{\lambda^\prime}}{\Cr{c-lo-bdd-den2}}2\Cr{c-ene0} =: \Cl[m]{m-dis-den1}\varepsilon^{\lambda^\prime}. 
\end{equation}
Hence by \eqref{dis-bdd} and \eqref{dis-den4} 
\begin{equation}\label{dis-con1} 
\int_{A_2 \cap B_r(y)} \left(\dfrac{\varepsilon |\nabla u_{\varepsilon}|^2}{2} - \dfrac{W(u_{\varepsilon})}{\varepsilon} \right)^+ dx \le \mathcal{L}^n(A_2) \Cr{c-up-bdd-xi} \varepsilon^{-\lambda} \le \Cr{m-dis-den1} \Cr{c-up-bdd-xi} \varepsilon^{\lambda^\prime-\lambda}
\end{equation}
if $y \not \in N_{\Cr{c-curva}/2}$ and 
\begin{equation}\label{dis-con2} 
\int_{A_2 \cap B_r(y)} \left(\dfrac{\varepsilon |\nabla u_{\varepsilon}|^2}{2} - \dfrac{W(u_{\varepsilon})}{\varepsilon} \right)^+ dx + \int_{A_2 \cap \tilde{B}_r(y)} \left(\dfrac{\varepsilon |\nabla u_{\varepsilon}|^2}{2} - \dfrac{W(u_{\varepsilon})}{\varepsilon} \right)^+ dx \le 2\Cr{m-dis-den1} \Cr{c-up-bdd-xi} \varepsilon^{\lambda^\prime-\lambda}
\end{equation} 
if $y \in N_{\Cr{c-curva}/2}$. 

Next we estimate the diffused surface energy on the intersection of $\tilde{B}_r(y)$ and the complement of $A_2$ with $y \in N_{\Cr{c-curva}/2}$ which decays very quickly. 
Define $\phi \in {\rm Lip}(\tilde{B}_{2r}(y))$ such that 
\[ \phi(x) := 
\begin{cases}
1 & \mbox{if} \quad x \in \tilde{B}_r(y) \setminus A_2, \\
0 & \mbox{if} \quad {\rm dist} (x, \tilde{B}_r(y) \setminus A_2) \ge \varepsilon^{\lambda^\prime},
\end{cases}
\quad |\nabla \phi| \le 2\varepsilon^{-\lambda^\prime}, \quad 0 \le \phi \le 1. 
\]
Note that $\tilde{B}_{2r}(y) \cap \Omega \subset B_{10r}(y)\cap \Omega$ since $\tilde{B}_{2r}(y) \cap \Omega = \emptyset$ if ${\rm dist}(y, \partial \Omega) > 2r$ and \eqref{ref-ball-bdd} replaced $a,r$ by $y,2r$, respectively, holds if ${\rm dist}(y, \partial \Omega) \le 2r$. 
By $r \ge \varepsilon^{\lambda^\prime}$, $2\Cr{c-lo-bdd-den1} \varepsilon^{\lambda^\prime} > \varepsilon^{\lambda^\prime}$ and the definitions of $A_1$ and $\phi$, we have ${\rm spt} \phi \cap A_1 = \emptyset$, hence 
\begin{equation}\label{dis-den6} 
|u_\varepsilon(x,s)| \ge \alpha, \quad \mbox{for} \quad x \in {\rm spt} \phi \cap \Omega, \; s \in [t-\varepsilon^{2\lambda^\prime}, t]. 
\end{equation}
For each $j$ differentiate the equation \eqref{ac} with respect to $x_j$, multiply $\phi^2 \partial_{x_j} u_\varepsilon$, sum over $j$ and integrate to obtain 
\begin{equation} \label{dis-den5}
\dfrac{d}{dt} \int_\Omega \dfrac{|\nabla u_\varepsilon|^2}{2}\phi^2 \; dx = \int_\Omega \left( \langle \nabla u_\varepsilon, \Delta \nabla u_\varepsilon \rangle - \dfrac{W^{\prime\prime}(u_\varepsilon)}{\varepsilon^2}|\nabla u_\varepsilon|^2 \right) \phi^2 \; dx. 
\end{equation}
By integration by parts, the Cauchy-Schwarz inequality and the Neumann boundary condition \eqref{neumann}, \eqref{dis-den5} gives 
\begin{equation}\label{dis-den7}
\dfrac{d}{dt} \int_\Omega \dfrac{|\nabla u_\varepsilon|^2}{2}\phi^2 \; dx \le \int_\Omega |\nabla \phi|^2 |\nabla u_\varepsilon|^2 \; dx - \int_\Omega \dfrac{W^{\prime \prime}(u_\varepsilon)}{\varepsilon^2} |\nabla u_\varepsilon|^2 \phi^2 \; dx. 
\end{equation}
From \eqref{dis-den6}, the assumption (W3) and the definition of $\phi$, we have by \eqref{dis-den7} 
\begin{equation}\label{dis-den8}
\dfrac{d}{dt} \int_\Omega \dfrac{|\nabla u_\varepsilon|^2}{2}\phi^2 \; dx \le 4 \varepsilon^{-2\lambda^\prime} \int_{{\rm spt} \phi \cap \Omega} |\nabla u_\varepsilon|^2 \; dx - \dfrac{\beta}{\varepsilon^2} \int_\Omega |\nabla u_\varepsilon|^2 \phi^2 \; dx. 
\end{equation}
Integrating \eqref{dis-den8} over $[t-\varepsilon^{2\lambda^\prime},t]$, we obtain 
\begin{equation}\label{dis-den9}
\begin{aligned}
\int_\Omega \dfrac{|\nabla u_\varepsilon|^2}{2}\phi^2 (x,t) \; dx \le &\; e^{-\beta\varepsilon^{2(\lambda^\prime -1)}} \int_\Omega \dfrac{|\nabla u_\varepsilon|^2}{2}\phi^2 (x,t-\varepsilon^{2\lambda^\prime}) \; dx \\
&\; + \int^t_{t-\varepsilon^{2\lambda^\prime}} e^{-\frac{\beta}{\varepsilon^2}(t-s)} 4\varepsilon^{-2\lambda^\prime} \left(\int_{{\rm spt} \phi \cap \Omega} |\nabla u_\varepsilon|^2 \; dx\right) \; ds. 
\end{aligned}
\end{equation}
By ${\rm spt} \phi \subset \tilde{B}_{2r}(y)$, $r \le \Cr{c-curva}/2$ and \eqref{as-bdd-den} we have 
\begin{equation}\label{dis-den10}
\sup_{s \in [t-\varepsilon^{2\lambda^\prime},t]} \int_{{\rm spt}\phi \cap \Omega} \dfrac{|\nabla u_\varepsilon|^2}{2}(x,s) \; dx \le D_1\omega_{n-1}(2r)^{n-1}. 
\end{equation}
Combining \eqref{dis-den9}, \eqref{dis-den10}, $\lambda^\prime < 1$ and $\lambda^\prime - \lambda < 2(1-\lambda^\prime)$, we obtain 
\begin{equation}\label{dis-con3}
\begin{aligned}
\int_{(\tilde{B}_r(y)\cap \Omega)\setminus A_2} \dfrac{|\nabla u_\varepsilon(x,t)|^2}{2} \; dx \le&\; \int_\Omega \dfrac{|\nabla u_\varepsilon|^2}{2} \phi^2(x,t) \; dx \\
\le&\; D_1 \omega_{n-1} (2r)^{n-1} \left( e^{-\beta\varepsilon^{2(\lambda^\prime-1)}} + \dfrac{8}{\beta}\varepsilon^{2(1-\lambda^\prime)} \right) \\
\le&\; \dfrac{9 D_1 \omega_{n-1} (2r)^{n-1}}{\beta} \varepsilon^{\lambda^\prime - \lambda} 
\end{aligned}
\end{equation}
for sufficiently small $\varepsilon$ depending only on $\beta$. 
Similarly, we may obtain 
\begin{equation}\label{dis-con4}
\int_{(B_r(y)\cap \Omega)\setminus A_2} \dfrac{|\nabla u_\varepsilon(x,t)|^2}{2} \; dx \le \dfrac{9 D_1 \omega_{n-1} (2r)^{n-1}}{\beta} \varepsilon^{\lambda^\prime - \lambda}
\end{equation}
for all $y \in \overline{\Omega}$ by replacing $\phi$ as $\phi \in {\rm Lip}(B_{2r}(y))$ such that 
\[ \phi(x) = 
\begin{cases}
1 & \mbox{if} \quad x \in B_r(y) \setminus A_2, \\
0 & \mbox{if} \quad {\rm dist} (x, B_r(y) \setminus A_2) \ge \varepsilon^{\lambda^\prime},
\end{cases}
\quad |\nabla \phi| \le 2\varepsilon^{-\lambda^\prime}, \quad 0 \le \phi \le 1. 
\] 
By \eqref{dis-con1}, \eqref{dis-con2}, \eqref{dis-con3} and \eqref{dis-con4}, we obtain the conclusion with an appropriate choice of $\Cr{e-dis-den}$ and $\Cr{c-dis-den}$. 
\end{proof}

\begin{lem}\label{lem-dis-mono}
Assume \eqref{as-bdd-den}. 
There exists a constant $\Cl[c]{c-dis-mono-bdd}$ depending only on $n$, $D_1$, $\alpha$, $W$, $\lambda$, $\kappa$, $\Cr{c-as-gra}$, $\Cr{c-as-dis}$, $\Cr{c-curva}$ and $\Omega$ such that for $\varepsilon_i < \Cr{e-dis-den}$, $y \in \overline{\Omega}$, $t \in [0,T_1]$ and $t\le s$, 
\begin{equation}\label{dis-mono-conclusion1}
\int^t_0 \int_\Omega \dfrac{\rho_{1,(y,s)}(x,\tau)+\rho_{2,(y,s)}(x,\tau)}{2(s-\tau)} \left(\dfrac{\varepsilon_i |\nabla u_{\varepsilon_i}|^2}{2} - \dfrac{W(u_{\varepsilon_i})}{\varepsilon_i}\right)^+ dx d\tau \le \Cr{c-dis-mono-bdd} \varepsilon_i^{\lambda^\prime-\lambda} (1+ |\log \varepsilon_i| + (\log s)^+) 
\end{equation}
if $y \in N_{\Cr{c-curva}/2}$ and 
\begin{equation}\label{dis-mono-conclusion2}
\int^t_0 \int_\Omega \dfrac{\rho_{1,(y,s)}(x,\tau)}{2(s-\tau)} \left(\dfrac{\varepsilon_i |\nabla u_{\varepsilon_i}|^2}{2} - \dfrac{W(u_{\varepsilon_i})}{\varepsilon_i}\right)^+ dx d\tau \le \Cr{c-dis-mono-bdd} \varepsilon_i^{\lambda^\prime-\lambda} (1+ |\log \varepsilon_i| + (\log s)^+) 
\end{equation}
if $y \not\in N_{\Cr{c-curva}/2}$.  
\end{lem}

\begin{proof}
Omit the subscript $i$. 
First, we show 
\begin{equation}\label{dis-mono-con1} 
\int^t_0 \int_\Omega \dfrac{\rho_{2,(y,s)}(x,\tau)}{2(s-\tau)} \left(\dfrac{\varepsilon |\nabla u_{\varepsilon}|^2}{2} - \dfrac{W(u_{\varepsilon})}{\varepsilon}\right)^+ dx d\tau \le C \varepsilon^{\lambda^\prime-\lambda} (1+ |\log \varepsilon| + (\log s)^+) 
\end{equation}
for a constant $C$ to be chosen latter in the case of $y \in N_{\Cr{c-curva}/2}$. 
If $t \le 2\varepsilon^{2\lambda^\prime}$ then by using \eqref{dis-bdd} and the similar argument for \eqref{lo-bdd-xi1} we have 
\begin{equation}\label{dis-mono1}
\int_0^t \int_\Omega \dfrac{\rho_{2,(y,s)}(x,\tau)}{2(s-\tau)} \left(\dfrac{\varepsilon |\nabla u_{\varepsilon}|^2}{2} - \dfrac{W(u_{\varepsilon})}{\varepsilon}\right)^+ dx d\tau  \le \int^t_0 \dfrac{5^n \sqrt{\pi}\Cr{c-up-bdd-xi}\varepsilon^{-\lambda}}{\sqrt{s-\tau}} \; d\tau \le 2 \cdot 5^n \sqrt{2\pi}\Cr{c-up-bdd-xi}\varepsilon^{\lambda^\prime-\lambda}.
\end{equation}
By the similar argument, if $s \ge t \ge s-2\varepsilon^{2\lambda^\prime}$ then we have 
\begin{equation}\label{dis-mono2}
\int_{s-2\varepsilon^{2\lambda^\prime}}^t \int_\Omega \dfrac{\rho_{2,(y,s)}(x,\tau)}{2(s-\tau)} \left(\dfrac{\varepsilon |\nabla u_{\varepsilon}|^2}{2} - \dfrac{W(u_{\varepsilon})}{\varepsilon}\right)^+ dx d\tau  \le 2 \cdot 5^n \sqrt{2\pi}\Cr{c-up-bdd-xi}\varepsilon^{\lambda^\prime-\lambda}.
\end{equation}
Hence we only need to estimate integral over $[2\varepsilon^{2\lambda^\prime},t]$ with $t \le s-2\varepsilon^{2\lambda^\prime}$. 
First we estimate on $\tilde{B}_{\varepsilon^{\lambda^\prime}}(y)\cap \Omega$. 
We compute using Lemma \ref{lem-ref-ball}, \eqref{dis-bdd} and $s-t \ge 2\varepsilon^{2\lambda^\prime}$ that 
\begin{equation}\label{dis-mono3}
\begin{aligned}
&\; \int_{2\varepsilon^{2\lambda^\prime}}^t \int_{\tilde{B}_{\varepsilon^{\lambda^\prime}}(y)\cap \Omega} \dfrac{\rho_{2,(y,s)}(x,\tau)}{2(s-\tau)} \left(\dfrac{\varepsilon |\nabla u_{\varepsilon}|^2}{2} - \dfrac{W(u_{\varepsilon})}{\varepsilon}\right)^+ dx d\tau \\
&\; \le \int_{2\varepsilon^{2\lambda^\prime}}^t \dfrac{5^n \Cr{c-up-bdd-xi} \omega_n \varepsilon^{n\lambda^\prime} \varepsilon^{-\lambda}}{2 (\sqrt{4 \pi})^{n-1} (s-\tau)^\frac{n+1}{2}} \; d\tau \le \dfrac{5^n\Cr{c-up-bdd-xi} \omega_n}{(n-1)(\sqrt{8 \pi})^{n-1}} \varepsilon^{\lambda^\prime - \lambda}. 
\end{aligned}
\end{equation}
On $\Omega \setminus \tilde{B}_{\varepsilon^{\lambda^\prime}}(y)$, by \eqref{dis-den-conclusion1}, $s-t \ge 2\varepsilon^{2\lambda^\prime}$ and computations similar to \eqref{lo-bdd-xi4}, we have 
\begin{equation}\label{dis-mono4}
\begin{aligned}
&\; \int^t_{2\varepsilon^{2\lambda^\prime}} \int_{\Omega \setminus \tilde{B}_{\varepsilon^{\lambda^\prime}}(y)} \dfrac{\rho_{2,(y,s)}(x,\tau)}{2(s-\tau)} \left(\dfrac{\varepsilon |\nabla u_{\varepsilon}|^2}{2} - \dfrac{W(u_{\varepsilon})}{\varepsilon}\right)^+ dx d\tau \\
\le &\; \int^t_{2\varepsilon^{2\lambda^\prime}} \dfrac{d\tau}{2(\sqrt{4\pi})^\frac{n-1}{2}(s-\tau)^\frac{n+1}{2}} \\
&\; \int_0^1\left\{\int_{((\Omega\cap \tilde{B}_{\Cr{c-curva}/2}(y)) \setminus \tilde{B}_{\varepsilon^{\lambda^\prime}}(y)) \cap \{x: e^{-\frac{|\tilde{x}-y|^2}{4(s-\tau)}} \ge l\}} \left(\dfrac{\varepsilon |\nabla u_{\varepsilon}|^2}{2} - \dfrac{W(u_{\varepsilon})}{\varepsilon}\right)^+ \; dx\right\} \; dl \\
\le &\; \Cr{c-dis-den} c(n) \varepsilon^{\lambda^\prime - \lambda} \int^t_{2\varepsilon^{2\lambda^\prime}} \dfrac{1}{s-\tau} \; d\tau 
\le \Cr{c-dis-den} c(n) \varepsilon^{\lambda^\prime - \lambda} (2\lambda^\prime \log(\varepsilon^{-1}) + \log s). 
\end{aligned}
\end{equation}
By \eqref{dis-mono1}--\eqref{dis-mono4}, we obtain \eqref{dis-mono-con1} with a constant $C$ depending only on $n, \Cr{c-up-bdd-xi}$ and $\Cr{c-dis-den}$. 
Similarly, we obtain 
\begin{equation}\label{dis-mono-con2} 
\int^t_0 \int_\Omega \dfrac{\rho_{1,(y,s)}(x,\tau)}{2(s-\tau)} \left(\dfrac{\varepsilon |\nabla u_{\varepsilon}|^2}{2} - \dfrac{W(u_{\varepsilon})}{\varepsilon}\right)^+ dx d\tau \le C \varepsilon^{\lambda^\prime-\lambda} (1 + |\log \varepsilon| + (\log s)^+) 
\end{equation}
for $y \in \overline{\Omega}$. 
Hence \eqref{dis-mono-con1} and \eqref{dis-mono-con2} imply the conclusion by choosing $\Cr{c-dis-mono-bdd} = 2C$. 
\end{proof}

\begin{proof}[Proof of Proposition \ref{thm-den-rat}]
Omit the subscript $i$. 
For $T>0$, we choose $\Cr{c-den-rat}$ as 
\[ \Cr{c-den-rat} := \max\left\{ \dfrac{(4\pi)^\frac{n-1}{2}\cdot e^{\Cr{c-mono-1}(T+\frac{\Cr{c-curva}^2}{16})^\frac{1}{4}}((1+5^{n-1})D_0 + \Cr{c-mono-2} T + 1)}{e^{-\frac{1}{4}}}, \dfrac{4^{n-1} \cdot 2\Cr{c-ene0}}{\Cr{c-curva}^{n-1}}\right\}. \]
Note that this choice of $\Cr{c-den-rat}$ does not depend on $D_1$ and let $D_1 := \Cr{c-den-rat} + 1$. 
For this $\Cr{c-den-rat}$, assume the conclusion \eqref{den-rat-con} was false. 
Then, by the continuity of $D_{\varepsilon}(t)$, there exist $y \in \overline{\Omega}$, $\tilde{t} \in (0,T]$, $0 < r \le \Cr{c-curva}$ and sufficiently small $\varepsilon$ such that 
\begin{equation}\label{den-rat-contra} 
\dfrac{\mu^{\tilde{t}}_\varepsilon(B_r(y)) + \mu^{\tilde{t}}_\varepsilon(\tilde{B}_r(y))}{\omega_{n-1}r^{n-1}} > \Cr{c-den-rat} \quad \mbox{if} \quad  y \in N_{\Cr{c-curva}/2}, \quad \quad \dfrac{\mu^{\tilde{t}}_\varepsilon(B_r(y)) }{\omega_{n-1}r^{n-1}} > \Cr{c-den-rat} \quad \mbox{if} \quad y \not\in N_{\Cr{c-curva}/2} 
\end{equation}
and $\sup_{t \in [0,\tilde{t}]}D(t) \le D_1$. 
First, we consider the case of $y \in N_{\Cr{c-curva}/2}$. 
For $r^\prime \ge \Cr{c-curva}/4$, we have by \eqref{energy-bdd} and the choice of $\Cr{c-den-rat}$
\begin{equation}\label{den-rat1} 
\dfrac{\mu^t_\varepsilon(B_{r^\prime}(y)) + \mu^t_\varepsilon(\tilde{B}_{r^\prime}(y))}{\omega_{n-1}{r^\prime}^{n-1}} \le \dfrac{4^{n-1} \cdot 2\Cr{c-ene0}}{\Cr{c-curva}^{n-1}} \le \Cr{c-den-rat}. 
\end{equation}
By \eqref{den-rat-contra} and \eqref{den-rat1}, we may see that $0 < r < \Cr{c-curva}/4$. 
Integrating \eqref{monotonicity} over $t \in (0,\tilde{t})$ with $s = \tilde{t}+r^2$ and applying \eqref{dis-mono-conclusion1}, we obtain by $\tilde{t} \le T$ and $s \le T + \frac{\Cr{c-curva}^2}{16}$
\begin{equation}\label{den-rat2}
\begin{aligned}
&\; e^{\Cr{c-mono-1}(s-t)^\frac{1}{4}} \int_\Omega \rho_{1,(y,s)}(x,t) + \rho_{2,(y,s)}(x,t) \; d\mu_\varepsilon^t(x) \Big|_{t=0}^{\tilde{t}} \\
&\; \le \int^{\tilde{t}}_0 e^{\Cr{c-mono-1}(s-t)^\frac{1}{4}} \left( \Cr{c-mono-2} + \int_\Omega \dfrac{\rho_{1,(y,s)}(x,t) + \rho_{2,(y,s)}(x,t)}{2(s-t)} \; d\xi^t_\varepsilon(x) \right) dt \\
&\; \le e^{\Cr{c-mono-1}(T+\frac{\Cr{c-curva}^2}{16})^\frac{1}{4}}\left\{\Cr{c-mono-2} T + \Cr{c-dis-mono-bdd} \varepsilon^{\lambda^\prime - \lambda}\left(1 + |\log \varepsilon| + \left(\log \left(T+\frac{\Cr{c-curva}^2}{16}\right)\right)^+\right)\right\}.
\end{aligned}
\end{equation}
By $s \le T + \frac{\Cr{c-curva}^2}{16}$, \eqref{pro-d0} and computations similar to \eqref{lo-bdd-xi4}, we obtain 
\begin{equation}\label{den-rat3}
e^{\Cr{c-mono-1}s^\frac{1}{4}} \int_\Omega \rho_{1,(y,s)}(x,0) + \rho_{2,(y,s)}(x,0) \; d\mu_\varepsilon^0(x) \le e^{\Cr{c-mono-1}(T+\frac{\Cr{c-curva}^2}{16})^\frac{1}{4}}((1+5^{n-1})D_0). 
\end{equation}
By $s = \tilde{t} +r^2$, $r < \Cr{c-curva}/4$, $\eta = 1$ on $B_{\Cr{c-curva}/4}$ and \eqref{den-rat-contra}, we have 
\begin{equation}\label{den-rat4}
\begin{aligned} 
&\; e^{\Cr{c-mono-1}(s-\tilde{t})^\frac{1}{4}} \int_\Omega \rho_{1,(y,s)}(x,\tilde{t}) + \rho_{2,(y,s)}(x,\tilde{t}) \; d\mu_\varepsilon^{\tilde{t}}(x) \\
&\; \ge \int_{\Omega \cap B_r(y)} \dfrac{e^{-\frac{|x-y|^2}{4r^2}}}{(4\pi r^2)^\frac{n-1}{2}} \; d\mu_\varepsilon^{\tilde{t}} + \int_{\Omega \cap \tilde{B}_r(y)} \dfrac{e^{-\frac{|\tilde{x}-y|^2}{4r^2}}}{(4\pi r^2)^\frac{n-1}{2}} \; d\mu_\varepsilon^{\tilde{t}} \\
&\; \ge \dfrac{e^{-\frac{1}{4}}}{(4\pi)^\frac{n-1}{2}r^{n-1}}(\mu_\varepsilon^{\tilde{t}}(B_r(y)) + \mu_\varepsilon^{\tilde{t}}(\tilde{B}_r(y))) > \dfrac{e^{-\frac{1}{4}}}{(4\pi)^\frac{n-1}{2}}\Cr{c-den-rat}. 
\end{aligned}
\end{equation}
Now, we choose $0 < \Cr{e-den-rat} \le \Cr{e-dis-den}$ so that 
\[ \Cr{c-dis-mono-bdd} \varepsilon^{\lambda^\prime - \lambda}\left(1 + |\log \varepsilon| + \left(\log \left(T+\frac{\Cr{c-curva}^2}{16}\right)\right)^+\right) \le 1 \]
for $\varepsilon \in (0, \Cr{e-den-rat})$. 
Then, by combining \eqref{den-rat2}--\eqref{den-rat4} and the choice of $\Cr{c-den-rat}$, we obtain a contradiction for $\varepsilon \in (0, \Cr{e-den-rat})$. 
In the case of $y \not \in N_{\Cr{c-curva}/2}$, we may obtain a contradiction by similar computations as above. 
\end{proof}


\section{Vanishing of the discrepancy}\label{sec:vani-dis}

In the following, we define the Radon measure $\mu_{\varepsilon_i}$ and $|\xi_{\varepsilon_i}|$ on $\mathbb{R}^n \times [0,\infty)$ as 
\begin{align*}
d\mu_{\varepsilon_i} := d\mu_{\varepsilon_i}^t dt, \quad d|\xi_{\varepsilon_i}| := \left|\dfrac{\varepsilon_i|\nabla u_{\varepsilon_i}|}{2} - \dfrac{W(u_{\varepsilon_i})}{\varepsilon_i}\right| d\mathcal{L}^n\lfloor_\Omega dt. 
\end{align*} 
From the boundedness \eqref{energy-bdd}, we obtain subsequence limits $\mu$ and $|\xi|$ of $\mu_{\varepsilon_i}$ and $|\xi_{\varepsilon_i}|$ on $\mathbb{R}^n \times [0,\infty)$, respectively. 
Since $\mu_{\varepsilon_i}^t(\Omega)$ is bounded uniformly with respect to $\varepsilon_i$ and $t \in [0,\infty)$, the dominated convergence theorem shows $d\mu = d\mu^tdt$, where $\mu^t$ is the limit measure of $\mu^t_{\varepsilon_i}$ obtained by Proposition \ref{pro-con-den}. 
On the other hand, we note that ${\rm spt} \mu$ may not be the same as $\cup_{t \ge 0} {\rm spt}\mu^t \times \{t\}$. 
Though the following lemma can be proved as \cite[Lemma 5.1]{TT}, we include the proof for the convenience of the reader. 

\begin{lem}\label{lem-chara-spt}
For all $t \ge 0$, 
\[
{\rm spt} \mu^t \subset \{x \in \overline{\Omega} : (x,t) \in {\rm spt} \mu\}. 
\]
\end{lem}

\begin{proof}
Suppose $x \in {\rm spt} \mu^t$ and assume for a contradiction that $(x,t) \not \in {\rm spt} \mu$. 
Then there exists $r>0$ such that $\mu(B_r(x) \times (t-r^2, t+r^2)) = 0$. 
Take $\phi \in C^2_c(B_r(x) ; \mathbb{R}^+)$ with $\phi = 1$ on $B_{r/2}(x)$. 
Since $x \in {\rm spt} \mu^t$, we have $\mu^t(\phi) > 0$. 
By using integration by parts and the Neumann boundary condition \eqref{neumann}, we obtain 
\begin{equation}\label{chara-spt1}
\begin{aligned}
\dfrac{d}{dt} \mu_{\varepsilon_i}^t(\phi) =&\; \int_{\Omega} -\varepsilon_i \phi (\partial_t u_{\varepsilon_i})^2 - \varepsilon_i \langle \nabla \phi, \nabla u_{\varepsilon_i}\rangle \partial_t u_{\varepsilon_i} \; dy \\
=&\; \int_{\Omega \cap {\rm spt}\phi} -\varepsilon_i \phi \left(\partial_t u_{\varepsilon_i} + \dfrac{\langle \nabla \phi, \nabla u_{\varepsilon_i} \rangle}{2\phi} \right)^2 + \varepsilon_i \dfrac{(\langle \nabla \phi, \nabla u_{\varepsilon_i} \rangle)^2}{4\phi} \; dy \le c(\|\phi\|_{C^2}, \Cr{c-ene0}), 
\end{aligned}
\end{equation}
where $c(\|\phi\|_{C^2}, \Cr{c-ene0})$ is a constant depending only on $\|\phi\|_{C^2}$ and $\Cr{c-ene0}$. 
Here we have used $\frac{|\nabla \phi|^2}{2\phi} \le \|\phi\|_{C^2}$ and \eqref{energy-bdd}. 
Integrating \eqref{chara-spt1} over $[s,t]$ and taking the limit $i \to \infty$, we see that $\mu^s(\phi) > \mu^t(\phi)/2 > 0$ if $s$ is sufficient close to $t$. 
Thus, we obtain 
\[ \iint_{\overline{\Omega} \times (t-r^2, t+r^2)} \phi(x) \; d\mu(y,s) >0 \]
from $d\mu = d\mu^tdt$, which contradicts $\mu(B_r(x) \times (t-r^2, t+r^2)) = 0$. 
\end{proof}

In this section, our aim is to prove the vanishing of $|\xi|$. 

\begin{pro}\label{thm-vani-dis}
$|\xi| = 0$ on $\mathbb{R}^n \times (0,\infty)$. 
\end{pro}

The proof of the vanishing of $\xi$ in this paper is similar to that of \cite{MT}. 
However, Mizuno and Tonegawa \cite{MT} used the inequality $|x-\tilde{y}| \ge |x-y|$ for any two interior points $x$ and $y$ of convex domains (such that $\tilde{y}$ is well-defined) to control $\rho_{2,(x,s)}(y,t)$ as $\rho_{2,(x,s)}(y,t) \le \rho_{1,(x,s)}(y,t) = \rho_{1,(y,s)}(x,t)$. 
In this paper, we modify \cite[Lemma 3.4]{I} (and other related arguments) to include the reflection argument and apply an inequality between $|x-\tilde{y}|$ and $|\tilde{x} - y|$ to control $\rho_{2,(x,s)}(y,t)$ by $\rho_{2,(y,s)}(x,t)$. 

For all $t \ge 0$ and the limit measure $\mu^t$, we define 
\[ 
\overline{\mu}^t_{r,y} := 
\begin{cases}
\displaystyle \int_{\overline{\Omega}} \dfrac{\eta(|x-y|)e^{-\frac{|x-y|^2}{4r^2}} + \eta(|\tilde{x}-y|)e^{-\frac{|\tilde{x}-y|^2}{4r^2}}}{(2\sqrt{\pi} r)^{n-1}} \; d\mu^t(x) & \mbox{if} \quad y \in N_{\Cr{c-curva}/2}, \\
\displaystyle \int_{\overline{\Omega}} \dfrac{\eta(|x-y|)e^{-\frac{|x-y|^2}{4r^2}}}{(2\sqrt{\pi} r)^{n-1}} \; d\mu^t(x) & \mbox{if} \quad y \not\in N_{\Cr{c-curva}/2}.
\end{cases}
\]

\begin{lem}\label{lem-con-heat}
For any $T>0$ and $\delta > 0$, there exist $0< \Cl[c]{c-lim-rat1} <1$ and $\Cl[c]{c-lim-rat2}$ depending only on  $T$, $n$, $D_0$, $\alpha$, $W$, $\lambda$, $\kappa$, $\Cr{c-as-gra}$, $\Cr{c-as-dis}$, $\Cr{c-curva}$ and $\Omega$ with the following properties: 

(1) For $0 < r \le \Cr{c-curva}/2$, $y, y_0 \in \overline{\Omega}$ with $|y-y_0| \le \Cr{c-lim-rat1} r$ and $0 \le t \le T$, 
\[ \overline{\mu}^t_{r,y} \le \overline{\mu}^t_{r,y_0} + \delta. \]

(2) For $0<r, R$ with $1 \le R/r \le 1+\Cr{c-lim-rat2}$, $y \in \overline{\Omega}$ and $0 \le t \le T$, 
\[ \overline{\mu}^t_{R,y} \le \overline{\mu}^t_{r,y} + \delta. \]
\end{lem}

\begin{proof}
In order to prove (1), assume $|y-y_0| \le \Cr{c-lim-rat1} r$, where $\Cr{c-lim-rat1} \in (0,1)$ is a constant to be chosen later. 
First, we estimate 
\begin{equation}\label{lim-rati0} 
\int_{\overline{\Omega}} \dfrac{\eta(|\tilde{x}-y|) e^{-\frac{|\tilde{x}-y|^2}{4r^2}}}{(2\sqrt{\pi}r)^{n-1}} \; d\mu^t(x) 
\end{equation}
in the case of $y, y_0 \in N_{\Cr{c-curva}/2}$. 
For any $x \in N_{6\Cr{c-curva}}$, let 
\[ f(y) := \eta(|\tilde{x}-y|)e^{-\frac{|\tilde{x}-y|^2}{4r^2}}. \]
By the Taylor expansion, we obtain 
\begin{equation}\label{lim-rati1} 
\begin{aligned}
f(y) =&\; f(y_0) + e^{-\frac{|\tilde{x}-y^\prime|^2}{4r^2}} \left(\dfrac{\eta(|\tilde{x}-y^\prime|)}{2r^2}\langle y^\prime-\tilde{x}, y-y_0 \rangle + \eta^\prime(|\tilde{x}-y^\prime|) \left\langle \dfrac{y^\prime - \tilde{x}}{|y^\prime - \tilde{x}|}, y-y_0 \right\rangle \right) \\
\le&\; f(y_0) + e^{-\frac{|\tilde{x}-y^\prime|^2}{4r^2}} \left(\Cr{c-lim-rat1}\eta(|\tilde{x}-y^\prime|)\dfrac{|y^\prime-\tilde{x}|}{2r} + \Cr{c-lim-rat1}r\left|\eta^\prime(|\tilde{x} - y^\prime|)\right|\right), 
\end{aligned}
\end{equation}
where $y^\prime = \theta y + (1-\theta) y_0$ with some $\theta \in (0,1)$. 
From $|\tilde{x} - y^\prime|^2 \ge \frac{2}{3}|\tilde{x} - y_0|^2 - \Cr{c-lim-rat1}^2r^2$ and $se^{-\frac{s^2}{2}} \le c$ for some constant $c$ and any $0 \le s < \infty$, \eqref{lim-rati1} gives 
\begin{equation}\label{lim-rati2}
f(y) \le f(y_0) + c\Cr{c-lim-rat1}e^\frac{\Cr{c-lim-rat1}^2}{8} e^{-\frac{|\tilde{x}-y_0|^2}{12r^2}}\left(\eta(|\tilde{x}-y^\prime|) + r \left|\eta^\prime(|\tilde{x}-y^\prime|)\right|\right). 
\end{equation}
Since $\eta$ and $|\eta^\prime|$ are bounded, $|\tilde{x} - y_0| \le |\tilde{x} - y^\prime| + \theta|y-y_0| < \Cr{c-curva}$ if $|\tilde{x} - y^\prime| < \Cr{c-curva}/2$ and ${\rm spt} (\eta(|\tilde{\cdot} - y^\prime|)) \subset \tilde{B}_{\Cr{c-curva}/2}(y^\prime)$, \eqref{lim-rati2} gives 
\begin{equation}\label{lim-rati3}
\begin{aligned} 
\int_{\overline{\Omega}} \dfrac{\eta(|\tilde{x}-y|) e^{-\frac{|\tilde{x}-y|^2}{4r^2}}}{(2\sqrt{\pi}r)^{n-1}} \; d\mu^t(x) \le&\; \int_{\overline{\Omega}} \dfrac{\eta(|\tilde{x}-y_0|) e^{-\frac{|\tilde{x}-y_0|^2}{4r^2}}}{(2\sqrt{\pi}r)^{n-1}} \; d\mu^t(x) \\
&\; + c(\Cr{c-curva}) \Cr{c-lim-rat1}e^{\frac{\Cr{c-lim-rat1}^2}{8}} \int_{\overline{\Omega} \cap \tilde{B}_{\Cr{c-curva}}(y_0)} \dfrac{(1+r)e^{-\frac{|\tilde{x}-y_0|^2}{12r^2}}}{(2\sqrt{\pi}r)^{n-1}} \; d\mu^t(x), 
\end{aligned}
\end{equation}
where $c(\Cr{c-curva})$ is a positive constant depending only on $\Cr{c-curva}$. 
For the last integral of \eqref{lim-rati3}, by applying Proposition \ref{thm-den-rat}, $r \le \Cr{c-curva}/2$ and computations similar to \eqref{lo-bdd-xi4}, we obtain 
\begin{equation}\label{lim-rati4}
\int_{\overline{\Omega}} \dfrac{\eta(|\tilde{x}-y|) e^{-\frac{|\tilde{x}-y|^2}{4r^2}}}{(2\sqrt{\pi}r)^{n-1}} \; d\mu^t(x) \le \int_{\overline{\Omega}} \dfrac{\eta(|\tilde{x}-y_0|) e^{-\frac{|\tilde{x}-y_0|^2}{4r^2}}}{(2\sqrt{\pi}r)^{n-1}} \; d\mu^t(x) + c(\Cr{c-curva}, \Cr{c-den-rat})\Cr{c-lim-rat1}e^\frac{\Cr{c-lim-rat1}^2}{8} 
\end{equation}
for $y, y_0 \in \overline{\Omega} \cap N_{\Cr{c-curva}/2}$, where $c(\Cr{c-curva}, \Cr{c-den-rat})$ is a positive constant depending only on $\Cr{c-curva}$ and $\Cr{c-den-rat}$. 
By the similar argument as above, we obtain 
\begin{equation}\label{lim-rati5}
\int_{\overline{\Omega}} \dfrac{\eta(|x-y|) e^{-\frac{|x-y|^2}{4r^2}}}{(2\sqrt{\pi}r)^{n-1}} \; d\mu^t(x) \le \int_{\overline{\Omega}} \dfrac{\eta(|x-y_0|) e^{-\frac{|x-y_0|^2}{4r^2}}}{(2\sqrt{\pi}r)^{n-1}} \; d\mu^t(x) + c(\Cr{c-curva}, \Cr{c-den-rat})\Cr{c-lim-rat1}e^\frac{\Cr{c-lim-rat1}^2}{8} 
\end{equation}
for $y, y_0 \in \overline{\Omega}$. 
Since ${\rm spt} (\eta(|\tilde{\cdot} - y^\prime|)) \cap \overline{\Omega} = \emptyset$ if $y \not \in N_{\Cr{c-curva}/2}$, we can regard the integral \eqref{lim-rati0} as zero, and hence \eqref{lim-rati4} and \eqref{lim-rati5} imply the conclusion of (1) with an appropriate choice of $\Cr{c-lim-rat1}$. 

We may prove (2) by the similar argument by using Taylor expansion for $e^{-\frac{|x-y|^2}{4R^2}}$ with respect to $R$ around $r$ and applying the inequality $r \le R$ for the denominator of the integral function of $\overline{\mu}^t_{r,y}$. 
\end{proof}

The following lemma is needed when exchanging the center and the space variable of the reflected backward heat kernel $\rho_2$. 

\begin{lem}
(1) For $x \in N_{6\Cr{c-curva}}$ and $b \in \partial \Omega$, 
\begin{equation}\label{ref-po-1}
|\tilde{x} - b| \le \left(1+ \dfrac{2 \kappa |x-b|}{1- \kappa |x - b|} \right) |x-b|. 
\end{equation}

(2) For $x, y \in \overline{\Omega}$ with $|x - \tilde{y}| \le \Cr{c-curva}/2$ and $y \in N_{\Cr{c-curva}/2}$, 
\begin{equation}\label{ref-po-2}
|\tilde{x} - y| \le 2|x-\tilde{y}|, \quad |\tilde{x} - y| \le (1+12\kappa|\tilde{x} - y|)|x - \tilde{y}|, 
\end{equation}
where $\kappa$ is the constant defined by \eqref{def-kappa}.
\end{lem}

\begin{proof}
(1) is proved in \cite{GJ}, thus we refer to \cite{GJ} for the details. 
For $x, y \in \overline{\Omega}$ with $|x - \tilde{y}| \le \Cr{c-curva}/2$ and $y \in N_{\Cr{c-curva}/2}$, since $x \in \overline{\Omega}$ and $\tilde{y} \not \in \Omega$, we may fix a boundary point $b \in \partial \Omega$ such that 
\begin{equation}\label{ref1}
|x - \tilde{y}| = |x - b| + |b - \tilde{y}|. 
\end{equation}
By \eqref{ref1} and $|x - \tilde{y}| \le \Cr{c-curva}/2$, we obtain  
\begin{equation}\label{ref3}
|x - b|, |b - \tilde{y}| \le \dfrac{\Cr{c-curva}}{2}. 
\end{equation}
From $\Cr{c-curva} \in (0, (6\kappa)^{-1}]$, \eqref{ref-po-1} and \eqref{ref3} imply 
\begin{equation}\label{ref4}
|\tilde{x}-b| \le (1+ 3\kappa|x-b|) |x - b|, \quad |y - b| \le (1+3\kappa|\tilde{y}-b|)|\tilde{y} - b|. 
\end{equation}
Since $|\tilde{x} - y| \le |\tilde{x} - b| + |b - y|$, we obtain by \eqref{ref1}, \eqref{ref4} and $|x-b|, |\tilde{y} - b| \le |x-\tilde{y}|$
\begin{equation}\label{ref5} 
|\tilde{x} - y| \le |x-\tilde{y}| + 3\kappa(|x-b|^2 + |\tilde{y}-b|^2) \le |x-\tilde{y}| + 6\kappa|x-\tilde{y}|^2. 
\end{equation} 
Thus we obtain the first inequality of \eqref{ref-po-2} from $|x - \tilde{y}| \le (12\kappa)^{-1}$. 
We also note that the first inequality of \eqref{ref-po-2} and $|x-\tilde{y}| \le \Cr{c-curva}/2$ imply $|\tilde{x} - y| \le \Cr{c-curva}$, thus we may obtain 
\begin{equation}\label{ref2}
|x-\tilde{y}| \le 2|\tilde{x}-y|
\end{equation}
by the similar argument using a boundary point $b^\prime \in \partial \Omega$ such that $|\tilde{x} - y| = |\tilde{x} - b^\prime| + |b^\prime - y| $ instead of $b$. 
By combining \eqref{ref5} and \eqref{ref2}, we have the second inequality of \eqref{ref-po-2}. 
\end{proof}

The statement of the following lemma is exactly the same as \cite[Lemma 6.1]{MT} without the convexity assumption on $\Omega$. 
However they used the convexity of $\Omega$ at some technical points, thus we give a new proof for the following lemma. 

\begin{lem}\label{sptmu1}
For any $(y,s) \in {\rm spt} \mu$ with $y \in \overline{\Omega}$ and $s > 0$, there exists a sequence $\{x_i, t_i\}_{i=1}^\infty$ and a subsequence $\varepsilon_i$ (denoted by the same index) such that $t_i > 0$, $x_i \in \Omega$, $(x_i,t_i) \to (y,s)$ as $i \to \infty$ and $|u_{\varepsilon_i}(x_i,t_i)| < \alpha$ for all $i \in \mathbb{N}$. 
\end{lem}

\begin{rk}\label{rk:concent}
From Lemma \ref{sptmu1} and Lemma \ref{lem-chara-spt}, we can see that ${\rm spt} \mu^t \cap \{x : (x,t) \in A\}$ is empty if $|u_{\varepsilon_i}| \ge \alpha$ for any sufficiently small $\varepsilon_i$ in an open set $A \subset \overline{\Omega} \times [0,\infty)$. 
Roughly speaking, this means that the diffused interface energy concentrate on the domain $\{|u_{\varepsilon_i}| < \alpha\}$. 
\end{rk}

\begin{proof}[Proof of Lemma \ref{sptmu1}]
For simplicity we omit the subscript $i$. 
For a contradiction, assume that there exists $0 < r_0 < \sqrt{s}$ such that 
\begin{equation}\label{as-alpha} 
\inf_{(B_{r_0}(y) \cap \Omega) \times (s-r_0^2, s+r_0^2)} |u_\varepsilon| \ge \alpha 
\end{equation}
for all sufficiently small $\varepsilon > 0$. 
Fix $\phi \in C^1_c (B_{r_0}(y))$ such that 
\[ |\nabla \phi| \le \dfrac{3}{r_0}, \quad \phi \equiv 1 \quad \mbox{on} \quad B_{r_0/2}(y). \]
Multiplying \eqref{ac} by $\varepsilon \phi^2 W^\prime(u_\varepsilon)$, integrating on $\Omega$, integrating by parts and applying the Neumann boundary condition \eqref{neumann}, the assumptions (W3) and \eqref{as-alpha} imply 
\begin{equation}\label{mu-spt1}
\begin{aligned}
\varepsilon \dfrac{d}{dt} \int_\Omega \phi^2 W(u_\varepsilon) \; dx =&\; \int_\Omega \varepsilon \phi^2 W^\prime(u_\varepsilon)\Delta u_\varepsilon - \dfrac{(W^\prime(u_\varepsilon))^2}{\varepsilon} \; dx \\ 
\le &\; - \int_\Omega \varepsilon \beta \phi^2 |\nabla u_\varepsilon|^2 + \varepsilon 2 \phi W^\prime(u_\varepsilon) \langle \nabla u_\varepsilon, \nabla \phi \rangle + \dfrac{(W^\prime(u_\varepsilon))^2}{\varepsilon} \; dx
\end{aligned}
\end{equation}
for $s-r_0^2 < t < s+r_0^2$. 
By applying the Young inequality and rearranging terms, \eqref{mu-spt1} implies 
\begin{align*}
\int_\Omega \phi^2 \left(\varepsilon \beta |\nabla u_\varepsilon|^2 + \dfrac{(W^\prime(u_\varepsilon))^2}{2\varepsilon} \right) \; dx 
\le 2 \varepsilon^3 \int_\Omega |\nabla \phi|^2 |\nabla u_\varepsilon|^2 \; dx - \varepsilon \dfrac{d}{dt} \int_\Omega \phi^2 W(u_\varepsilon) \; dx 
\end{align*}
for $s-r_0^2 < t < s+r_0^2$. 
Integrating from $s-r_0^2$ to $s+r_0^2$ with respect to $t$, we have by the boundedness \eqref{u-bdd} and \eqref{energy-bdd} 
\begin{equation}\label{mu-spt2} 
\int_{s-r_0^2}^{s+r_0^2} \int_{B_{r_0/2}(y)} \varepsilon |\nabla u_\varepsilon|^2 + \dfrac{(W^\prime(u_\varepsilon))^2}{\varepsilon} \; dx dt \to 0 \quad \mbox{as} \; \varepsilon \to 0. 
\end{equation}
By the continuity of $u_\varepsilon$ and \eqref{as-alpha}, we may assume $\alpha \le u_\varepsilon \le 1$ on $(B_{r_0}(y) \cap \Omega) \times (s-r_0^2, s+r_0^2)$ without loss of generality. 
Otherwise we have $-1 \le u_\varepsilon \le -\alpha$ and we may argue similarly. 
From the assumption (W1), there exists a positive constant $c(W)$ such that $W(s) \le c(W)(s-1)^2$ for all $s \in [\alpha, 1]$. 
Furthermore, the assumptions (W1) and (W3) imply $W^\prime(s) = W^\prime(s) - W^\prime(1) \le \beta(s-1) \le 0$ for all $s \in [\alpha, 1]$. 
Thus, the inequality 
\begin{equation}\label{mu-spt3} 
\int_{s-r_0^2}^{s+r_0^2} \int_{B_{r_0/2}(y)} \dfrac{W(u_\varepsilon)}{\varepsilon} \; dx dt \le c(W,\beta)\int_{s-r_0^2}^{s+r_0^2} \int_{B_{r_0/2}(y)} \dfrac{(W^\prime(u_\varepsilon))^2}{\varepsilon} \; dx dt 
\end{equation}
holds for some positive constant $c(W,\beta)$. 
Hence we conclude by combining \eqref{mu-spt2} and \eqref{mu-spt3}  
\[ \mu(B_{r_0/2}(y) \times (s-r_0^2, s+r_0^2)) = 0, \]
which contradicts $(y,s) \in {\rm spt} \mu$. 
\end{proof}

\begin{lem}\label{lem-con-spt}
For any $T>0$, there exist $\delta_0, r_1, \Cl[c]{c-con-spt} >0$ depending only on $T$, $n$, $D_0$, $\alpha$, $W$, $\lambda$, $\kappa$, $\Cr{c-as-gra}$, $\Cr{c-as-dis}$, $\Cr{c-curva}$ and $\Omega$ such that the following holds: 
For $0 < t < s < \min\{t+r_1^2, T\}$ and $y \in \overline{\Omega}$, assume 
\begin{equation}\label{con-spt-conclusion1} 
\overline{\mu}_{r,y}^s < \delta_0, 
\end{equation}
where $r = \sqrt{s-t}$. 
Then $(y^\prime,t^\prime) \not\in {\rm spt} \mu$ for all $y^\prime \in B_{\Cr{c-con-spt} r}(y) \cap \overline{\Omega}$, where $t^\prime = 2s-t$. 
\end{lem}

\begin{proof}
First, we argue in the case of $y^\prime \in N_{\Cr{c-curva}/2}$. 
Let us assume $(y^\prime, t^\prime) \in {\rm spt} \mu$ for a contradiction. 
From Lemma \ref{sptmu1}, there exists a sequence $\{y_i, t_i\}_{i=1}^\infty$ such that $(y_i,t_i) \to (y^\prime, t^\prime)$ as $i \to \infty$ and $|u_{\varepsilon_i}(y_i,t_i)| < \alpha$ for all $i \in \mathbb{N}$. 
Note that $y_i \in N_{\Cr{c-curva}/2}$ for sufficiently large $i$. 
Put $r_i := \gamma_0 \varepsilon_i$ and $T_i := t_i + r_i^2$, where $\gamma_0 > 0$ is the constant satisfying \eqref{gamma-0} with $y=y_i$. 
By the similar argument for \eqref{lo-bdd-xi0}, we obtain
\begin{equation}\label{con-spt1}
\int_{B_{r_i}(y_i)} \rho_{1,(y_i,T_i)}(x,t_i) \; d\mu_{\varepsilon_i}^{t_i}(x) \ge \Cl[m]{m-con-spt1}, 
\end{equation}
where $\Cr{m-con-spt1}$ is a constant depending only on $\alpha, W, \Omega$ and $\Cr{c-bdd-nabla}$. 
Substituting $y_i$ and $T_i$ for $y$ and $s$ in  \eqref{monotonicity}, respectively, integrating the substituted inequality over $t \in (s,t_i)$ and applying Lemma \ref{lem-dis-mono}, we obtain by \eqref{con-spt1} 
\begin{equation*}
\begin{aligned}
\Cr{m-con-spt1} \le&\; e^{\Cr{c-mono-1}(T_i - s)^\frac{1}{4}} \int_\Omega \rho_{1,(y_i,T_i)}(x,s) + \rho_{2,(y_i,T_i)}(x,s) \; d\mu_{\varepsilon^i}^s \\
&\; + e^{\Cr{c-mono-1}(T_i - s)^\frac{1}{4}}\left(\Cr{c-mono-2}(T_i-s) + \Cr{c-dis-mono-bdd}\varepsilon^{\lambda^\prime - \lambda}(1+|\log \varepsilon_i| + (\log T_i))\right)
\end{aligned}
\end{equation*}
for sufficiently small $\varepsilon_i$. 
Letting $i \to \infty$, we have 
\begin{equation}\label{con-spt2}
\Cr{m-con-spt1} \le e^{\Cr{c-mono-1}(t^\prime - s)^\frac{1}{4}} \int_\Omega \rho_{1,(y^\prime,t^\prime)}(x,s) + \rho_{2,(y^\prime,t^\prime)}(x,s) \; d\mu^s + e^{\Cr{c-mono-1}(t^\prime - s)^\frac{1}{4}}\Cr{c-mono-2}(t^\prime-s) 
\end{equation}
Since $t^\prime - s = s-t = r^2$, \eqref{con-spt2} is equivalent to 
\begin{equation}\label{con-spt3}
\Cr{m-con-spt1} \le e^{\Cr{c-mono-1}r^\frac{1}{2}} \overline{\mu}^s_{r,y^\prime} + e^{\Cr{c-mono-1}r^\frac{1}{2}}\Cr{c-mono-2}r^2.  
\end{equation}
Now, we choose sufficiently small $r_1 \in (0, \Cr{c-curva}/2)$ such that $s-t = r^2 < r_1^2$ implies 
\begin{equation}\label{con-spt4}
e^{\Cr{c-mono-1}r^\frac{1}{2}} \le 2, \quad e^{\Cr{c-mono-1}r^\frac{1}{2}}\Cr{c-mono-2}r^2 \le \dfrac{\Cr{m-con-spt1}}{2}. 
\end{equation}
Furthermore, by setting $\Cr{c-con-spt} = \Cr{c-lim-rat1}$, where $\Cr{c-lim-rat1}$ is in Lemma \ref{lem-con-heat} with $\delta = \Cr{m-con-spt1}/8$, \eqref{con-spt3}, \eqref{con-spt4} and Lemma \ref{lem-con-heat} imply 
\begin{equation}\label{con-spt5}
\dfrac{\Cr{m-con-spt1}}{8} \le \overline{\mu}^s_{r,y}. 
\end{equation}
Here $s \le T$ is used. 
Letting $\delta_0 < \Cr{m-con-spt1}/8$, we have a contradiction from \eqref{con-spt5} and \eqref{con-spt-conclusion1}. 
In the other cases, $y^\prime \not\in N_{\Cr{c-curva}/2}$, we may obtain a contradiction as above with the same constants $\delta_0, r_0$ and $\Cr{c-con-spt}$. 
\end{proof}

\begin{cor}\label{cor-bdd-spt}
For $0 \le t \le T$, there exists $\Cl[c]{c-bdd-spt}$ depending only on $T, n, D_0, \alpha, W, \lambda, \kappa, \Cr{c-as-gra}, \Cr{c-as-dis}, \Cr{c-ene0}, \Cr{c-curva}, \Cr{c-bdd-nabla}$ and $\Omega$ such that 
\[
\mathcal{H}^{n-1}({\rm spt}\mu_t) \le \Cr{c-bdd-spt}. 
\]
\end{cor}

\begin{proof}
For any $(y,t) \in {\rm spt}\mu$ such that $y \in N_{\Cr{c-curva}/2}$ and $t \in [0,T]$, we obtain by the similar argument for \eqref{con-spt3} 
\begin{equation}\label{bdd-spt1}
\dfrac{\Cr{m-con-spt1}}{4} \le \int_{\overline{\Omega}} \rho_{1,(y,t+r^2)}(x,t) + \rho_{2,(y,t+r^2)}(x,t) \; d\mu^{t+r^2}(x)
\end{equation}
for any $r \in (0,r_1)$, where $r_1$ is a constant given in Lemma \ref{lem-con-spt} and $\Cr{m-con-spt1}$ is a constant depending only on $\alpha, W, \Omega$ and $\Cr{c-bdd-nabla}$. 
For $0 < L \le \Cr{c-curva}/(2r)$, using Proposition \ref{thm-den-rat} and the similar argument for \eqref{lo-bdd-xi4}, we have 
\begin{equation}\label{bdd-spt2}
\begin{aligned}
&\; \int_{\overline{\Omega} \setminus B_{5Lr}(y)} \rho_{1,(y,t+r^2)}(x,t) \; d\mu^{t+r^2}(x) + \int_{\overline{\Omega} \setminus \tilde{B}_{Lr}(y)} \rho_{2(y,t+r^2)}(x,t) \; d\mu^{t+r^2}(x) \\
&\; \le \dfrac{\Cr{c-den-rat} \omega_{n-1}}{\pi^{\frac{n-1}{2}}} \left(\int_{\frac{25L^2}{4}}^\infty a^{\frac{n-1}{2}}e^{-a} \; da + \int_{\frac{L^2}{4}}^\infty a^{\frac{n-1}{2}}e^{-a} \; da\right). 
\end{aligned}
\end{equation}
Thus by choosing sufficiently large $L$ depending only on $T, n, D_0, \alpha, W, \lambda, \kappa, \Cr{c-as-gra}, \Cr{c-as-dis}, \Cr{c-curva}, \Cr{c-bdd-nabla}$ and $\Omega$, \eqref{bdd-spt1} and \eqref{bdd-spt2} show
\[
\dfrac{\Cr{m-con-spt1}}{8} \le \int_{\overline{\Omega} \cap B_{5Lr}(y)} \rho_{1,(y,t+r^2)}(x,t) \; d\mu^{t+r^2}(x) + \int_{\overline{\Omega} \cap \tilde{B}_{Lr}(y)} \rho_{2(y,t+r^2)}(x,t) \; d\mu^{t+r^2}(x). 
\]
From $\rho_{i,(y,t+r^2)(\cdot,t)} \le (4\pi)^{-(n-1)/2}r^{-(n-1)}$ for $i=1,2$ and Lemma \ref{lem-ref-ball}, we obtain 
\begin{equation}\label{bdd-spt3}
\dfrac{(4\pi)^{\frac{n-1}{2}}\Cr{m-con-spt1}}{8}r^{n-1} \le 2\mu^{t+r^2}(B_{5Lr}(y)). 
\end{equation}
In the case of $y \not\in N_{\Cr{c-curva}/2}$, we have \eqref{bdd-spt3} by the similar argument. 
Now we fix $t \in [0,T]$. 
Let $\mathcal{B} = \{\overline{B}_{5Lr}(y) : (y,t) \in {\rm spt}\mu \}$ which is a covering of $\{y \in \overline{\Omega} : (y,t) \in {\rm spt} \mu\}$. 
By the Besicovitch covering theorem, there exist a finite sub-collection $\mathcal{B}_1, \cdots, \mathcal{B}_{B(n)}$ such that each $\mathcal{B}_i$ is a pairwise disjoint family of closed balls and 
\begin{equation}\label{bdd-spt4}
\{y \in \overline{\Omega} : (y,t) \in {\rm spt} \mu\} \subset \cup_{i=1}^{B(n)} \cup_{\overline{B}_{5Lr}(y_j) \in \mathcal{B}_i} \overline{B}_{5Lr}(y_j). 
\end{equation}
Let $\mathcal{H}^{n-1}_\delta$ be defined as in \cite{S}, so that $\mathcal{H}^{n-1} = \lim_{\delta \downarrow 0} \mathcal{H}^{n-1}_\delta$. 
By the definition, \eqref{energy-bdd}, \eqref{bdd-spt3} and \eqref{bdd-spt4} we obtain 
\[ 
\begin{aligned}
\mathcal{H}^{n-1}_{10Lr}(\{y \in \overline{\Omega} : (y,t) \in {\rm spt} \mu\}) \le&\; \sum_{i=1}^{B(n)} \sum_{\overline{B}_{5Lr}(y_j) \in \mathcal{B}_i} \omega_{n-1} (5Lr)^{n-1} \\
\le&\; \sum_{i=1}^{B(n)} \dfrac{16 \omega_{n-1} (5L)^{n-1}}{(4\pi)^{\frac{n-1}{2}}\Cr{m-con-spt1}} \sum_{\overline{B}_{5Lr}(y_j) \in \mathcal{B}_i} \mu^{t+r^2}(B_{5Lr}(y_j)) \\
\le&\; \sum_{i=1}^{B(n)} \dfrac{16 \omega_{n-1} (5L)^{n-1}}{(4\pi)^{\frac{n-1}{2}}\Cr{m-con-spt1}} \mu^{t+r^2}(\overline{\Omega}) \le \Cr{c-bdd-spt}, 
\end{aligned}
\]
where $\Cr{c-bdd-spt}$ is a constant depending only on $T, n, D_0, \alpha, W, \lambda, \kappa, \Cr{c-as-gra}, \Cr{c-as-dis}, \Cr{c-ene0}, \Cr{c-curva}, \Cr{c-bdd-nabla}$ and $\Omega$. 
Letting $r \downarrow 0$, we obtain the boundedness of $\mathcal{H}^{n-1}(\{y \in \overline{\Omega} : (y,t) \in {\rm spt} \mu\})$ and hence Lemma \ref{lem-chara-spt} implies the conclusion. 
\end{proof}

\begin{lem}\label{lem-vani-spt}
For $T>0$, let $\delta_0(T)$ be the constant given in Lemma \ref{lem-con-spt}. 
Then $\mu(Z^-(T))=0$, where 
\[ Z^-(T) = \left\{(y,t) \in {\rm spt}\mu : \limsup_{s \downarrow t} \overline{\mu}^s_{\sqrt{s-t}, y} < \delta_0(T), \quad 0 < t < T \right\}. \]
\end{lem}

\begin{rk}
We consider the meaning of $\overline{\mu}^s_{\sqrt{s-t}, y}$ in a simple case. 
Let $\mu^t$ be described as $\mu^t = \mathcal{H}^{n-1}\lfloor_{M_t}$ for a smooth and proper $(n-1)$-dimensional mean curvature flow $\{M_t\}$ in $\Omega$ with the right angle condition. 
For an interior point $(y,t')$ of $\mathcal{M} := \cup_{t} M_t \times \{t\}$, we consider the re-scaling operator $\mathcal{D}_r: (x,t) \mapsto (r^{-1}(x-y), r^{-2}(t-t'))$ for $r>0$. 
Then we can see that $\mathcal{D}_r\mathcal{M}$ converges to the tangent flow $\mathcal{M}'$ of $\mathcal{M}$ at $(y,t')$ and $\overline{\mu}^{t'+r^2}_{r,y}$ converges to the integration 
\[ \int_{M'_1} \dfrac{1}{(4\pi)^{\frac{n-1}{2}}} e^{-\frac{|z|^2}{4}} \; d\mathcal{H}^{n-1}(z) \]
as $r \downarrow 0$, where $\mathcal{M}' = \cup_{\tau \in (-\infty,\infty)} M'_\tau \times \{\tau\}$. 
In this case, $M'_\tau$ is identically the tangent space of $M_{t'}$ at $y$, thus the Gaussian density is equal to $1$ and coincides with the Gaussian density of $\mathcal{M}'$. 
We also note that some properties of tangent flows and Gaussian densities are studied by Edelen \cite{E} for Brakke flows with a generalized right angle condition and by White \cite{W} for classical mean curvature flows without boundary conditions. 
\end{rk}

\begin{proof}[Proof of Lemma \ref{lem-vani-spt}]
We do not write out the dependence on $T$ in the following for simplicity. 
Corresponding to $T$, let $\delta_0, r_1$ and $\Cr{c-con-spt}$ be constants given in Lemma \ref{lem-con-spt}. 
For $0 < \tau < r_1^2$ define 
\[ Z^\tau := \left\{ (y,t) \in {\rm spt} \mu : \overline{\mu}^s_{\sqrt{s-t},y} < \delta_0 \quad \mbox{for} \quad 0 < t<s<\min\{t+\tau, T\} \right\}. \]
If we take a sequence $\tau_m > 0$ with $\lim_{m \to \infty} \tau_m = 0$, then $Z^- \subset \cup_{m=1}^\infty Z^{\tau_m}$. 
Hence we only need to show $\mu(Z^\tau) = 0$. 

Let $(y,t) \in Z^\tau$ be fixed and we define 
\[ P(y,t) := \{(y^\prime,t^\prime) \in \overline{\Omega} \times [0,T) : 2\Cr{c-con-spt}^{-2}|y^\prime-y|^2 < |t^\prime -t| < 2\tau \}. \]
We claim that $P(y,t) \cap Z^\tau = \emptyset$. 
Indeed, suppose for a contradiction that $(y^\prime, t^\prime) \in P(x,t) \cap Z^\tau$. 
Assume $t^\prime > t$ and put $s=(t+t^\prime)/2$. 
Then $s<T, t < s < t+\tau, |y-y^\prime| < \Cr{c-con-spt}\sqrt{(t^\prime-t)/2} = \Cr{c-con-spt} \sqrt{s-t}$ and $\overline{\mu}^s_{\sqrt{s-t},y} < \delta_0$. 
Hence by Lemma \ref{lem-con-spt}, $(y^\prime.t^\prime) \not \in {\rm spt}\mu$, which contradicts $(y^\prime,t^\prime) \in Z^\tau$. 
If $t^\prime < t$, by the similar argument, we obtain $(y,t) \not \in {\rm spt} \mu$ which is a contradiction. 
This proves $P(y,t) \cap Z^\tau = \emptyset$. 

For a fixed $(y_0,t_0) \in \overline{\Omega} \times [0, T)$, define 
\[ Z^{r,y_0,t_0} := Z^\tau \cap \left( B_{\frac{\Cr{c-con-spt}}{2}\sqrt{\tau}}(y_0) \times (t_0-\tau, t_0+\tau) \right). \]
Then $Z^\tau$ is a countable union of $Z^{\tau, y_m, t_m}$ with $(y_m, t_m)$ spaced appropriately. 
Hence we only need to show that $\mu(Z^{r,y_0,t_0})=0$. 
For $0<\rho \le \Cr{c-curva}$, we may find a covering of $\pi_\Omega (Z^{r,y_0,t_0}) := \{y \in \overline{\Omega} : (y,t) \in Z^{r,y_0,t_0}\}$ by a collection of balls $\{B_{r_i}(y_i)\}_{i=1}^\infty$, where $(y_i,t_i) \in Z^{r,y_0,t_0}$, $r_i \le \rho$ so that 
\begin{equation}\label{measure-zero1}
\sum_{i=1}^\infty \omega_n r_i^n \le c(n) \mathcal{L}^n(B_{\frac{\Cr{c-con-spt}}{2}\sqrt{\tau}}(x_0)). 
\end{equation}
For such a covering, we find 
\begin{equation} \label{measure-zero2}
Z^{r,y_0,t_0} \subset \cup_{i=1}^\infty B_{r_i}(y_i) \times (t_i-2r_i^2\Cr{c-con-spt}^{-2}, t_i+2r_i^2\Cr{c-con-spt}^{-2}). 
\end{equation}
Indeed, if $(y,t) \in Z^{r,y_0,t_0}$, then $y \in B_{r_i}(y_i)$ for some $i \in \mathbb{N}$. 
Since $P(y,t) \cap Z^\tau = \emptyset$, we have 
\[ |t-t_i| \le 2|x-x_i|^2\Cr{c-con-spt}^{-2} < 2r_i^2\Cr{c-con-spt}^{-2}. \]
Combining Proposition \ref{thm-den-rat}, \eqref{measure-zero1}, \eqref{measure-zero2} and $r_i \le \rho \le \Cr{c-curva}$, we obtain
\begin{align*}
\mu(Z^{r,y_0,t_0}) \le&\; \sum_{i=1}^\infty \mu(B_{r_i}(y_i) \times (t_i-2r_i^2\Cr{c-con-spt}^{-2}, t_i+2r_i^2\Cr{c-con-spt}^{-2})) \le \sum_{i=1}^\infty \Cr{c-den-rat} \omega_{n-1} r_i^{n-1} \cdot 4 \Cr{c-con-spt}^{-2} r_i^2 \\
\le &\; 4 \rho \Cr{c-den-rat} \Cr{c-con-spt}^{-2} \omega_{n-1} \omega_n^{-1} c(n) \mathcal{L}^n(B_{\frac{\Cr{c-con-spt}}{2}\sqrt{\tau}}(x_0)). 
\end{align*}
Since $0<\rho<\Cr{c-curva}$ is arbitrary, we have $\mu(Z^{r,y_0,t_0})=0$. 
This concludes the proof. 
\end{proof}

\begin{proof}[Proof of Proposition \ref{thm-vani-dis}]
It is enough to prove $|\xi| = 0$ on $\mathbb{R}^n \times (0, T)$ for all $0<T$. 
In the following we fix $T$. 
Note that ${\rm spt} |\xi| \subset \overline{\Omega} \times [0, \infty)$ by the definition of $|\xi_{\varepsilon_i}|$. 
For $y \in N_{\Cr{c-curva}/2} \cap \overline{\Omega}$ and $0 \le t < s < T$, since \eqref{ref-po-2} holds and $\eta$ is a monotone decreasing function, we have  
\begin{equation}\label{dis-vari1}
\begin{aligned}
&\; \iint_{\Omega\times(0,s)} \dfrac{\eta(|x-y|)e^{-\frac{|x-y|^2}{4(s-t)}} + \eta(|x-\tilde{y}|)e^{-\frac{|x-\tilde{y}|^2}{4(s-t)}}}{2^n \pi^\frac{n-1}{2}(s-t)^\frac{n+1}{2}} \; d|\xi_{\varepsilon_i}|(x,t) \\
\le &\; \iint_{\Omega\times(0,s)} \dfrac{\eta\left(|x-y|\right)e^{-\frac{|x-y|^2}{4(s-t)}} + \eta\left(\frac{|\tilde{x}-y|}{2}\right)e^{-\frac{|\tilde{x}-y|^2}{(1+12\kappa|\tilde{x}-y|)^2 4(s-t)}}}{2^n \pi^\frac{n-1}{2}(s-t)^\frac{n+1}{2}} \; d|\xi_{\varepsilon_i}|(x,t) \\
=&\; \iint_{\Omega\times(0,s)} \dfrac{\eta\left(|x-y|\right)e^{-\frac{|x-y|^2}{4(s-t)}} + \eta(|\tilde{x}-y|)e^{-\frac{|\tilde{x}-y|^2}{4(s-t)}}}{2^n \pi^\frac{n-1}{2}(s-t)^\frac{n+1}{2}} \; d|\xi_{\varepsilon_i}|(x,t) \\
&\; + \iint_{\Omega\times(0,s)} \dfrac{\left\{\eta\left(\frac{|\tilde{x}-y|}{2}\right) - \eta\left(|\tilde{x}-y|\right)\right\}e^{-\frac{|\tilde{x}-y|^2}{(1+12\kappa|\tilde{x}-y|)^2 4(s-t)}}}{2^n \pi^\frac{n-1}{2}(s-t)^\frac{n+1}{2}} \; d|\xi_{\varepsilon_i}|(x,t) \\
&\; + \iint_{\Omega\times(0,s)} \dfrac{\eta\left(|\tilde{x}-y|\right)\left(e^{-\frac{|\tilde{x}-y|^2}{(1+12\kappa|\tilde{x}-y|)^2 4(s-t)}}-e^{-\frac{|\tilde{x}-y|^2}{4(s-t)}}\right)}{2^n \pi^\frac{n-1}{2}(s-t)^\frac{n+1}{2}} \; d|\xi_{\varepsilon_i}|(x,t) =: I_1 + I_2 + I_3.
\end{aligned}
\end{equation}
We estimate each integration $I_i (i=1, 2, 3)$ on the right hand side of \eqref{dis-vari1}.  
By integrating \eqref{monotonicity} over $t \in (0,s)$, we obtain by \eqref{pro-d0}, \eqref{dis-mono-conclusion1} and $s < T$ 
\begin{equation}\label{dis-vari2}
\begin{aligned}
I_1 = &\; \int^{s}_0 \int_\Omega \dfrac{\rho_{1,(y,s)}(x,t) + \rho_{2,(y,s)}(x,t)}{2(s-t)} \left|\dfrac{\varepsilon_i |\nabla u_{\varepsilon_i}|^2}{2} - \dfrac{W(u_{\varepsilon_i})}{\varepsilon_i}\right|\; dx dt \\
\le &\; T e^{\Cr{c-mono-1} T^\frac{1}{4}} + e^{\Cr{c-mono-1} T^\frac{1}{4}} (1+5^{n-1})D_0 + 2 \Cr{c-dis-mono-bdd} \varepsilon_i^{\lambda^\prime - \lambda}(1 + |\log \varepsilon_i| + (\log T)^+). 
\end{aligned}
\end{equation}
From the definition of $\eta$, $\Cr{c-curva} \le (6\kappa)^{-1}$, $|\xi_{\varepsilon_i}| < \mu_{\varepsilon_i}$, \eqref{energy-bdd} and $s < T$, we obtain 
\begin{equation}\label{dis-vari3}
\begin{aligned}
I_2 \le&\; \iint_{\{x \in \Omega: |\tilde{x} - y| \ge \frac{\Cr{c-curva}}{4}\} \times (0,s)} \dfrac{\left|\eta\left(\frac{|\tilde{x} - y|}{2}\right) - \eta(|\tilde{x}-y|)\right|e^{-\frac{\Cr{c-curva}^2}{8^2} \cdot \frac{1}{4(s-t)}}}{2^n \pi^{\frac{n-1}{2}}(s-t)^{\frac{n+1}{2}}} \; d|\xi_{\varepsilon_i}|(x,t) \\
\le&\; c \mu_{\varepsilon_i}(\Omega \times (0,s)) \le c\Cr{c-ene0} T, 
\end{aligned}
\end{equation}
where $c$ is a constant satisfying $2^{-n} \pi^{-\frac{n-1}{2}} l^{-\frac{n+1}{2}} e^{-\frac{\Cr{c-curva}^2}{8^2}} \le c$ for $l \in (0,\infty)$. 
By the Taylor expansion for $e^{-\frac{|\tilde{x}-y|^2}{(1+12\kappa r)^24(s-t)}}$ with respect to $r$ around $0$, we have 
\begin{equation}\label{dis-vari4}
e^{-\frac{|\tilde{x}-y|^2}{(1+12\kappa|\tilde{x}-y|)^2 4(s-t)}}-e^{-\frac{|\tilde{x}-y|^2}{4(s-t)}} = \dfrac{6\kappa|\tilde{x}-y|^3}{(1+12\theta\kappa|\tilde{x}-y|)^3(s-t)}e^{-\frac{|\tilde{x}-y|^2}{(1+12\theta\kappa|\tilde{x}-y|)^2 4(s-t)}}
\end{equation}
for some constant $\theta \in (0,1)$. 
Applying \eqref{dis-vari4}, ${\rm spt}\eta  \subset [0,\Cr{c-curva}/2) \subset [0,1/(12\kappa))$, $|\xi_{\varepsilon_i}| < \mu_{\varepsilon_i}$, the uniform boundedness of $r^3 e^{-\frac{r^2}{32}}$ with respect to $r \in [0,\infty)$, Proposition \ref{thm-den-rat} and computations similar to \eqref{lo-bdd-xi4}, we obtain 
\begin{equation}\label{dis-vari5}
\begin{aligned}
I_3 \le&\; \iint_{\Omega \times (0,s)} \dfrac{3\kappa|\tilde{x}-y|^3\eta(|\tilde{x}-y|)}{2^{n-1}\pi^{\frac{n-1}{2}}(s-t)^{\frac{n+3}{2}}}e^{-\frac{|\tilde{x}-y|^2}{16(s-t)}} \; d|\xi_{\varepsilon_i}|(x,t) \\
\le&\; \int^s_0\int_{\Omega}\dfrac{c(n,\kappa)\eta(|\tilde{x}-y|)}{(s-t)^{\frac{n}{2}}}e^{-\frac{|\tilde{x}-y|^2}{32(s-t)}} d\mu^t_{\varepsilon_i}(x)dt \le \int_0^s \dfrac{c(n,\kappa, \Cr{c-den-rat})}{(s-t)^\frac{1}{2}} \; dt \le c(n,\kappa, \Cr{c-den-rat}, T), 
\end{aligned}
\end{equation}
where $c(n,\kappa, \Cr{c-den-rat}, T)$ is a positive constant depending only on $n,\kappa, \Cr{c-den-rat}$ and $T$. 
Combining \eqref{dis-vari1}--\eqref{dis-vari3} and \eqref{dis-vari5}, we have 
\begin{equation}\label{dis-vari6}
\begin{aligned}
&\; \int^s_0 \int_\Omega \dfrac{\eta(|x-y|)e^{-\frac{|x-y|^2}{4(s-t)}} + \eta(|x-\tilde{y}|)e^{-\frac{|x-\tilde{y}|^2}{4(s-t)}}}{2^n \pi^\frac{n-1}{2}(s-t)^\frac{n+1}{2}} \left|\dfrac{\varepsilon_i |\nabla u_{\varepsilon_i}|^2}{2} - \dfrac{W(u_{\varepsilon_i})}{\varepsilon_i}\right|\; dx dt \\
&\; \le c(n,\kappa,\Cr{c-ene0},\Cr{c-mono-1},\Cr{c-den-rat},\Cr{c-dis-mono-bdd},T,D_0)\left( 1+ \varepsilon_i^{\lambda^\prime - \lambda}(1 + |\log \varepsilon_i| + (\log 4T)^+) \right), 
\end{aligned}
\end{equation}
where $c(n,\kappa,\Cr{c-ene0},\Cr{c-mono-1},\Cr{c-den-rat},\Cr{c-dis-mono-bdd},T,D_0)$ is a positive constant depending only on $n,\kappa,\Cr{c-ene0},\Cr{c-mono-1},\Cr{c-den-rat},\Cr{c-dis-mono-bdd},T$ and $D_0$. 
For $y \in \Omega \setminus N_{\Cr{c-curva}/2}$, the similar argument using \eqref{monotonicity-2} and \eqref{dis-mono-conclusion2} in place of \eqref{monotonicity} and \eqref{dis-mono-conclusion1} gives the same estimate with the second term in the integral being zero. 
Taking $i \to 0$ and integrating the limit of \eqref{dis-vari6} over $(y, s) \in \overline{\Omega} \times (0, T)$, we obtain 
\begin{equation}\label{dis-vari7}
\int^T_0 ds \int_{\overline{\Omega}} d\mu^s(y) \iint_{\overline{\Omega} \times (0,T)} \dfrac{\eta(|x-y|)e^{-\frac{|x-y|^2}{4(s-t)}} + \eta(|x-\tilde{y}|)e^{-\frac{|x-\tilde{y}|^2}{4(s-t)}}}{2^n \pi^\frac{n-1}{2}(s-t)^\frac{n+1}{2}} \; d|\xi|(x,t) < \infty. 
\end{equation}
By the Fubini theorem, \eqref{dis-vari7} is turned into 
\[
\iint_{\overline{\Omega} \times (0,T)} d|\xi|(x,t) \int_t^T \dfrac{1}{2(s-t)} \overline{\mu}^s_{\sqrt{s-t},x} \; ds  < \infty. 
\]
Thus we have 
\begin{equation}\label{dis-vari8}
\int_t^T \dfrac{1}{2(s-t)} \overline{\mu}^s_{\sqrt{s-t},x} \; ds  < \infty
\end{equation}
for $|\xi|$ almost all $(x,t) \in \overline{\Omega} \times (0,T)$. 
We next prove that for $|\xi|$ almost all $(x,t)$, 
\begin{equation}\label{dis-vari-claim}
\lim_{s \downarrow t} \overline{\mu}^s_{\sqrt{s-t},x} = 0. 
\end{equation}
We fix a point $(x,t)$ satisfying \eqref{dis-vari8} and assume $x \in N_{\Cr{c-curva}/2}$ in the following. 
For $t < s$, we define $l := \log (s-t)$ and $h(s) := \overline{\mu}^s_{\sqrt{s-t},x}$. 
Then \eqref{dis-vari8} is translated into 
\begin{equation}\label{dis-vari9}
\int^{\log (T-s)}_{-\infty} h(t+e^l) \; dl < \infty. 
\end{equation}
Let $0 < \theta < 1$ be arbitrary for the moment. 
Due to \eqref{dis-vari9}, we may choose a decreasing sequence $\{l_j\}_{j=1}^\infty$ such that $l_j \to -\infty$, $l_j - l_{j+1} < \theta$ and $h(t+e^{l_j}) < \theta$ for all $j$. 
For any $-\infty < l < l_1$, we may choose $j \ge 2$ such that $l_j \le l < l_{j-1}$. 
By applying \eqref{monotonicity} and \eqref{dis-mono-conclusion1}, we obtain 
\begin{equation}\label{dis-vari10}
\begin{aligned}
h(t+e^l) =&\; \int_{\overline{\Omega}} \rho_{1,(x,t+2e^l)}(y,t+e^l) + \rho_{2,(x,t+2e^l)}(y,t+e^l) \; d\mu^{t+e^l}(y) \\
\le &\; e^{\Cr{c-mono-1}(2e^l - e^{l_j})^\frac{1}{4}} \int_{\overline{\Omega}}\rho_{1,(x,t+2e^l)}(y,t+e^{l_j}) + \rho_{2,(x,t+2e^l)}(y,t+e^{l_j}) \; d\mu^{t+e^{l_j}}(y) \\
=&\; e^{\Cr{c-mono-1}R_j^\frac{1}{2}} \overline{\mu}^{t+e^{l_j}}_{R_j, x}, 
\end{aligned}
\end{equation}
where $R_l = \sqrt{2e^l - e^{l_j}}$. 
Let $r_j = \sqrt{e^{l_j}}$. 
Since $l \ge l_j$, we have $R_l \ge r_j$. 
Furthermore, $l - l_j < l_{j-1} - l_j < \theta$ implies $R_l^2/r_j^2 < 2e^\theta - 1$ which may be made arbitrarily close to 1 by restricting $\theta$ to be small. 
For arbitrary $\delta > 0$, we restrict $\theta$ so that $R_l/r_j < 1+\Cr{c-lim-rat2}$, where $\Cr{c-lim-rat2}$ is given by Lemma \ref{lem-con-heat} corresponding to $\delta$. 
Then \eqref{dis-vari10} implies 
\begin{equation}\label{dis-vari11}
h(t+e^l) \le e^{\Cr{c-mono-1}R_l^\frac{1}{2}} \overline{\mu}^{t+e^{l_j}}_{R_l, x} \le e^{\Cr{c-mono-1}R_l^\frac{1}{2}} (\overline{\mu}^{t+e^{l_j}}_{r_j,x} + \delta) = e^{\Cr{c-mono-1}R_l^\frac{1}{2}} (h(t+e^{l_j}) + \delta) < e^{\Cr{c-mono-1}R_l^\frac{1}{2}} (\theta + \delta).
\end{equation}
In the case of $x \in \Omega \setminus N_{\Cr{c-curva}/2}$, we may prove \eqref{dis-vari11} by the similar argument. 
Since $\delta$ and $\theta$ are arbitrary and $\lim_{l \to -\infty} R_l = 1$ for any $\theta$, \eqref{dis-vari11} shows 
\[ \limsup_{l \to -\infty} h(t+e^l) = 0 \quad \mbox{for} \; |\xi| \; \mbox{almost all} \; (x,t) \in \overline{\Omega} \times (0,T) \]
as well as \eqref{dis-vari-claim}. 
This proves that $|\xi|((\overline{\Omega} \times (0,T) \setminus Z^-(T))=0$, since otherwise, we have $\limsup_{s \downarrow t} \overline{\mu}^s_{\sqrt{s-t},x} \ge \delta_0(T)$ on a set of positive measure with respect to $|\xi|$. 
Lemma \ref{lem-vani-spt} shows $\mu(Z^-(T)) = 0$, and since $|\xi| \le \mu$ by the definitions of these measures, we have $|\xi|(\overline{\Omega} \times (0,T)) = 0$. 
\end{proof}


\section{Proof of the main theorems}\label{sec:main}

In order to prove the main theorems, we have to analyze an associated varifold with the diffused surface energy as in \cite{MT}. 
Thus, for the solution $u_{\varepsilon_i}$ of \eqref{ac}, we associate a varifold as 
\[ V^t_{\varepsilon_i} := \int_{\Omega \cap \{|\nabla u_{\varepsilon_i}| \neq 0\}} \phi\left(x, I-\dfrac{\nabla u_{\varepsilon_i}}{|\nabla u_{\varepsilon_i}|} \otimes \dfrac{\nabla u_{\varepsilon_i}}{|\nabla u_{\varepsilon_i}|}\right) \; d\mu^t_{\varepsilon_i}(x) \quad \mbox{for} \quad \phi \in C(G_{n-1}(\mathbb{R}^n)). \]
Note that $\|V^t_{\varepsilon_i}\| = \mu^t_{\varepsilon_i}\lfloor_{\{|\nabla u_{\varepsilon_i}| \neq 0\}}$. 
We derive a formula for the first variation of $V^t_{\varepsilon_i}$ up to the boundary. 

\begin{rk}
We note that $I - \frac{\nabla u_{\varepsilon_i}(x,t)}{|\nabla u_{\varepsilon_i}(x,t)|} \otimes \frac{\nabla u_{\varepsilon_i}(x,t)}{|\nabla u_{\varepsilon_i}(x,t)|}$ is the orthogonal projection of $\mathbb{R}^n$ onto the tangent space of the level set $\{y \in \overline{\Omega} : u_{\varepsilon_i}(y,t) = u_{\varepsilon_i}(x,t)\}$ at $x$. 
Roughly speaking, since we may expect that $u_{\varepsilon_i}(\cdot,t)$ converges to $\pm1$ almost everywhere on $\overline{\Omega}$, all level sets $\{y \in \overline{\Omega} : u_{\varepsilon_i}(y,t) = s\}$ should converge to a hypersurface except for $s = \pm 1$. 
Thus, by the concentration of the diffused surface energy as in Remark \ref{rk:concent}, we may expect $I - \frac{\nabla u_{\varepsilon_i}(\cdot,t)}{|\nabla u_{\varepsilon_i}(\cdot,t)|} \otimes \frac{\nabla u_{\varepsilon_i}(\cdot,t)}{|\nabla u_{\varepsilon_i}(\cdot,t)|}$ converges to the tangent space of the limit surface on the support of the limit measure $\mu^t$. 
\end{rk}

\begin{lem}
For $g \in C^1(\mathbb{R}^n ; \mathbb{R}^n)$ and all $t \in [0,\infty)$, 
\begin{equation}\label{first-vari} 
\begin{aligned}
\delta V^t_{\varepsilon_i}(g) =&\; \int_\Omega \varepsilon_i \partial_t u_{\varepsilon_i} \langle g, \nabla u_{\varepsilon_i} \rangle \; dx + \int_{\Omega \cap \{|\nabla u_{\varepsilon_i}| \neq 0\}} \nabla g \cdot \left(\dfrac{\nabla u_{\varepsilon_i}}{|\nabla u_{\varepsilon_i}|} \otimes \dfrac{\nabla u_{\varepsilon_i}}{|\nabla u_{\varepsilon_i}|}\right) \; d\xi_{\varepsilon_i}^t \\
&\; + \int_{\partial \Omega} \langle g, \nu \rangle \left(\dfrac{\varepsilon_i |\nabla u_{\varepsilon_i}|^2}{2} + \dfrac{W(u_{\varepsilon_i})}{\varepsilon_i} \right) \; d\mathcal{H}^{n-1} - \int_{\Omega \cap \{|\nabla u_{\varepsilon_i}| = 0\}} \dfrac{W(u_{\varepsilon_i})}{\varepsilon_i} {\rm div} g \; dx. 
\end{aligned}
\end{equation}
\end{lem}

\begin{proof}
Omit the subindex $i$. 
By the definition of the first variation of varifolds, we have 
\begin{equation}\label{first-vari1}
\delta V^t_\varepsilon (g) = \int_{\Omega \cap \{|\nabla u_{\varepsilon}| \neq 0\}} \nabla g \cdot \left(I-\dfrac{\nabla u_{\varepsilon}}{|\nabla u_{\varepsilon}|} \otimes \dfrac{\nabla u_{\varepsilon}}{|\nabla u_{\varepsilon}|}\right) \; d\mu^t_{\varepsilon}. 
\end{equation}
Using the boundary condition \eqref{neumann} and integration by parts, we have 
\begin{equation}\label{first-vari2}
\int_\Omega \dfrac{|\nabla u_\varepsilon|^2}{2} {\rm div} g \; dx = \int_{\partial \Omega} \langle g, \nu \rangle \dfrac{|\nabla u_\varepsilon|^2}{2} \; d\mathcal{H}^{n-1} + \int_\Omega \nabla g \cdot (\nabla u_\varepsilon \otimes \nabla u_\varepsilon) + \langle g, \nabla u_\varepsilon \rangle \Delta u_\varepsilon \; dx. 
\end{equation}
Also by integration by parts, 
\begin{equation}\label{first-vari3}
\begin{aligned}
\int_{\Omega \cap \{|\nabla u_{\varepsilon}| \neq 0\}} W(u_\varepsilon) {\rm div} g \; dx =&\; - \int_{\Omega \cap \{|\nabla u_\varepsilon| = 0\}} W(u_\varepsilon) {\rm div} g \; dx - \int_\Omega \langle g, \nabla u_\varepsilon \rangle W^\prime(u_\varepsilon) \; dx \\
&\; + \int_{\partial \Omega} \langle g, \nu \rangle W(u_\varepsilon) \; d\mathcal{H}^{n-1}. 
\end{aligned}
\end{equation}
Substituting \eqref{first-vari2} and \eqref{first-vari3} into \eqref{first-vari1}, applying the equation \eqref{ac} and recalling the definition of $\xi_\varepsilon^t$, we obtain \eqref{first-vari}. 
\end{proof}

\begin{lem}
There exists a constant $\Cl[c]{c-boundary-ene}$ depending only on $n, \Cr{c-ene0}, \Cr{c-curva}, \kappa$ and $\Omega$ such that 
\begin{equation}\label{boundary-ene}
\int_{\partial \Omega} \dfrac{\varepsilon_i |\nabla u_{\varepsilon_i}|^2}{2} + \dfrac{W(u_{\varepsilon_i})}{\varepsilon_i} \; d\mathcal{H}^{n-1} \le \int_\Omega \varepsilon_i (\partial_t u_{\varepsilon_i})^2 \; dx + \Cr{c-boundary-ene}
\end{equation}
for all $t \in [0, \infty)$. 
\end{lem}

\begin{proof}
Let $\phi \in C^2(\overline{\Omega})$ be a positive function so that $\phi(x) = {\rm dist}(x, \partial \Omega)$ near $\partial \Omega$ and smoothly becomes a constant function on $\Omega \setminus N_{\Cr{c-curva}}$. 
We may construct such a function so that $\| \phi \|_{C^2(\overline{\Omega})}$ is bounded depending only on $n, \Cr{c-curva}, \kappa$ and $\Omega$. 
We also note that $\langle \nabla \phi, \nu \rangle = -1$ on $\partial \Omega$. 
By substituting $\nabla \phi$ into \eqref{first-vari}, applying Young's inequality and using $\xi_{\varepsilon_i}^t \le \mu_{\varepsilon_i}^t$, \eqref{energy-bdd} and the definition of the first variation $\delta V^t$, we obtain 
\begin{equation}\label{bdd-boundary-ene1}
\begin{aligned}
&\; \int_{\partial \Omega} \dfrac{\varepsilon_i |\nabla u_{\varepsilon_i}|^2}{2} + \dfrac{W(u_{\varepsilon_i})}{\varepsilon_i} \; d\mathcal{H}^{n-1} \\
=&\; \int_{\partial \Omega} \langle \nabla \phi, \nu\rangle \dfrac{\varepsilon_i |\nabla u_{\varepsilon_i}|^2}{2} + \dfrac{W(u_{\varepsilon_i})}{\varepsilon_i} \; d\mathcal{H}^{n-1} \\
=&\; \int_{\Omega \cap \{|\nabla u_{\varepsilon_i}| \neq 0\}} \nabla^2 \phi \cdot \left(I-\dfrac{\nabla u_{\varepsilon_i}}{|\nabla u_{\varepsilon_i}|} \otimes \dfrac{\nabla u_{\varepsilon_i}}{|\nabla u_{\varepsilon_i}|}\right) \; d\mu^t_{\varepsilon} - \int_\Omega \varepsilon_i \partial_t u_{\varepsilon_i} \langle \nabla \phi, \nabla u_{\varepsilon_i} \rangle \; dx \\
&\; - \int_{\Omega \cap \{|\nabla u_{\varepsilon_i}| \neq 0\}} \nabla^2 \phi \cdot \left(\dfrac{\nabla u_{\varepsilon_i}}{|\nabla u_{\varepsilon_i}|} \otimes \dfrac{\nabla u_{\varepsilon_i}}{|\nabla u_{\varepsilon_i}|}\right) \; d\xi_{\varepsilon_i}^t + \int_{\Omega \cap \{|\nabla u_{\varepsilon_i}| = 0\}} \dfrac{W(u_{\varepsilon_i})}{\varepsilon_i} \Delta \phi \; dx \\
\le&\; c(\|\phi\|_{C^2(\overline{\Omega})}) \mu_{\varepsilon_i}^t(\Omega) + c(\|\phi\|_{C^1(\overline{\Omega})}) \int_{\Omega} \dfrac{\varepsilon_i|\nabla u_{\varepsilon_i}|^2}{2} \; dx + \int_\Omega \varepsilon_i (\partial_t u_{\varepsilon_i})^2 \; dx \\
\le&\; c(\|\phi\|_{C^2(\overline{\Omega})}, \Cr{c-ene0}) + \int_\Omega \varepsilon_i (\partial_t u_{\varepsilon_i})^2 \; dx, 
\end{aligned}
\end{equation}
where $c(\|\phi\|_{C^2(\overline{\Omega})}, \Cr{c-ene0})$ is a positive constant depending only on $\|\phi\|_{C^2(\overline{\Omega})}$ and $\Cr{c-ene0}$. 
Here, we have used the boundedness of the operator norm of $I$ and $\frac{\nabla u_{\varepsilon_i}}{|\nabla u_{\varepsilon_i}|} \otimes \frac{\nabla u_{\varepsilon_i}}{|\nabla u_{\varepsilon_i}|}$. 
From the dependence of $\|\phi\|_{C^2(\overline{\Omega})}$, \eqref{bdd-boundary-ene1} implies the conclusion. 
\end{proof}

\begin{pro}\label{recti-limit}
Assume $V^t_{\varepsilon_{i_j}}$ converges to $\tilde{V}^t \in \mathbf{V}_{n-1}(\mathbb{R}^n)$ and 
\begin{equation}\label{as-convergence}
\liminf_{j \to \infty} \int_\Omega \varepsilon_{i_j} (\partial_t u_{\varepsilon_{i_j}})^2 \; dx < \infty, \quad \lim_{j \to \infty} \int_\Omega \left|\dfrac{\varepsilon_{i_j} |\nabla u_{\varepsilon_{i_j}}|^2}{2} - \dfrac{W(u_{\varepsilon_{i_j}})}{\varepsilon_{i_j}}\right| \; dx = 0
\end{equation}
for a subsequence $\varepsilon_{i_j}$ and a time $t \ge 0$. 
Then 
\begin{equation}\label{bdd-first-vari-byc}
|\delta \tilde{V}^t (g)| \le \left( 2\liminf_{j \to \infty} \int_\Omega \varepsilon_{i_j} (\partial_t u_{\varepsilon_{i_j}})^2 \; dx + \Cr{c-ene0} + \Cr{c-boundary-ene} \right) \sup_{\Omega} |g| 
\end{equation}
for $g \in C^1_c(\mathbb{R}^n; \mathbb{R}^n)$, $\mu^t$ is rectifiable and $\tilde{V}^t$ is the rectifiable varifold associated to $\mu^t$. 
\end{pro}

\begin{proof}
We note that $\|V^t_{\varepsilon_{i_j}}\|$ converges to $\mu^t$ from the second assumption on \eqref{as-convergence}, thus it is enough to prove the rectifiability of $\tilde{V}^t$ from the uniqueness of the rectifiable varifold. 
Let 
\[ c(t) := \liminf_{j \to \infty} \int_\Omega \varepsilon_{i_j} (\partial_t u_{\varepsilon_{i_j}})^2 \; dx. \]
Since $\lim_{j \to \infty} \delta V_{\varepsilon_{i_j}}^t = \delta \tilde{V}^t$, it is easy to see by \eqref{energy-bdd}, \eqref{first-vari}, \eqref{boundary-ene} and Young's inequality
\[
|\delta \tilde{V}^t(g)| \le (2c(t) + \Cr{c-ene0} + \Cr{c-boundary-ene}) \max_{\overline{\Omega}} |g| 
\]
for $g \in C^1_c(\mathbb{R}^n; \mathbb{R}^n)$. 
This shows that the total variation $\|\delta \tilde{V}^t\|$ is a Radon measure. 
Thus, Allard's rectifiability theorem \cite[5.5.\ (1)]{A1} shows $\tilde{V}^t\lfloor_{\{x : \limsup_{r \downarrow 0} \|\tilde{V}^t\|(B_r(x))/(\omega_{n-1}r^{n-1}) > 0\}\times \mathbf{G}(n,n-1)}$ is rectifiable. 
On the other hand, a standard measure theoretic argument (see for example \cite[3.2(2)]{S}) and Corillary \ref{cor-bdd-spt} show 
\[
\mu^t \left(\left\{x \in {\rm spt} \mu^t : \limsup_{r \downarrow 0} \dfrac{\|\tilde{V}^t\|(B_r(x))}{\omega_{n-1}r^{n-1}} < l \right\}\right) \le 2^{n-1} l \mathcal{H}^{n-1}({\rm spt}\mu^t) \le 2^{n-1} l \Cr{c-bdd-spt}  
\]
for any $l > 0$, thus we obtain 
\[ \mu^t \left(\left\{x \in {\rm spt} \mu^t : \lim_{r \downarrow 0} \dfrac{\|\tilde{V}^t\|(B_r(x))}{\omega_{n-1}r^{n-1}} = 0 \right\}\right) = 0. \]
This equality and $\|\tilde{V}^t\| = \mu^t$ imply 
\[ \tilde{V}^t\lfloor_{\{x : \limsup_{r \downarrow 0} \|\tilde{V}^t\|(B_r(x))/(\omega_{n-1}r^{n-1}) > 0\}\times \mathbf{G}(n,n-1)} = \tilde{V}^t \]
and hence we have the conclusion. 
\end{proof}

\begin{proof}[Proof of Theorem \ref{main1} and Theorem \ref{main2}]
From \eqref{energy-bdd} and Proposition \ref{thm-vani-dis}, we may see 
\begin{equation}\label{t-choose}
\liminf_{i \to \infty} \int_\Omega \varepsilon_i (\partial_t u_{\varepsilon_i})^2 \; dx < \infty \quad \mbox{and} \quad \lim_{i \to \infty} \int_\Omega \left|\dfrac{\varepsilon_i |\nabla u_{\varepsilon_i}|^2}{2} - \dfrac{W(u_{\varepsilon_i})}{\varepsilon_i}\right| \; dx = 0
\end{equation}
for a.e.\ $t \ge 0$. 
We fix a time $t$ satisfying \eqref{t-choose}. 
By the boundedness of the diffused surface energy \eqref{energy-bdd}, the definition of $V_{\varepsilon_i}^t$ and \eqref{t-choose}, there exist a subsequence $\varepsilon_{i_j}$ such that 
\begin{equation}\label{t-choose2} 
\lim_{j \to \infty} \int_\Omega \varepsilon_{i_j} (\partial_t u_{\varepsilon_{i_j}})^2 \; dx = \liminf_{i \to \infty} \int_\Omega \varepsilon_i (\partial_t u_{\varepsilon_i})^2 \; dx 
\end{equation}
and $V_{\varepsilon_{i_j}}^t$ converges to a varifold $\tilde{V}^t$. 
Then we can apply Proposition \ref{recti-limit} and hence we have the conclusion except for the boundedness of $\int^T_0 \|\delta V^t\|(\overline{\Omega}) \; dt$. 
Since the right hand side of \eqref{bdd-first-vari-byc} is locally uniformly integrable, Fatou's lemma shows this boundedness. 
\end{proof}

\begin{proof}[Proof of Theorem \ref{main3}]
We fix a time $t$ satisfying \eqref{t-choose} and take a subsequence $\varepsilon_{i_j}$ such that \eqref{t-choose2} holds and $V^t_{\varepsilon_{i_j}}$ converges to $V^t$. 
By \eqref{first-vari}, we have 
\[ |\delta V^t(g)| \le \left(\int_\Omega |g|^2 \; d\|V^t\|\right)^\frac{1}{2} \liminf_{i \to \infty} \left(\int_\Omega \varepsilon_{i} (\partial_t u_{\varepsilon_{i}})^2 \; dx \right)^\frac{1}{2} \]
for $g \in C^1_c(\Omega; \mathbb{R}^n)$. 
This shows $\| \delta V^t \lfloor_\Omega \| \ll \|V^t\lfloor_\Omega\|$ and $\delta V^t\lfloor_\Omega = - h^t \|V^t\lfloor_\Omega\|$ for $h^t \in L^2(\|V^t\lfloor_\Omega\|)$. 
Now, for given arbitrary $\delta > 0$, let $\nu^\delta \in C^1(\overline{\Omega}; \mathbb{R}^n)$ be such that $\nu^\delta \lfloor_{\partial \Omega} = \nu$, $|\nu^\delta| \le 1$ and ${\rm spt}\; \nu^\delta \subset N_\delta$. 
For $g \in C^1(\overline{\Omega}; \mathbb{R}^n)$, define $\tilde{g} := g - \langle g, \nu^\delta \rangle \nu^\delta$. 
Then $\langle \tilde{g}, \nu \rangle = 0$ on $\partial \Omega$ thus $\delta V^t\lfloor_{\partial \Omega}^\top (g) = \delta V^t\lfloor_{\partial \Omega} (\tilde{g})$. 
By \eqref{first-vari}, \eqref{t-choose} and $|\tilde{g}| \le |g|$, we have 
\begin{equation}\label{proof-main3}
\begin{aligned}
\delta V^t\lfloor_{\partial \Omega}^\top(g) + \delta V^t\lfloor_\Omega(g) =&\; \delta V^t(\tilde{g}) + \delta V^t\lfloor_\Omega (g-\tilde{g}) \\
\le&\; \left(\int_{\overline{\Omega}} |g|^2 \; d\|V^t\|\right)^\frac{1}{2} \liminf_{i \to \infty} \left(\int_\Omega \varepsilon_{i} (\partial_t u_{\varepsilon_{i}})^2 \; dx\right)^\frac{1}{2} + \delta V^t\lfloor_\Omega (g-\tilde{g}). 
\end{aligned}
\end{equation}
Since ${\rm spt} \; \nu^\delta \subset N_\delta$, we have 
\begin{equation}\label{proof-main4}
|\delta V^t\lfloor_\Omega (g-\tilde{g})| \le \sup|g| \int_{\Omega \cap N_\delta} |h^t| \; d\|V^t\| \to 0 
\end{equation}
as $\delta \to 0$. 
Combining \eqref{energy-bdd}, \eqref{proof-main3} and \eqref{proof-main4}, we conclude (A1) by letting 
\[
h_b^t := 
\begin{cases}
- \frac{\delta V^t\lfloor_{\partial \Omega}^\top}{\|V^t\|} & \mbox{on} \quad \partial \Omega \\
- \frac{\delta V^t\lfloor_{\Omega}}{\|V^t\|} & \mbox{on} \quad \Omega. 
\end{cases}
\]
Furthermore, we may carry out an approximation argument (see \cite[Proposition 8.1]{TT} for detail) to obtain 
\begin{equation}\label{proof-main5}
\int_{\overline{\Omega}} \phi |h_b^t|^2 \; d\|V^t\| \le \liminf_{i \to \infty} \int_\Omega \varepsilon_i (\partial_t u_{\varepsilon_i})^2 \phi \; dx
\end{equation}
for general $\phi \in C_c(\mathbb{R}^n; \mathbb{R}^+)$. 
Integrate \eqref{proof-main5} with $\phi\lfloor_{\overline{\Omega}}\equiv 1$ over $t \in (0,\infty)$ and apply Fatou's Lemma and \eqref{energy-bdd} to conclude (A2). 

Next, we prove (A3). 
It is enough to prove \eqref{brakke-neumann} for $\phi \in C^2(\overline{\Omega} \times [0,\infty); \mathbb{R}^+)$ with $\langle \nabla \phi(\cdot, t), \nu \rangle = 0$ on $\partial \Omega$. 
From \eqref{ac} and \eqref{neumann}, we have 
\begin{equation}\label{proof-main7}
\int_\Omega \phi \; d\mu^t_{\varepsilon_i} \Big|^{t_2}_{t=t_1} = \int^{t_2}_{t_1} \left(\int_\Omega -\varepsilon_i(\partial_t u_{\varepsilon_i})^2 \phi - \varepsilon_i \partial_t u_{\varepsilon_i} \langle \nabla \phi, \nabla u_{\varepsilon_i} \rangle \; dx + \int_\Omega \partial_t \phi \; d\mu^t_{\varepsilon_i} \right) dt 
\end{equation}
for all $0 \le t_1 < t_2 < \infty$. 
Since $\mu^t_{\varepsilon_i}$ converges to $\|V^t\|$ for all $t \ge 0$, the left hand side of \eqref{proof-main7} converges to that of \eqref{brakke-neumann}, and so is the last term of the right hand side. 
Thus we may finish the proof if we prove 
\begin{equation}\label{remain-main}
\lim_{i \to \infty} \int^{t_2}_{t_1} \int_\Omega \varepsilon_i(\partial_t u_{\varepsilon_i})^2 \phi + \varepsilon_i \partial_t u_{\varepsilon_i} \langle \nabla \phi, \nabla u_{\varepsilon_i} \rangle \; dx dt \ge \int^{t_2}_{t_1} \int_{\overline{\Omega}} \phi|h_b^t|^2 - \langle \nabla \phi, h_b^t \rangle \; d\|V^t\| dt. 
\end{equation}
Here, we note that \eqref{proof-main7} also implies by H\"older's inequality and \eqref{energy-bdd}
\begin{equation}\label{add}
\begin{aligned}
&\; \lim_{i \to \infty} \int^{t_2}_{t_1} \int_\Omega \varepsilon_i(\partial_t u_{\varepsilon_i})^2 \phi + \varepsilon_i \partial_t u_{\varepsilon_i} \langle \nabla \phi, \nabla u_{\varepsilon_i} \rangle \; dx dt \\
=&\; \mu^t(\phi)\Big|_{t=t_1}^{t_2} - \iint_{\overline{\Omega} \times [t_1, t_2]} \partial_t \phi \; d\mu \le c(\Cr{c-ene0}, t_1, t_2, \|\phi\|_{C^1(\Omega \times [t_1,t_2])}), 
\end{aligned}
\end{equation}
where $c(\Cr{c-ene0}, t_1, t_2, \|\phi\|_{C^1(\Omega \times [t_1,t_2])})$ is a constant depending only on $\Cr{c-ene0}, t_1, t_2$ and $\|\phi\|_{C^1(\Omega \times [t_1,t_2])}$. 
From $\frac{|\nabla \phi(\cdot,t)|^2}{2\phi(\cdot,t)} \le \|\phi(\cdot,t)\|_{C^2(\Omega)}$ for any $t \ge 0$ and \eqref{energy-bdd}, we obtain 
\[ 
\begin{aligned}
&\; \int_\Omega \varepsilon_i(\partial_t u_{\varepsilon_i})^2 \phi + \varepsilon_i \partial_t u_{\varepsilon_i} \langle \nabla \phi, \nabla u_{\varepsilon_i} \rangle \; dx \\
=&\; \int_\Omega \varepsilon_i \phi \left(\partial_t u_{\varepsilon_i} + \dfrac{\langle \nabla \phi, \nabla u_{\varepsilon_i}\rangle}{2\phi}\right)^2 - \dfrac{\varepsilon_i \langle \nabla \phi, \nabla u_{\varepsilon_i} \rangle^2}{4\phi} \; dx \ge - \Cr{c-ene0} \|\phi(\cdot,t)\|_{C^2(\Omega)} 
\end{aligned}
\]
for any $t \ge 0$. 
Thus by Fatou's lemma,  
\begin{equation}\label{proof-main8}
\begin{aligned}
&\; \lim_{i \to \infty} \int^{t_2}_{t_1} \int_\Omega \varepsilon_i(\partial_t u_{\varepsilon_i})^2 \phi + \varepsilon_i \partial_t u_{\varepsilon_i} \langle \nabla \phi, \nabla u_{\varepsilon_i} \rangle \; dx dt \\
&\; \ge \int^{t_2}_{t_1} \liminf_{i \to \infty} \int_\Omega \varepsilon_i(\partial_t u_{\varepsilon_i})^2 \phi + \varepsilon_i \partial_t u_{\varepsilon_i} \langle \nabla \phi, \nabla u_{\varepsilon_i} \rangle \; dx dt. 
\end{aligned}
\end{equation}
By \eqref{add}, \eqref{proof-main8} and Proposition \ref{thm-vani-dis}, for a.e.\ $t \in (t_1,t_2)$, we can choose a subsequence $\varepsilon_{i_j}$ (depending on $t$) such that 
\begin{equation}\label{proof-main9}
\begin{aligned}
\lim_{j \to \infty} \int_\Omega \varepsilon_{i_j}(\partial_t u_{\varepsilon_{i_j}})^2 \phi + \varepsilon_{i_j} \partial_t u_{\varepsilon_{i_j}} \langle \nabla \phi, \nabla u_{\varepsilon_{i_j}} \rangle \; dx =&\; \liminf_{i \to \infty} \int_\Omega \varepsilon_i(\partial_t u_{\varepsilon_i})^2 \phi + \varepsilon_i \partial_t u_{\varepsilon_i} \langle \nabla \phi, \nabla u_{\varepsilon_i} \rangle \; dx < \infty \\
\lim_{j \to \infty} \int_\Omega \left|\dfrac{\varepsilon_{i_j} |\nabla u_{\varepsilon_{i_j}}|^2}{2} - \dfrac{W(u_{\varepsilon_{i_j}})}{\varepsilon_{i_j}}\right| \; dx =&\; 0
\end{aligned}
\end{equation}
and $V^t_{\varepsilon_{i_j}}$ converges to some varifold $\tilde{V}^t \in \mathbf{V}_{n-1}(\mathbb{R}^n)$. 
We fix such $t$ and subsequence $\varepsilon_{i_j}$. 
By \eqref{energy-bdd} and Young's inequality, we obtain 
\[ \int_\Omega \varepsilon_{i_j}(\partial_t u_{\varepsilon_{i_j}})^2 \phi + \varepsilon_{i_j} \partial_t u_{\varepsilon_{i_j}} \langle \nabla \phi, \nabla u_{\varepsilon_{i_j}} \rangle \; dx \ge \dfrac{1}{2} \int_\Omega \varepsilon_{i_j}(\partial_t u_{\varepsilon_{i_j}})^2 \phi \; dx - c(\Cr{c-ene0}, \|\phi\|_{C^2}), \]
hence \eqref{proof-main9} implies 
\[ \limsup_{j \to \infty} \int_\Omega \varepsilon_{i_j}(\partial_t u_{\varepsilon_{i_j}})^2 \phi \; dx < \infty. \]
Arguing as the proof of Proposition \ref{recti-limit}, we may prove $\tilde{V}^t\lfloor_{\{x : \phi(x,t) > 0\}\times \mathbf{G}(n,n-1)}$ is rectifiable and $\tilde{V}^t\lfloor_{\{x : \phi(x,t) > 0\}\times \mathbf{G}(n,n-1)} = V^t\lfloor_{\{x : \phi(x,t) > 0\}\times \mathbf{G}(n,n-1)}$. 
For $\tilde{\phi} \in C^2_c(\{\phi > 0\}; \mathbb{R}^+)$ with $\langle \nabla \tilde{\phi}, \nu \rangle = 0$ on $\partial \Omega$ and $\tilde{\phi} \le \phi$, we obtain by the definition of $h^t_b$ and \eqref{first-vari}
\begin{equation}\label{proof-main10}
-\int_{\overline{\Omega}} \langle \nabla\tilde{\phi}, h_b^t \rangle \; d\|V^t\| = \delta V(\nabla \tilde{\phi}) = \lim_{j \to \infty} \int_\Omega \varepsilon_{i_j} \partial_t u_{\varepsilon_{i_j}} \langle \nabla \tilde{\phi}, \nabla u_{\varepsilon_{i_j}} \rangle \; dx. 
\end{equation}
From $h^t_b \in L^2(\|V^t\|)$ and 
\begin{align*} 
\int_\Omega \varepsilon_{i_j} \partial_t u_{\varepsilon_{i_j}} \langle \nabla \tilde{\phi} - \nabla \phi, \nabla u_{\varepsilon_{i_j}} \rangle \; dx \le&\; \left(\int_\Omega \varepsilon_{i_j} (\partial_t u_{\varepsilon_{i_j}})^2 \phi \; dx\right)^{\frac{1}{2}} \left(\int_\Omega \dfrac{|\nabla \tilde{\phi} - \nabla \phi|^2}{\phi - \tilde{\phi}} \varepsilon_{i_j} |\nabla u_{\varepsilon_{i_j}}|^2 \; dx \right)^\frac{1}{2} \\
\le &\; 2\Cr{c-ene0}^\frac{1}{2} \left(\int_\Omega \varepsilon_{i_j} (\partial_t u_{\varepsilon_{i_j}})^2 \phi \; dx\right)^{\frac{1}{2}} \|\tilde{\phi} - \phi \|_{C^2}, 
\end{align*}
we may obtain 
\begin{equation}\label{proof-main11}
-\int_{\overline{\Omega}} \langle \nabla \phi, h_b^t \rangle \; d\|V^t\| = \lim_{j \to \infty} \int_\Omega \varepsilon_{i_j} \partial_t u_{\varepsilon_{i_j}} \langle \nabla \phi, \nabla u_{\varepsilon_{i_j}} \rangle \; dx 
\end{equation}
by letting $\tilde{\phi} \to \phi$ in $C^2$ for \eqref{proof-main10}. 
Hence we conclude \eqref{remain-main} from \eqref{proof-main5}, \eqref{proof-main8}, \eqref{proof-main9} and \eqref{proof-main11}. 
\end{proof}


\end{document}